\documentclass[reqno]{amsart}
\usepackage[utf8]{inputenc}
\usepackage{amssymb, amsmath, amsthm, color, enumerate, mathabx, mathrsfs}
\usepackage{bbm}
\usepackage{pdfpages}
\usepackage{url}
\usepackage{tcolorbox}

\usepackage{tcolorbox}

\numberwithin{equation}{section}

\newcommand{\R}{\mathbb{R}}

\newcommand{\Z}{\mathbb{Z}}
\newcommand{\N}{\mathbb{N}}

\newcommand{\B}{\mathbb{B}}
\newcommand{\D}{\mathbb{D}}

\newcommand{\cB}{\mathcal{B}}
\newcommand{\cC}{\mathcal{C}}
\newcommand{\cE}{\mathcal{E}}
\newcommand{\cG}{\mathcal{G}}

\newcommand{\cH}{\mathcal{H}}
\newcommand{\cI}{\mathcal{I}}

\newcommand{\cL}{\mathcal{L}}
\newcommand{\cN}{\mathcal{N}}
\newcommand{\cS}{\mathcal{S}}

\newcommand{\cZ}{\mathcal{Z}}

\newcommand{\ud}{\, \mathrm{d}}

\newcommand{\bx}{\mathbf{x}}

\newcommand{\bz}{\mathbf{z}}

\newcommand{\bxi}{\boldsymbol{\xi}}
\newcommand{\balpha}{\boldsymbol{\alpha}}
\newcommand{\bbeta}{\boldsymbol{\beta}}
\newcommand{\bgamma}{\boldsymbol{\gamma}}

\newcommand{\fa}{\mathfrak{a}}
\newcommand{\fe}{\mathfrak{e}}
\newcommand{\fh}{\mathfrak{h}}
\newcommand{\fn}{\mathfrak{h}}
\newcommand{\fA}{\mathfrak{A}}
\newcommand{\fB}{\mathfrak{B}}
\newcommand{\fE}{\mathfrak{E}}
\newcommand{\fG}{\mathfrak{G}}

\newcommand{\fm}{\mathfrak{m}}
\newcommand{\fN}{\mathfrak{N}}

\def\inn#1#2{\langle#1,#2\rangle}

\theoremstyle{plain}

\newtheorem{theorem}{Theorem}[section]
\newtheorem{lemma}[theorem]{Lemma}
\newtheorem{proposition}[theorem]{Proposition}

\theoremstyle{definition}
\newtheorem{definition}[theorem]{Definition}

\newcommand{\sgn}{\mathrm{sgn}\,}

\newcommand{\main}{\mathrm{main}}
\newcommand{\error}{\mathrm{error}}

\newcommand{\floor}[1]{\lfloor #1 \rfloor }

\newcommand{\supp}{\mathrm{supp}\,}

\newcommand{\diag}{\mathrm{diag}}

\def\anorm#1{ \Big|  #1 \Big| }
\def\Norm#1{ \Big\|  #1 \Big\| }
\def\norm#1{\big\|  #1 \big\| }
\def\inn#1#2{\langle#1,#2\rangle}

\begin{document}

\title[On a planar Pierce--Yung operator]{On a planar Pierce--Yung operator}

\date{\today}


\begin{abstract} We show that the operator
\begin{equation*}
    \cC f(x,y) := \sup_{v\in \R} \Big|\mathrm{p.v.} \int_\R f(x-t, y-t^2) e^{i v t^3} \frac{\ud t}{t} \Big|
\end{equation*}
is bounded on $L^p(\R^2)$ for every $1 < p < \infty$. This gives an affirmative answer to a question of Pierce and Yung. 
\end{abstract}

\author[D. Beltran]{David Beltran}
\address{David Beltran: Departament d’An\`alisi Matem\`atica, Universitat de Val\`encia, Dr. Moliner 50, 46100 Burjassot, Spain}
\email{david.beltran@uv.es}
%\thanks{D.B. is supported by the AEI grants RYC2020-029151-I and PID2022-140977NA-I00}

\author[S. Guo]{Shaoming Guo}
\address{Shaoming Guo: Chern Institute of Mathematics, Nankai University, Tianjin, China\\
and \\
Department of Mathematics, University of Wisconsin-Madison, Madison, WI-53706, USA}
\email{shaomingguo2018@gmail.com}

\author[J. Hickman]{Jonathan Hickman}
\address{Jonathan Hickman: School of Mathematics, James Clerk Maxwell Building, The King's Buildings, Peter Guthrie Tait Road, Edinburgh, EH9 3FD, UK.}
\email{jonathan.hickman@ed.ac.uk}

\maketitle

%%%%%%%%%%%%%%%%%%%%%%%%%%%%%%%%%%%%%%%%%%%%%%%%%%%%%%%%%%%%%%%%%%%%%%%%%%%%%%%%%%%%%%%%%%%%%%%%

%    Introduction

%%%%%%%%%%%%%%%%%%%%%%%%%%%%%%%%%%%%%%%%%%%%%%%%%%%%%%%%%%%%%%%%%%%%%%%%%%%%%%%%%%%%%%%%%%%%%%%%

\section{Introduction}

\subsection{Statement of the main theorem}\label{240613subsection1_1}
Let $d\ge 2$ be an integer and $K \colon \R^{d-1}\to \R$ be a Calder\'on--Zygmund kernel (see \cite[(1.9)]{PY2019} for the precise requirements). Pierce and Yung \cite{PY2019} considered the following operator\footnote{Indeed, they considered a much wider class of operators, with $v(t_1^3 + \cdots + t_{n-1}^3)$ replaced by more general polynomials. We shall come back to this point later.}
\begin{equation}\label{240613e1_0}
\mathcal{C}_d f(x, y):= \sup_{v\in \R} \anorm{
\int_{\R^{d-1}} f(x-t, y-|t|^2)
 e^{iv(t_1^3 + \cdots + t_{n-1}^3)} 
 K(t)\ud t
 },
\end{equation}
where $x\in \R^{d-1}$, $y\in \R$, $t \in \R^{d-1}$. Moreover, in the same paper, whenever $d\ge 3$ they proved that\footnote{Given a list of objects $L$ and real numbers $A$, $B \geq 0$, here and throughout we write $A \lesssim_L B$ or $B \gtrsim_L A$ to indicate $A \leq C_L B$ for some constant $C_L$ which depends only items in the list $L$. We write $A \sim_L B$ to indicate $A \lesssim_L B$ and $B \lesssim_L A$.}  
\begin{equation}\label{240613e1_1}
\|\mathcal{C}_d f\|_{L^p(\R^d)}\lesssim_{p, d, K} \|f\|_{L^p(\R^d)}
\end{equation}
holds for all $1 < p < \infty$. Roughly speaking, their proof relies crucially on a $TT^*$ argument and sophisticated oscillatory integral estimates, both of which require $d\ge 3$.  In particular, they asked whether \eqref{240613e1_1} holds for $d=2$. In this paper, we give an affirmative answer to this question. 

\begin{theorem}\label{thm:main}
For $x, y\in \R$, consider the operator
\begin{equation}\label{240613e1_2}
    \cC f(x,y) := \sup_{v \in \R} \Big|\mathrm{p.v.} \int_\R f(x-t, y-t^2) e^{i v t^3} \frac{\ud t}{t} \Big|,
\end{equation}
initially defined for Schwartz functions $f \in \cS(\R^2)$. 
We have 
\begin{equation*}
\|\cC f\|_{L^p(\R^2)} \lesssim_{p} \|f\|_{L^p(\R^2)}
\end{equation*}
for all $1 < p < \infty$. 
\end{theorem}

We remark that our methods apply to more general Calder\'on--Zygmund kernels than the Hilbert kernel, but we focus on this case for simplicity.

%The operator $\cC$ incorporates Radon-type behaviour into the classical polynomial Carleson operator of Stein--Wainger \cite{SW2001}. The study of such objects was initiated by Pierce and Yung in \cite{PY2019}; we refer to the introduction of \cite{PY2019} for their motivation and the relationship between $\cC$ and other classical operators in harmonic analysis. 

%The choice of cubic oscillation $e^{ivt^3}$ is natural since linear and quadratic modulations result in modulation-invariance properties which necessitate the use of intricate time-frequency analysis techniques. This is analogous to the work of Stein--Wainger \cite{SW2001} on the classical polynomial Carleson operator. By avoiding linear terms in their polynomial modulations, Stein--Wainger~\cite{SW2001} were able to establish a general theory using simple $L^2$ methods and oscillatory integral estimates. By contrast, once linear terms are admitted time-frequency analysis plays a fundamental role and the problem becomes significantly more difficult: see the recent breakthrough work of Lie \cite{Lie2020}. For the class of Carleson operators of Radon type, we remark that partial results have been obtained for linear and quadratic modulations in \cite{Roos2019, Ramos2021, becker2023degree, MR4687376} using time-frequency techniques. However, this is not the direction of the present work, which is firmly rooted within the Stein--Wainger~\cite{SW2001} paradigm. 

\subsection{Earlier attempts.} Several earlier attempts were made to solve the Pierce--Yung problem. Guo, Pierce, Roos and Yung \cite{GPRY2017} studied a weaker maximal operator than \eqref{240613e1_2}. More precisely, they proved that 
\begin{equation}\label{240613e1_5}
    \Norm{
    \sup_{v\in \R} \Norm{
    \mathrm{p.v.}\int_{\R}
    f(x-t, y-t^2) e^{iv t^3}\frac{\ud t}{t}
    }_{
    L^p_y(\R)
    }
    }_{L^p_x(\R)} \lesssim_{p} \|f\|_{L^p(\R^2)}
\end{equation}
holds for every $p\in (1, \infty)$. Here $L^p_y$ (respectively, $L^p_x$) means we take the $L^p$ norm in the $y$ (respectively, $x$) variable. To better compare \eqref{240613e1_5} with \eqref{240613e1_2}, we linearise the relevant maximal operators. 

Our result in Theorem~\ref{thm:main} says that 
\begin{equation*}
    \Norm{
    \mathrm{p.v.}\int_{\R}
    f(x-t, y-t^2)e^{iv(x, y)t^3}\frac{\ud t}{t}
    }_{L^p(\R^2)}\lesssim_p \|f\|_{L^p(\R^2)}
\end{equation*}
holds for every $p\in (1, \infty)$ and every measurable function $v(x, y)$. However,  \eqref{240613e1_5} is equivalent to saying that 
\begin{equation}\label{240613e1_7xx}
    \Norm{
    \mathrm{p.v.}\int_{\R}
    f(x-t, y-t^2)e^{iv(x)t^3}\frac{\ud t}{t}
    }_{L^p(\R^2)}\lesssim_p \|f\|_{L^p(\R^2)}
\end{equation}
holds for every $p\in (1, \infty)$ and every measurable function $v(x)$. We emphasise that the measurable function $v(x)$ in \eqref{240613e1_7xx} does \textit{not} depend on the $y$ variable. This allows the authors in \cite{GPRY2017} to take a Fourier transform in the $y$ variable and turn the analysis into a much simpler problem, which is essentially one-dimensional. 

More recently, Roos \cite{Roos2019}, Ramos \cite{Ramos2021} and Becker \cite{MR4687376, becker2023degree} tried to bound planar Pierce--Yung operators\footnote{They all consider a more general form of the Pierce--Yung operator defined above.} by viewing them as  modulation-invariant operators, in the spirit of Carleson's operator (see Carleson \cite{MR0199631}) and the bilinear Hilbert transform (see Lacey and Thiele \cite{MR1491450, MR1689336}). However, the Fourier multiplier for the Hilbert transform along the parabola is very rough, in the sense that the Fourier multiplier for a Calder\'on--Zygmund kernel has singularities only at the origin and at infinity, while that  for the Hilbert transform along the parabola is singular along an entire line in the frequency space $\R^2$. Consequently, the aforementioned results \cite{Roos2019, Ramos2021, MR4687376, becker2023degree} all have certain logarithmic losses.

\subsection{Why dimension matters.} Here we describe two key differences between the $d\ge 3$ and $d=2$ cases.

As mentioned in \S\ref{240613subsection1_1}, to prove their result \eqref{240613e1_1} for $d\ge 3$, Pierce and Yung \cite{PY2019} used a $TT^*$ method. In particular, they wrote the kernel of the resulting $TT^*$ operator as an oscillatory integral and then carried out sophisticated oscillatory integral estimates. 

To be more precise, let $\varphi_0 \colon \R\to \R$ be a  smooth function supported on the interval $[1, 2]$. 
Pierce and Yung \cite{PY2019} first linearized the maximal operator \eqref{240613e1_0} and aimed to show that there exists $\varepsilon>0$, depending only on the dimension $d$ and the kernel $K$, such that 
\begin{equation}\label{240613e1_7}
\norm{
\mathcal{C}_{d, j} f
}_{L^2(\R^d)} \lesssim_{d, K} 2^{-\varepsilon j} \|f\|_{L^2(\R^d)}
\end{equation}
holds for all $j\in \N$ and $d\ge 3$,
where
\begin{equation*}
\mathcal{C}_{d, j} f(x, y):= \int_{\R^{d-1}} f(x-t, y-|t|^2)
 e^{iv(x, y) |t|^3} 
 \varphi_0(2^{-j}|t|)
 K(t)\ud t
\end{equation*}
 for $v(x, y)$ a measurable function taking values in the interval $[1, 2]$. Once the $L^2$ estimate \eqref{240613e1_7} is proved, one may interpolate with trivial $L^1$ and $L^{\infty}$ estimates (without any decaying factor $2^{-\varepsilon j}$) to obtain the full range $p\in (1, \infty)$.

 The decaying factor $2^{-\varepsilon j}$ is natural to expect: by Minkowski's inequality, we trivially have 
 \begin{equation*}
 \Norm{
 \int_{\R^{d-1}} f(x-t, y-|t|^2)
 \varphi_0(2^{-j}|t|)
 K(t)\ud t
 }_{L^2(\R^d)}\lesssim_{d, K} \|f\|_{L^2(\R^d)},
 \end{equation*}
 uniformly in $j\in \N$. Adding the extra oscillatory term $e^{i v(x, y) |t|^3}$, it is very natural to expect a decay estimate of the form \eqref{240613e1_7}. However, proving the decay estimate \eqref{240613e1_7} is highly non-trivial as $v$ does not have any regularity and can be an arbitrary measurable function. Moreover, our estimates must also be uniform over all such choices of $v$.
 
 The strategy of Pierce and Yung \cite{PY2019} is to regularise the operator $\mathcal{C}_{d, j}$ by considering $\mathcal{C}_{d, j} \mathcal{C}_{d, j}^*$, the composition of the adjoint operator $\mathcal{C}_{d, j}^*$ with $\mathcal{C}_{d, j}$. The operator $\mathcal{C}_{d, j} \mathcal{C}_{d, j}^*$ may be better behaved than the operator $\mathcal{C}_{d, j}$ as it involves an integral in $\R^{d-1}\times \R^{d-1}$, rather than $\R^{d-1}$ itself. As the function $f$ depends on $d$ variables, when expressing $\mathcal{C}_{d, j} \mathcal{C}_{d, j}^*$ as a convolution operator, the kernel can be written as an integral in 
\begin{equation}\label{240614e1_9uuu}
(d-1)+ (d-1)-d= d-2
\end{equation}
variables. From \eqref{240614e1_9uuu}, one immediately sees how the constraint $d\ge 3$ in Pierce and Yung \cite{PY2019} appears: when $d=2$ the convolution kernel is not given by an integral formula, and is just a very rough measure. 

There is a second place where $d\ge 3$ plays a key role in \cite{PY2019}. This can be explained by the (vague) principle that integrating along the paraboloid in $\R^d$ for $d\ge 3$ is less `singular' than integrating along the parabola in $\R^2$. Indeed, the larger $d$ gets, the less `singular' it is to integrate along the paraboloid in $\R^d$. 

Following the above principle, Stein \cite{MR0420116} proved that the spherical maximal operator in $\R^d$ is bounded whenever $d\ge 3$. Some time later, the case $d=2$ was resolved in Bourgain's breakthrough work \cite{MR0874045}. Sogge \cite{MR1098614} and Mockenhaupt, Seeger and Sogge \cite{MR1168960, MR1173929} further developed and significantly expanded upon Bourgain's ideas, leading to the discovery of the local smoothing phenomenon for linear wave equations.

The difference between Stein's approach to the spherical maximal function \cite{MR0420116} and  Mockenhaupt, Seeger and Sogge's approach to the circular maximal function \cite{MR1173929} is very similar to the difference between the approach of Pierce and Yung \cite{PY2019} and the approach in the current paper. Neither Stein \cite{MR0420116} nor Pierce and Yung \cite{PY2019} required any local smoothing estimates in their work. By contrast, local smoothing estimates are crucial in Mockenhaupt, Seeger and Sogge's work \cite{MR1173929} and in our proof of Theorem~\ref{thm:main}. It is perhaps helpful to remark here that Mockenhaupt, Seeger and Sogge \cite{MR1173929} used local smoothing estimates for linear wave operators; in the current paper, what we encounter is a local smoothing estimate for a variable-coefficient Schr\"odinger operator (see Theorem~\ref{thm: general propagator} below). These local smoothing estimates follow from the recent work \cite{CGGHMW}, which is reviewed in \S\ref{240613subsection1_6}. At their heart lies an observation of Wisewell \cite{Wisewell2005} regarding differences between the geometry of Kakeya and Nikodym sets of curves.

\subsection{An $L^2$ bound and its surprising proof.}

In this subsection we discuss the proof strategy for Theorem~\ref{thm:main} and further remark on the variable-coefficient local smoothing estimates for Sch\"odinger operators that feature in the argument.

To prove Theorem~\ref{thm:main}, we still prove a decay estimate of the form \eqref{240613e1_7}. More precisely, we prove that there exists a universal constant $\varepsilon>0$ such that
\begin{equation}\label{240613e1_7zzz}
\norm{
\mathcal{C}_{j} f
}_{L^2(\R^2)} \lesssim 2^{-\varepsilon j} \|f\|_{L^2(\R^2)}
\end{equation}
holds for all $j\in \N$,
where
\begin{equation*}
\mathcal{C}_{j} f(x, y):= \int_{\R} f(x-t, y-t^2)
 e^{iv(x, y) t^3} 
 \varphi_0(2^{-j}|t|)
\frac{\ud t}{t}
\end{equation*}
 for $v(x, y)$ a measurable function taking values in the interval $[1, 2]$. What is surprising to us is that we do not know how to prove \eqref{240613e1_7zzz} directly, and instead we show (see Theorem~\ref{thm:Cl} below) that there exist $p_{\circ}>2$ and $\varepsilon_{\circ}> 0$ such that 
\begin{equation}\label{240613e1_7ttt}
\norm{
\mathcal{C}_{j} f
}_{L^{p_{\circ}}(\R^2)} \lesssim 2^{-\varepsilon_{\circ} j} \|f\|_{L^{p_{\circ}}(\R^2)},
\end{equation}
for all $j\in \N$. This, together with the trivial bound 
\begin{equation*}
\norm{
\mathcal{C}_{j} f
}_{L^1(\R^2)} \lesssim \|f\|_{L^1(\R^2)},
\end{equation*}
implies \eqref{240613e1_7zzz}. In the literature related to (singular) Radon transforms in harmonic analysis, we are not aware of another instance where an $L^2$ decay estimate is proven via interpolation; usually it is natural to work directly in $L^2$ because of, for instance, self-adjoint properties and Plancherel's theorem. 

The estimate \eqref{240613e1_7ttt} (or, more generally, 
Theorem~\ref{thm:Cl} below) follows in part from a theorem for variable coefficient Schr\"odinger-type propagators recently proved in \cite{CGGHMW}. The general result in \cite{CGGHMW} can be thought of as a local smoothing variant of the celebrated work of Bourgain~\cite{Bourgain1991} on oscillatory integral operators, and is partly inspired by important observations of Wisewell~\cite{Wisewell2005} concerning Kakeya and Nikodym sets of curves. To apply the theory from \cite{CGGHMW}, one must check that the relevant phase functions satisfy certain conditions in addition to those of H\"ormander~\cite{Hormander1973}, called the \textit{Nikodym non-compression hypothesis} (see Definition~\ref{dfn: Nik non compress} below). This hypothesis ensures certain families of curves associated to the operator $\mathcal{C}_j$ cannot compress into small sets (what Bourgain~\cite{Bourgain1991} called the Kakeya compression phenomenon): we refer the reader to \cite{CGGHMW} for further details. 

We emphasise that \eqref{240613e1_7ttt} is not an off-the-shelf application of the aforementioned local smoothing results from \cite{CGGHMW}. An array of classical oscillatory integral and $L^2$-based techniques are applied to reduce the problem to analysing the most singular part of the operator where the local smoothing result applies. These arguments are partly inspired by methods developed by Pramanik--Seeger \cite{PS2007} to prove local smoothing-type results in the presence of fold singularities. Furthermore, it transpires that, even after the above reductions, the Nikodym non-compression hypothesis is not globally valid, and an exceptional piece of the operator must be analysed via alternative means. However, this piece is amenable to (variable coefficient) decoupling theory, which is a standard tool in the study of local smoothing estimates \cite{Wolff2000, BD2015, MR4078231, BHS}.

%%%%%%%%%%%%%%%%%%%%%%%%%%%%%%%%%%%%%%%%%%%%%%%%%%%%%%%%%%%%%%%%%%%%%%%%%%%%%%%%%%%%%%%%%%%%%%%%

%    H\"ormander-type phase conditions and the Nikodym non-compression hypothesis

%%%%%%%%%%%%%%%%%%%%%%%%%%%%%%%%%%%%%%%%%%%%%%%%%%%%%%%%%%%%%%%%%%%%%%%%%%%%%%%%%%%%%%%%%%%%%%%%

\subsection{H\"ormander-type phase conditions and the Nikodym non-compression hypothesis}\label{240613subsection1_6} In this subsection we review a key result from \cite{CGGHMW} that we shall use later. 
We begin by reviewing some of the standard definitions in the oscillatory integral operator literature (see, for example, \cite{CGGHMW, MR4047925}), in the special case of translation-invariant phases. For $n\ge 2$ and $\rho>0$, let
\begin{equation*}
    \D_{\rho}^n:= \R^{n-1} \times \B_{\rho}^1 \times \B_{\rho}^{n-1},
\end{equation*}
where  $\B_{\rho}^d$ denotes the open ball in $\R^d$ of radius $\rho$ centred at the origin. 

\begin{definition} We say $\Phi \colon \D_{\rho}^n \to \R$ is a \textit{translation-invariant phase} if there exists some smooth function $\phi \colon \B_{\rho}^1 \times \B_{\rho}^{n-1} \to \R$ such that
\begin{equation}\label{240611e11}
    \Phi(\bx,w;\bxi) := \bx \cdot \bxi + \phi(w;\bxi) \qquad \textrm{for all $(\bx,w;\bxi) \in \D_{\rho}^n$.}
\end{equation}
Furthermore, if $\Phi$ is of the form \eqref{240611e11}, then we say that $\Phi$ is a \textit{translation-invariant, H\"ormander-type phase} if 
\begin{equation}\label{eq: Hormander}
    \det
    \partial_w \partial^2_{\bxi\bxi} \phi(w; \bxi)
    \neq 0 \quad \textrm{for all $(w;\bxi) \in \B_{\rho}^1 \times \B_{\rho}^2$}.
\end{equation}
Often we pair the phase with a smooth cutoff function $a \in C_c^{\infty}(\B_{\rho}^1 \times \B_{\rho}^{n-1})$, in which case we refer to the pairing $[\Phi; a]$ as a \textit{phase-amplitude pair}. 
\end{definition}

Given $[\Phi; a]$ a phase-amplitude pair and $\lambda \geq 1$, we consider the rescaled data
\begin{equation*}
    \Phi^{\lambda}(\bx,w;\bxi) := \lambda \Phi(\bx/\lambda, w/\lambda;\bxi) \quad \textrm{and} \quad a^{\lambda}(w;\bxi) := a(w/\lambda;\bxi).
\end{equation*}
The \textit{variable Schr\"odinger propagator} $U^{\lambda} = U^{\lambda}[\Phi; a]$ associated to $[\Phi; a]$ is the operator
\begin{equation}\label{eq: Schrod prop}
U^{\lambda}f(\bx,w) :=  \int_{\widehat{\R}^2} e^{i \Phi^{\lambda}(\bx,w; \bxi)}a^{\lambda}(w;\bxi) \hat{f}(\bxi)\ud \bxi,
\end{equation}
initially defined as acting on $f \in \cS(\R^2)$.

Any variable Schr\"odinger propagator associated to a H\"ormander-type phase satisfies a range of \textit{universal} estimates, given by Stein's oscillatory integral theorem \cite{Stein1986}. Our approach requires estimates for specific operators (related to the $\cC_j$ above) which go beyond this universal range. Examples due to Bourgain~\cite{Bourgain1991} and Wisewell~\cite{Wisewell2005} demonstrate Stein's theorem \cite{Stein1986} is sharp over the class of general translation-invariant, H\"ormander-type phases. Thus, to proceed we must impose additional conditions on our class of phase functions. 

\begin{definition}\label{dfn: Nik non compress} For $0 < \rho < 1$, let $\Phi \colon \D^2_{\rho}  \to \R$ be a real analytic translation-invariant phase of the form \eqref{240611e11}. For $\gamma \in (0, 1)$, we say that $\Phi$ satisfies  \textit{the Nikodym non-compression hypothesis with parameter $\gamma$} if :
\begin{enumerate}[a)]
\item For all $(w; \bxi) \in \B^1_{\rho} \times \B^2_{\rho}$, we have
$\det \partial^2_{\bxi\bxi}\phi(w; \bxi)=0$; 
    \item For all $(w; \bxi) \in \B^1_{\rho} \times \B^2_{\rho}$, we have
    \begin{equation*}
        \left| \det
        \begin{bmatrix}
            \partial_w \partial^2_{\xi_1 \xi_1} \phi(w; \bxi), & \partial_w \partial^2_{\xi_1 \xi_2} \phi(w; \bxi), & \partial_w \partial^2_{\xi_2 \xi_2} \phi(w; \bxi)\\
             \partial^2_w \partial^2_{\xi_1 \xi_1} \phi(w; \bxi), & \partial^2_w \partial^2_{\xi_1 \xi_2} \phi(w; \bxi), & \partial^2_w \partial^2_{\xi_2 \xi_2} \phi(w; \bxi)\\
              \partial^3_w \partial^2_{\xi_1 \xi_1} \phi(w; \bxi), & \partial^3_w \partial^2_{\xi_1 \xi_2} \phi(w; \bxi), & \partial^3_w \partial^2_{\xi_2 \xi_2} \phi(w; \bxi)
        \end{bmatrix}
        \right| \ge \gamma.
    \end{equation*}
\end{enumerate}
\end{definition}

Definition~\ref{dfn: Nik non compress} is a quantified variant of a special case of the Nikodym non-compression hypothesis from \cite{CGGHMW}. Under this hypothesis, we have the following (weakly) quantified local smoothing estimate.

\begin{theorem}[{\cite[Theorem 2.8 II)]{CGGHMW}}]\label{thm: general propagator}
There exist some $2 < p_{\circ} < 4$, $\beta(p_{\circ})< 1/2$ and large $M_{\circ} \geq 1$, $N_{\circ} \geq 1$ such that the following holds. Let $\rho > 0$, $0 < \gamma \leq 1$ and $W\ge 1$ and $[\Phi; a]$ be a phase-amplitude pair where $\Phi \colon \D^2_{\rho} \to \R$ is a real analytic translation-invariant phase which satisfies the Nikodym non-compression hypothesis with parameter $\gamma$. Further suppose that 
 \begin{equation}\label{eq: Hormander gamma}
     |\det\partial_w \partial^2_{\bxi\bxi} \phi(w; \bxi)
     |\ge \gamma \qquad \textrm{for all $(w; \bxi)\in \B^1_{\rho}\times \B^2_{\rho}$};
 \end{equation}
 the amplitude satisfies $\| a \|_{L^\infty(\B_\rho^1  \times \B_\rho^2)}=1$, and
 \begin{equation}\label{eq: phase/amp derivs}
     |\partial_w^{\alpha} \partial^{\bbeta}_{\bxi} \phi(w; \bxi)|\leq W \quad \textrm{and} \quad |\partial_w^{\alpha} \partial^{\bbeta}_{\bxi} a(w; \bxi)|\leq W \text{ for all } (w; \bxi)\in \B^1_{\rho} \times \B_\rho^2
 \end{equation}
 and all $\alpha \in \N_0$, $\bbeta = (\beta_1, \beta_2) \in \N_0^2$ satisfying $1 \leq \alpha + \beta_1 + \beta_2 \leq N_{\circ}$. 
Then
\begin{equation}\label{240611main_estimate}
    \|U^{\lambda}[\Phi; a] f\|_{L^{p_{\circ}}(\R^3)} \lesssim \lambda^{\beta(p_{\circ})}
    \big(W\gamma^{-1}\big)^{M_{\circ}}
    \|f\|_{L^{p_{\circ}}(\R^2)}
\end{equation}
for all $\lambda\ge 1$. 
\end{theorem}

This result follows from \cite[Theorem 2.8 II)]{CGGHMW} by tracking the dependencies on the parameters $\gamma$ and $W$ in the proof: see Appendix~\ref{appendix: quantification} for further details.\footnote{In \cite[Theorem 2.8 II)]{CGGHMW} the operators $U^{\lambda}[\Phi; a]$ are smoothly localised in the $\bx$ variable to a $\lambda$-ball. However, the global variant in Theorem~\ref{thm: general propagator} easily follows by the pseudo-local properties of the operator (see, for instance, the proof of Proposition~\ref{prop: ls decoupling} at the end of \S\ref{subsec: decoupling arg} below).} 

Note that, by Plancherel's theorem (since our phases are translation-invariant),  \eqref{240611main_estimate} holds with 
\begin{equation}\label{240611estimate_2}
    p_{\circ}=2, \qquad \beta(p_{\circ})=1/2.
\end{equation}
Moreover, by Stein's oscillatory integral theorem \cite{Stein1986} (which generalises the classical Stein--Tomas restriction theorem), we see that 
\eqref{240611main_estimate} holds with 
\begin{equation}\label{240611estimate_3}
    p_{\circ}=4, \qquad \beta(p_{\circ})=1/2.
\end{equation}
Therefore, if we do not aim for $\beta(p_{\circ})< 1/2$ in Theorem~\ref{thm: general propagator}, then we can just interpolate \eqref{240611estimate_2} and \eqref{240611estimate_3}. Unfortunately, the resulting estimate \eqref{240611main_estimate} with $2\le p_{\circ}\le 4$ and  $\beta(p_{\circ})=1/2$ falls slightly short of our needs: it is crucial that we have $\beta(p_{\circ})< 1/2$ in Theorem~\ref{thm: general propagator} in order to sum up high frequencies (see Theorem~\ref{thm:Cl} below).

After proving Theorem~\ref{thm: general propagator}, we can interpolate it with \eqref{240611estimate_2} and \eqref{240611estimate_3} to see that Theorem~\ref{thm: general propagator} holds for all $2< p_{\circ}< 4$. This key estimate is used later in the proof of Propositions~\ref{prop:Lp smoothing low} and \ref{prop:Lp smoothing diag}.

\subsection{The full planar Pierce--Yung operator.} It may be possible to generalise the argument in the current paper to prove that 
\begin{equation*}
    \sup_{v\in \R} \anorm{\mathrm{p.v.}
    \int_{\R}
    f(x-t, y-t^2)e^{iv t^m}\frac{\ud t}{t}
    }
\end{equation*}
is bounded on $L^p(\R^2)$ for every $p\in (1, \infty)$ and every integer $m\ge 4$. However, it remains an interesting question whether the full Pierce--Yung operator 
\begin{equation*}
    \sup_{v_3, \dots, v_m \in \R} 
    \anorm{
    \mathrm{p.v.}\int_{\R}
    f(x-t, 
    y-t^2
    )
    e^{
    i(v_3 t^3+\dots+ v_m t^m)
    }
    \frac{\ud t}{t}
    }
\end{equation*}
is bounded on $L^p(\R^2)$ for every integer $m\ge 4$ and $p\in (1, \infty)$. 

%%%%%%%%%%%%%%%%%%%%%%%%%%%%%%%%%%%%%%%%%%%%%%%%%%%%%%%%%%%%%%%%%%%%%%%%%%%%%%%%%%%%%%%%%%%%%%%%

%   Notational conventions

%%%%%%%%%%%%%%%%%%%%%%%%%%%%%%%%%%%%%%%%%%%%%%%%%%%%%%%%%%%%%%%%%%%%%%%%%%%%%%%%%%%%%%%%%%%%%%%%

\subsection*{Notational conventions} We let $\widehat{\R}$ denote the \textit{frequency domain}, which is the Pontryagin dual group of $\R$  understood here as simply a copy of $\R$. Given $f \in L^1(\R^d)$ and $g \in L^1(\widehat{\R}^d)$ we define the Fourier transform and inverse Fourier transform by
\begin{equation*}
    \hat{f}(\bxi) := \int_{\R^d} e^{-i \inn{\bx}{\bxi}} f(\bx)\,\ud \bx \quad \textrm{and} \quad \check{g}(x) := \frac{1}{(2\pi)^d}\int_{\widehat{\R}^d} e^{i \inn{\bx}{\bxi}} g(\bxi)\,\ud \bxi,
\end{equation*}
respectively. Given $a, b \in \R$ we write $a \wedge b := \min \{a,b\}$ and $a \vee b := \max \{a,b\}$.

%%%%%%%%%%%%%%%%%%%%%%%%%%%%%%%%%%%%%%%%%%%%%%%%%%%%%%%%%%%%%%%%%%%%%%%%%%%%%%%%%%%%%%%%%%%%%%%%

%   Acknowledgementes

%%%%%%%%%%%%%%%%%%%%%%%%%%%%%%%%%%%%%%%%%%%%%%%%%%%%%%%%%%%%%%%%%%%%%%%%%%%%%%%%%%%%%%%%%%%%%%%%

\subsection*{Acknowledgements}

D.B. is supported by the AEI grants RYC2020-029151-I and PID2022-140977NA-I00. S.G. is partly supported by the Nankai Zhide Foundation, and partly supported by NSF-2044828.

%%%%%%%%%%%%%%%%%%%%%%%%%%%%%%%%%%%%%%%%%%%%%%%%%%%%%%%%%%%%%%%%%%%%%%%%%%%%%%%%%%%%%%%%%%%%%%%%

%    Decomposition of the kernel

%%%%%%%%%%%%%%%%%%%%%%%%%%%%%%%%%%%%%%%%%%%%%%%%%%%%%%%%%%%%%%%%%%%%%%%%%%%%%%%%%%%%%%%%%%%%%%%%

\section{Decomposition of the kernel}\label{sec:kernel dec}

We begin with some elementary preliminaries for the proof of Theorem~\ref{thm:main}.

%%%%%%%%%%%%%%%%%%%%%%%%%%%%%%%%%%%%%%%%%%%%%%%%%%%%%%%%%%%%%%%%%%%%%%%%%%%%%%%%%%%%%%%%%%%%%%%%

%    Decomposition relative to the oscillatory factor
%%%%%%%%%%%%%%%%%%%%%%%%%%%%%%%%%%%%%%%%%%%%%%%%%%%%%%%%%%%%%%%%%%%%%%%%%%%%%%%%%%%%%%%%%%%%%%%%

\subsection{Decomposition relative to the oscillatory factor} Let $\beta_0 \in C^\infty_c(\R)$ be real, even and radially non-increasing, supported in $[-2,2]$ and satisfy $\beta_0(s)=1$ for $s \in [-1,1]$. Define $\beta(s):= \beta_0(s) - \beta_0(2s)$ so that
\begin{equation*}
    \sum_{j \in \Z} \beta(2^{-j} s) = 1 \quad \textrm{and} \quad \beta_0(s) = \sum_{j \leq 0} \beta(2^{-j} s) \quad \textrm{for all $s \in \R \setminus \{0\}$.}
\end{equation*}
For any $m \in \Z$, define
\begin{align*}
    H^v_{\leq m} f(x,y)& :=  \mathrm{p.v.}  \int_\R f(x-t,y-t^2) e^{i v t^3} \beta_0(2^{-m} t) \frac{\ud t}{t}, \\
    H^v_{m} f(x,y) & := \int_\R f(x-t,y-t^2) e^{i v t^3} \beta(2^{-m} t) \frac{\ud t}{t}.
\end{align*}
We then decompose
\begin{equation}\label{eq:C dec}
    \mathcal{C}f(x,y) \leq  \mathcal{C}_{\leq 0} f(x,y)+  \sum_{\ell > 0} \mathcal{C}_\ell f(x,y)
\end{equation}
where
\begin{equation*}
    \mathcal{C}_{\leq 0} f (x,y) := \sup_{j \in \Z} \sup_{v \in V_j} |H^v_{\leq -j} f(x,y)|,\qquad
    \mathcal{C}_{\ell} f (x,y) := \sup_{j \in \Z} \sup_{v \in V_j} |H^v_{-j+\ell} f(x,y)|
\end{equation*}
for $V_j := [2^{3j}, 2^{3(j+1)}]$.

%%%%%%%%%%%%%%%%%%%%%%%%%%%%%%%%%%%%%%%%%%%%%%%%%%%%%%%%%%%%%%%%%%%%%%%%%%%%%%%%%%%%%%%%%%%%%%%%

%    Non-oscillatory analysis

%%%%%%%%%%%%%%%%%%%%%%%%%%%%%%%%%%%%%%%%%%%%%%%%%%%%%%%%%%%%%%%%%%%%%%%%%%%%%%%%%%%%%%%%%%%%%%%%

\subsection{Non-oscillatory analysis} The term $\mathcal{C}_{\leq 0}$ has no oscillation and can be controlled by the maximal function and the Hilbert transform along the parabola. Recall that these are the operators
\begin{align*}
    M_{\mathrm{par}} f(x,y) & := \sup_{r >0} \frac{1}{2r} \int_{-r}^r |f(x-t,y-t^2)| \ud t, \\
    H_{\mathrm{par}} f(x,y) & := \mathrm{p.v.}  \int_{\R} f(x-t,y-t^2)  \frac{\ud t}{t}
\end{align*}
which are both well-known to be bounded on $L^p(\R^2)$ for $1 < p < \infty$ (see, for instance, \cite[Chapter 8, \S7]{Duoandikoetxea2001}). Moreover, it was shown in \cite[Theorem 5.1]{PY2019}  that $L^p$-boundedness also holds for the maximal truncations
\begin{equation}\label{eq: truncated HT}
    \sup_{\varepsilon>0} |H_{
    \mathrm{par}, \varepsilon
    } f(x,y)|, \qquad \text{ where } \quad  H_{
    \mathrm{par}, \varepsilon
    } f(x,y) := \int_{|y| \geq \varepsilon} f(x-t,y-t^2)  \frac{\ud t}{t}.
\end{equation}
In view of these known results, we quickly deduce the following lemma.

\begin{lemma}\label{lem:C0}
    For every $ 1 < p < \infty$, we have
    \begin{equation*}
        \| \mathcal{C}_{\leq 0} f \|_{L^p(\R^2)} \lesssim_p \| f \|_{L^p(\R^2)}.
    \end{equation*}
\end{lemma}

\begin{proof} Write
    \begin{equation*}
        H_{\leq - j}^v  = \big( H_{\leq - j}^v  - H_{\leq -j}^0 \big) + H_{\leq -j}^0 =: M_{ - j}^v + H_{\leq - j}^0. 
    \end{equation*}
    Then     \begin{align*}
        \sup_{v \in V_j} |M_{-j}^v f(x,y) | & = \sup_{v\in V_j}
        \Big| \int_{\R} f(x-t, y-t^2) \big( e^{i v t^3 } -1 \big) \beta_0(2^j t) \frac{\ud t}{t} \Big| \\
        & \lesssim 2^j \int_{\R} |f(x-t,y-t^2)| \beta_0(2^j t) \ud t \\
        & \lesssim \frac{1}{2^{-j+2}} \int_{|t| \leq 2^{-j+1}} |f(x-t,y-t^2)| \ud t.
    \end{align*}
    Thus,
    \begin{equation*}
        \sup_{j \in \Z} \sup_{v \in V_j} |M_{-j}^v f(x,y) | \lesssim M_{\mathrm{par}} f(x,y) 
    \end{equation*}
    and the $L^p$-boundedness for this term follows. 
    
    For the second term, since $|t| \leq 2^{-j+1}$ for all $t \in \supp \beta_0(2^j\,\cdot\,)$, we note that
    \begin{equation*}
    \begin{split}
        H^0_{\leq - j}f(x,y) & = H_{\mathrm{par}}f(x,y) - H_{
        \mathrm{par}, 
        2^{-j+1}}f(x,y)\\
        & + \int_{|t| \leq 2^{-j+1}} f(x-t,y-t^2) \big( \beta_0(2^jt) -1 \big) \frac{\ud t}{t},
    \end{split}
    \end{equation*}
    where $H_{
    \mathrm{par}, 
    2^{-j+1}}f$ is a truncation as defined in \eqref{eq: truncated HT}. Since $\beta_0(s)=1$ for $s \in [-1,1]$, we have 
    \begin{align*}
        & \Big|\int_{|t| \leq 2^{-j+1}} f(x-t,y-t^2) \big( \beta_0(2^jt) -1 \big) \frac{\ud t}{t} \Big| \\
        & \lesssim \frac{1}{2^{-j}} \int_{2^{-j} \leq |t| \leq 2^{-j+1}} |f(x-t,y-t^2)| \ud t \lesssim M_{\mathrm{par}} f(x,y).
    \end{align*}
    Therefore,
    \begin{align*}
        \sup_{j \in \Z} \sup_{v \in V_j} |H_{\leq -j}^0 f(x,y)| \leq |H_{\mathrm{par}} f(x,y)| + \sup_{\varepsilon>0} |H_\varepsilon f(x,y)| + M_{\mathrm{par}} f(x,y)
    \end{align*}
    and the $L^p$-boundedness follows.
\end{proof}

Regarding the terms $\mathcal{C}_\ell$, we have the following uniform bound.

\begin{lemma}\label{lem:uniform Cl}
    For all $1 < p \leq  \infty$, we have
    \begin{equation*}
        \| \mathcal{C}_\ell f \|_{L^p(\R^2)} \lesssim_p \| f \|_{L^p(\R^2)}.    
    \end{equation*}
\end{lemma}

\begin{proof} We trivially bound
    \begin{align*}
        |H_{-j+\ell}^v f(x,y)| & \leq \Big| \int_{\R} |f(x-t,y-t^2)| \beta(2^{j-\ell}t) \frac{\ud t }{|t|}\Big|  \\
        & \lesssim \frac{1}{2^{-j+\ell + 2}} \int_{|t| \leq 2^{- j + \ell + 1}} |f(x-t,y-t^2)| \ud t \lesssim M_{\mathrm{par}} f(x,y)
    \end{align*}
uniformly in $v$ and $j$, so the result follows from the $L^p(\R^2)$-boundedness of $M_{\mathrm{par}}$.
\end{proof}

%%%%%%%%%%%%%%%%%%%%%%%%%%%%%%%%%%%%%%%%%%%%%%%%%%%%%%%%%%%%%%%%%%%%%%%%%%%%%%%%%%%%%%%%%%%%%%%%

%    Our goal

%%%%%%%%%%%%%%%%%%%%%%%%%%%%%%%%%%%%%%%%%%%%%%%%%%%%%%%%%%%%%%%%%%%%%%%%%%%%%%%%%%%%%%%%%%%%%%%%

\subsection{Our goal} Lemma~\ref{lem:uniform Cl} falls short of proving Theorem~\ref{thm:main} since the estimates are not summable in $\ell$. The key ingredient in the proof of Theorem~\ref{thm:main} is the following improvement.

\begin{theorem}\label{thm:Cl}
    There exists some $p_{\circ}>2$ and $\varepsilon_{\circ}>0$ such that
        \begin{equation}\label{eq:Cl main}
        \| \mathcal{C}_\ell f \|_{L^{p_{\circ}}(\R^2)} \lesssim 2^{-\ell \varepsilon_{\circ}} \| f \|_{L^{p_{\circ}}(\R^2)}
    \end{equation}
    holds for all $\ell\ge 1$.
\end{theorem}

Once this is established, interpolation with Lemma~\ref{lem:uniform Cl} immediately upgrades \eqref{eq:Cl main} to all $1 < p < \infty$. Theorem~\ref{thm:main} is then an immediate consequence of this upgrade and Lemma~\ref{lem:C0}, by virtue of the decomposition \eqref{eq:C dec} and the triangle inequality.

%%%%%%%%%%%%%%%%%%%%%%%%%%%%%%%%%%%%%%%%%%%%%%%%%%%%%%%%%%%%%%%%%%%%%%%%%%%%%%%%%%%%%%%%%%%%%%%%

%   Reduction to key estimates

%%%%%%%%%%%%%%%%%%%%%%%%%%%%%%%%%%%%%%%%%%%%%%%%%%%%%%%%%%%%%%%%%%%%%%%%%%%%%%%%%%%%%%%%%%%%%%%%

\section{Reduction to key estimates}

%%%%%%%%%%%%%%%%%%%%%%%%%%%%%%%%%%%%%%%%%%%%%%%%%%%%%%%%%%%%%%%%%%%%%%%%%%%%%%%%%%%%%%%%%%%%%%%%

%    Frequency decomposition

%%%%%%%%%%%%%%%%%%%%%%%%%%%%%%%%%%%%%%%%%%%%%%%%%%%%%%%%%%%%%%%%%%%%%%%%%%%%%%%%%%%%%%%%%%%%%%%%

\subsection{Frequency decomposition}\label{subsec:freq loc} Let $\beta \in C_c^{\infty}(\R)$ be as in \S\ref{sec:kernel dec}. For any $m \in \Z$ and $i=1,2$, let $P_m^{(i)}$ denote the $i$th coordinate Littlewood--Paley projection associated with $\beta (2^{-m} \,\cdot\,)$: that is, for $f \in \cS(\R^2)$, we define
\begin{equation*}
    (P_m^{(1)} f)\;\widehat{}\;(\xi, \eta)  := \beta(2^{-m} \xi) \widehat{f}(\xi, \eta),   \qquad
    (P_m^{(2)} f)\;\widehat{}\;(\xi, \eta)  := \beta(2^{-m} \eta) \widehat{f}(\xi, \eta).     
\end{equation*}
Note that 
\begin{equation*}
    f= \sum_{m_1, m_2 \in \Z}  P_{m_1}^{(1)} P_{m_2}^{(2)} f.
\end{equation*}
It is also convenient to define the operators
\begin{equation*}
 P_{\leq m}^{(i)} := \sum_{n \leq m} P_{n}^{(i)},   
\end{equation*}
 which, for $i = 1$, $2$, are associated to the multipliers $\beta_0(2^{-m} \xi)$ and $\beta_0(2^{-m}\eta)$, respectively. We shall also make use of Littlewood--Paley projections $\widetilde{P}_m^{(i)}$ satisfying
 \begin{equation*}
P_m^{(i)}=P_m^{(i)}\widetilde{P}_m^{(i)}
 \end{equation*}
 and associated to a cut-off function $\widetilde{\beta} \in C_c^\infty(\R)$ supported in $[-4,-1/4]\cup [1/4,4]$ and that equals $1$ on $\supp \beta \subseteq [-2,-1/2]\cup[1/2,2]$. 

For $\ell \in \N$, $j \in \Z$ and $v \in V_j$, we decompose
\begin{equation*}
    H_{-j+\ell}^v f = \sum_{k_1 \in \Z} \sum_{k_2 \in \Z} H_{-j+\ell}^v P_{k_1+2\ell+j}^{(1)} P_{k_2+\ell+2j}^{(2)}f.
\end{equation*}
Moreover, we group the low frequencies and write, for $C_1, C_2 \geq 10^6$ satisfying $C_1 = 2 C_2 + 12$ fixed (absolute) constants,
\begin{align}
    H_{-j+\ell}^v f  &= H_{-j+\ell}^v P_{\leq -C_1+2\ell+j}^{(1)} P_{\leq -C_2+\ell+2j}^{(2)}f \notag \\
    & \quad + \sum_{k_1 >-C_1} H_{-j+\ell}^v P_{k_1+2\ell+j}^{(1)} P_{\leq -C_2+\ell+2j}^{(2)}f \notag \\
    & \quad 
    +\sum_{k_2 >-C_2} H_{-j+\ell}^v P_{\leq -C_1+2\ell+j}^{(1)} P_{k_2+\ell+2j}^{(2)}f \notag \\
    & \quad 
    +\sum_{k_1 > -C_1} \sum_{k_2 >-C_2} H_{-j+\ell}^v P_{k_1+2\ell+j}^{(1)} P_{k_2+\ell+2j}^{(2)}f. \label{eq:freq dec}
\end{align}
Note that the Fourier multiplier associated to $H_{-j+\ell}^v P_{k_1+2\ell+j}^{(1)} P_{k_2+\ell+j}^{(2)}$ is given by
\begin{equation}\label{eq:mult}
    \beta(2^{-k_1-2\ell-j} \xi ) \beta(2^{-k_2-\ell-2j} \eta) \int_\R e^{-i (t \xi + t^2 \eta - t^3 v)} \beta(2^{-\ell+j} t) \frac{\ud t }{t},
\end{equation}
and similar formulas hold for the other frequency projections. 

The low frequency part of the operator is easily dealt with. 
\begin{lemma}[Low frequency part]\label{lem:low}
        For $2 \leq p \leq \infty$, the inequality  
        \begin{equation*}
            \| \sup_{j \in \Z} \sup_{v \in V_j} |H^v_{-j+\ell} P_{\leq -C_1+2\ell+j}^{(1)} P_{\leq -C_2 +\ell + 2j}^{(2)} f| \|_{L^p(\R^2)} \lesssim_{p, N, C_1, C_2} 2^{- \ell N} \| f \|_{L^p(\R^2)} 
        \end{equation*}
        holds for all $N \in \N_0$ and all $\ell \in \N$.
\end{lemma}

The proof of Lemma~\ref{lem:low} is based on pointwise comparison with the strong maximal function, which is achieved by estimating the kernel using elementary integration-by-parts arguments. We postpone the details until \S\ref{subsec:low}. 

For the high frequencies, we introduce a rescaled variant $\cH_n^w$ of $H_{-j+\ell}^v$ for which we prove the key estimates and then relate back to $H_{-j+\ell}^v$. Given $n \in \N$, $w > 0$, define
\begin{equation*}
    \cH_n^w f(x,y) := \frac{1}{(2\pi)^2} \int_{\widehat{\R}^2} e^{i (x \xi + y \eta )} \int_\R e^{-i 2^n(t \xi + t^2 \eta - t^3 w)} \beta(t) \frac{\ud t}{t} \widehat{f}(\xi,\eta) \ud \xi \ud \eta
\end{equation*}
for all $f \in \cS(\R^2)$. Furthermore, for $(k_1, k_2) \in \Z^2$ we define the frequency localised operators
\begin{gather*}
\cH_{n,k_1,k_2}^w := \cH_n^w P_{k_1}^{(1)}P_{k_2}^{(2)}, \quad \cH_{n,\leq k_1, k_2}^w := \cH_n^w P_{\leq k_1}^{(1)}P_{k_2}^{(2)}, \quad
\cH_{n, k_1, \leq k_2}^w := \cH_n^w P_{k_1}^{(1)}P_{\leq k_2}^{(2)}.
\end{gather*}
With this notation, our goal is to prove the following.

\begin{theorem}\label{thm:key} There exist $p_{\circ}>2$ and $\varepsilon_{\circ}>0$, $\delta_{\circ}>0$ such that the inequalities
    \begin{align}
        \Big( \int_1^8 \| \cH^w_{n, k_1, \leq -C_2} f \|_{L^{p_{\circ}}(\R^2)}^{p_{\circ}} \ud w \Big)^{1/p_{\circ}} & \lesssim 2^{-n/p_{\circ}- n \varepsilon_{\circ}-k_1 \delta_{\circ}} \| f \|_{L^{p_{\circ}}(\R^2)}, \label{eq:low freq eta}
        \\
        \Big( \int_1^8 \| \cH^w_{n, \leq -C_1, k_2} f \|_{L^{p_{\circ}}(\R^2)}^{p_{\circ}} \ud w \Big)^{1/p_{\circ}} & \lesssim 2^{-n/p_{\circ}- n \varepsilon_{\circ}-k_2 \delta_{\circ}} \| f \|_{L^{p_{\circ}}(\R^2)}, \label{eq:low freq xi}
        \\
        \Big( \int_1^8 \| \cH_{n,k_1,k_2}^w f \|_{L^{p_{\circ}}(\R^2)}^{p_{\circ}} \ud w \Big)^{1/p_{\circ}} & \lesssim 2^{-n/p_{\circ}- n \varepsilon_{\circ}-(k_1+k_2) \delta_{\circ}} \| f \|_{L^{p_{\circ}}(\R^2)} \label{eq:high freq}
    \end{align}
    hold for all $n \in \N$, $k_1> -C_1$ and $k_2 > - C_2$.
\end{theorem}

Once the above theorem is established, Theorem~\ref{thm:Cl} follows from standard Littlewood--Paley, Sobolev embedding and scaling arguments.

\begin{proof}[Proof of Theorem~\ref{thm:Cl} (assuming Lemma~\ref{lem:low} and Theorem~\ref{thm:key})] 
Let $\ell \in \N$, $j \in \Z$, $k_1,k_2 \in \Z$. Performing the change of variables 
\begin{equation*}%\label{240303e4_8}
    (t; \xi,\eta) \mapsto (2^{\ell-j}t; 2^{2\ell + j} \xi, 2^{\ell + 2j}\eta)
\end{equation*}
in \eqref{eq:mult}, we obtain
\begin{equation*}
    H_{-j+\ell}^v P_{k_1+2\ell+j}^{(1)} P_{k_2 + \ell + 2j}^{(2)} f(x,y) = \cH_{3\ell}^{2^{-3j} v} P_{k_1}^{(1)} P_{k_2}^{(2)} f_{\ell,j} (2^{2\ell+j} x, 2^{\ell+2j}y)
\end{equation*}
where
\begin{equation*}
    f_{\ell,j}(x,y) := f(2^{-2\ell-j} x, 2^{-\ell-2j} y).
\end{equation*}
Recall that $V_j := [2^{3j}, 2^{3(j+1)}]$. By the above scaling, for any fixed $j \in \Z$, $p \geq 1$ and $M > 0$, the norm inequality
\begin{equation*}
    \| \sup_{v \in V_j} |H_{-j+\ell}^{v} P_{k_1+2\ell+j}^{(1)} P_{k_2+\ell+2j}^{(2)} f| \|_{L^p(\R^2)} \leq M\|f\|_{L^p(\R^2)}
\end{equation*}
holds for all $f \in \cS(\R^2)$ if and only if the norm inequality
\begin{equation*}
\| \sup_{1 \leq w \leq 8} |\cH^{w}_{3\ell, k_1, k_2} f| \|_{L^p(\R^2)}\leq M\|f\|_{L^p(\R^2)}
\end{equation*}
holds for all $f \in \cS(\R^2)$, with the same constant $M$. A similar statement holds if either $P_{k_1}^{(1)}$ or $P_{k_2}^{(2)}$ is replaced with $P_{\leq k_1}^{(1)}$ or $P_{\leq k_2}^{(2)}$.

For $g \in \cS(\R^2)$, by the Fundamental Theorem of Calculus and Hölder's inequality,
    \begin{equation*}
        \sup_{1 \leq w \leq 8} |\cH_{3\ell}^w g|^p \leq |\cH^1_{3\ell} g|^p + p \Big(\int_1^8 |\cH^w_{3\ell} g|^p \ud w \Big)^{1 - 1/p} \Big( \int_1^8 |\partial_w \cH_{3\ell}^w g|^p \ud w \Big)^{1/p}.
    \end{equation*}
Furthermore,
\begin{equation*}
\partial_w \cH^w_{3\ell} g = 3i2^{3\ell} \widetilde{\cH}^w_{3\ell} g,  
\end{equation*}
where the operator $\widetilde{\cH}^w_{3\ell}$ is defined by the Fourier multiplier
\begin{equation*}
    \int_{\R} e^{i 2^{3 \ell} (-t \xi - t^2 \eta + t^3 w)}\beta(t) t^2 \ud t,
\end{equation*}
which is essentially identical to the multiplier of $\cH^w_{3\ell}$. Combining the above observations and letting $g$ be suitable Littlewood--Paley projections of $f$, provided $C_1$, $C_2 \geq 1$ are chosen sufficiently large, Theorem~\ref{thm:key} implies
\begin{align*}
    \| \sup_{v \in V_j} | H^v_{-j+\ell} P_{ k_1+2\ell+j}^{(1)} P_{\leq -C_2+\ell+2j}^{(2)} f| \|_{L^{p_{\circ}}(\R^2)} & \lesssim 2^{- 3\ell \varepsilon_{\circ}-k_1 \delta_{\circ}} \| \widetilde{P}_{k_1+2\ell+j}^{(1)}f \|_{L^{p_{\circ}}(\R^2)}, \\
    \| \sup_{v \in V_j} | H^v_{-j+\ell} P_{\leq -C_1+2\ell+j}^{(1)} P_{k_2+\ell+2j}^{(2)} f| \|_{L^{p_{\circ}}(\R^2)}  & \lesssim 2^{- 3\ell \varepsilon_{\circ}-k_2 \delta_{\circ}} \| \widetilde{P}^{(2)}_{k_2+\ell+2j}f \|_{L^{p_{\circ}}(\R^2)}, \\
  \| \sup_{v \in V_j} | H^v_{-j+\ell} P_{ k_1+2\ell+j}^{(1)} P_{k_2+\ell+2j}^{(2)} f \|_{L^{p_{\circ}}(\R^2)}  &\lesssim 2^{- 3\ell \varepsilon_{\circ}-(k_1+k_2) \delta_{\circ}} \| \widetilde{P}^{(1)}_{k_1+2\ell+j} f \|_{L^{p_{\circ}}(\R^2)},  
\end{align*}  
where we have used that $P_m^{(i)}=P_m^{(i)}\widetilde{P}_m^{(i)}$ for $m \in \Z$ and $i = 1$, $2$. 

By (one-dimensional) Littlewood--Paley theory and Fubini's theorem, we have that
\begin{equation*}
    \Big( \sum_{j \in \Z} \| \widetilde{P}^{(i)}_j f \|_{L^p(\R^2)}^p \Big)^{1/p} \lesssim \| f \|_{L^p(\R^2)} \qquad \textrm{for all $2 \leq p < \infty$.}
\end{equation*}
 Using that $\sup_{j \in \Z} |a_j|^p \leq \sum_{j \in \Z} |a_j|^p$, the above estimates therefore upgrade to
\begin{align*}
    \sum_{k_1 > -C_1}\| \sup_{j \in \Z} \sup_{v \in V_j} | H^v_{-j+\ell} P_{k_1+2\ell+j}^{(1)} P_{\leq -C_2+\ell+2j}^{(2)} f| \|_{L^{p_{\circ}}(\R^2)}  & \lesssim 2^{- 3\ell \varepsilon_{\circ}} \| f \|_{L^{p_{\circ}}(\R^2)}, \\
    \sum_{k_2 > -C_2}\| \sup_{j \in \Z} \sup_{v \in V_j} | H^v_{-j+\ell} P_{\leq -C_1+2\ell+j}^{(1)} P_{k_2+\ell+2j}^{(2)} f| \|_{L^{p_{\circ}}(\R^2)}  & \lesssim 2^{- 3\ell \varepsilon_{\circ}} \| f \|_{L^{p_{\circ}}(\R^2)}, \\
    \sum_{k_1 > -C_1}\sum_{k_2 > -C_2}\| \sup_{j \in \Z} \sup_{v \in V_j} | H^v_{-j+\ell} P_{k_1+2\ell+j}^{(1)} P_{k_2+\ell+2j}^{(2)} f| \|_{L^{p_{\circ}}(\R^2)}  & \lesssim 2^{- 3\ell \varepsilon_{\circ}} \| f \|_{L^{p_{\circ}}(\R^2)}.
\end{align*}
In view of the decomposition \eqref{eq:freq dec}, theses estimates, together with Lemma~\ref{lem:low} and an application of the triangle inequality, imply \eqref{eq:Cl main}.
\end{proof}

%%%%%%%%%%%%%%%%%%%%%%%%%%%%%%%%%%%%%%%%%%%%%%%%%%%%%%%%%%%%%%%%%%%%%%%%%%%%%%%%%%%%%%%%%%%%%%%%

%    Critical points

%%%%%%%%%%%%%%%%%%%%%%%%%%%%%%%%%%%%%%%%%%%%%%%%%%%%%%%%%%%%%%%%%%%%%%%%%%%%%%%%%%%%%%%%%%%%%%%%

\subsection{Isolating the 
critical points}\label{subsec:critical points} 
Let $m_{n}^w$ denote the multiplier associated to the operator $\cH_n^w$, so that
\begin{equation}\label{eq: multiplier 1}
    m_{n}^w(\xi, \eta):=\int_{\R} e^{i 2^n\phi_{\xi,\eta}^w(t)} a(t) \ud t 
\end{equation}
where
\begin{equation}\label{eq: multiplier 2}
    \phi_{\xi,\eta}^w(t) := -t \xi - t^2 \eta + t^3 w \quad \textrm{and} \quad a(t):=\frac{\beta(t)}{t}.
\end{equation}
For $(k_1, k_2) \in \Z^2$, the multipliers of the frequency localised operators $\cH_{n,k_1,k_2}^w$, $\cH_{n,\leq k_1, k_2}^w$, $\cH_{n, k_1, \leq k_2}^w$ are then given by
\begin{align*}
   m_{n,k_1,k_2}^w(\xi,\eta) &:= \beta(2^{-k_1}\xi)\beta(2^{-k_2}\eta) m_{n}^w(\xi, \eta), \\
   m_{n,\leq k_1,k_2}^w(\xi,\eta) &:= \beta_0(2^{-k_1}\xi)\beta(2^{-k_2}\eta) m_{n}^w(\xi, \eta), \\
   m_{n, k_1,\leq k_2}^w(\xi,\eta) &:= \beta(2^{-k_1}\xi)\beta_0(2^{-k_2}\eta) m_{n}^w(\xi, \eta),
\end{align*}
respectively. 

We analyse the above multipliers using the method of stationary phase. This relies on understanding the critical points of $\phi_{\xi,\eta}^w$, which are the zeros of the function 
\begin{equation*}
    (\phi_{\xi,\eta}^w)'(t) = -\xi -2 t \eta + 3 t^2 w.
\end{equation*}
These are given by
\begin{equation}\label{240612e6_9uuu}
   t_{\pm}(w; \xi, \eta) := \frac{
    \eta\pm 
    \sqrt{\Delta(w; \xi, \eta)
    }
    }{3w} \quad \textrm{where} \quad \Delta(w; \xi, \eta) := \eta^2+3w\xi,
\end{equation}
provided $\Delta(w; \xi, \eta) \geq 0$. 
%    (\phi_{\xi,\eta}^w)''(t)&= -2 \eta + 6 t w, \\
%    (\phi_{\xi,\eta}^w)'''(t) &=  6  w
% and $(\phi_{\xi,\eta}^w)^{(k)}(t)=0$ for $k \geq 4$. 
We partition the range of $(k_1, k_2) \in \Z^2$ according to whether the critical points $t_\pm(w; \xi, \eta)$ exist for either $(\xi,\eta) \in \supp m_{n,k_1,k_2}^w$, $(\xi,\eta) \in \supp m_{n,\leq k_1,k_2}^w$ or $(\xi,\eta) \in \supp m_{n,k_1,\leq k_2}^w$, for some choice of $1 \leq w \leq 8$. To this end, define the sets of indices
\begin{align*}
    \cG_1 &:= \{(k_1, k_2) \in \Z^2 \cap (-C_1, \infty) \times (-C_2, \infty) : k_1 > (k_2 \vee 0) + 8 \} \\
    &= \{(k_1, k_2) \in \Z^2 : k_1 > 8 \textrm{ and } k_1 - 8 > k_2 > -C_2 \}
\end{align*}
and
\begin{align*}
 \cG_2 &:= \{(k_1, k_2) \in \Z^2 \cap (-C_1, \infty) \times (-C_2, \infty) : k_2 > (k_1 \vee 0) + 8 \}\\
 &= \{(k_1, k_2) \in \Z^2 : k_2 > 8 \textrm{ and } k_2 - 8 > k_1 > -C_1 \}.   
\end{align*}
We shall show below, in the proof of Lemma~\ref{lem:non-stationary multiplier}, that for $(k_1, k_2) \in \cG_1$ or $(k_1, k_2) \in \cG_2$ there are no critical points when $(\xi, \eta) \in \supp m_{n,k_1,\leq k_2}^w$ or $(\xi, \eta) \in \supp m_{n,\leq k_1,k_2}^w$, respectively. This leads to very favourable bounds for the corresponding multipliers: see Proposition~\ref{prop:L2 non-stationary}.  

It is easy to check that $\big(\Z^2 \cap (-C_1, \infty) \times (-C_2, \infty)\big) \setminus (\cG_1 \cup \cG_2) \subseteq \cB_{\diag}$ where
\begin{equation*}
   \cB_{\diag} := \{(k_1, k_2) \in \Z : |k_1 - k_2| \leq (C_1 \vee C_2) + 16 \text{ and } k_1>-C_1, k_2>-C_2-8 \}.
\end{equation*}
These `diagonal' exponents index pieces of the multiplier admitting critical points.

%%%%%%%%%%%%%%%%%%%%%%%%%%%%%%%%%%%%%%%%%%%%%%%%%%%%%%%%%%%%%%%%%%%%%%%%%%%%%%%%%%%%%%%%%%%%%%%%

%    Non-stationary case

%%%%%%%%%%%%%%%%%%%%%%%%%%%%%%%%%%%%%%%%%%%%%%%%%%%%%%%%%%%%%%%%%%%%%%%%%%%%%%%%%%%%%%%%%%%%%%%%

\subsection{Non-stationary case} The following proposition guarantees very strong estimates when the the exponents $(k_1, k_2)$ lie in $\cG_1$ or $\cG_2$.

\begin{proposition}[Non-stationary estimates]\label{prop:L2 non-stationary}
   Let $2 \leq p < \infty$ and $n \in \N$. 
   \begin{enumerate}[i)]
       \item If $(k_1, k_2) \in \cG_1$, then
        \begin{equation*}
        \|  \cH^w_{n,k_1, k_2} \|_{L^p(\R^2) \to L^p(\R^2)} +    \|  \cH^w_{n,k_1, \leq k_2} \|_{L^p(\R^2) \to L^p(\R^2)} \lesssim_{p,N} 2^{-k_1 N - n N}
        \end{equation*}
        holds for all $N \in \N_0$, uniformly in $1 \leq w \leq 8$.
        \item If $(k_1, k_2) \in \cG_2$, then
        \begin{equation*}
           \|  \cH^w_{n,k_1, k_2} \|_{L^p(\R^2) \to L^p(\R^2)} +   \| \cH^w_{n, \leq  k_1, k_2}\|_{L^p(\R^2) \to L^p(\R^2)} \lesssim_{p,N} 2^{-k_2 N - n N}
        \end{equation*}
        holds for all $N \in \N_0$, uniformly in $1 \leq w \leq 8$.
   \end{enumerate}
\end{proposition}

Proposition~\ref{prop:L2 non-stationary} is a consequence of elementary $L^2$-based arguments. We postpone the proof until \S\ref{sec:L2 theory}. 

%%%%%%%%%%%%%%%%%%%%%%%%%%%%%%%%%%%%%%%%%%%%%%%%%%%%%%%%%%%%%%%%%%%%%%%%%%%%%%%%%%%%%%%%%%%%%%%%

%    Stationary case

%%%%%%%%%%%%%%%%%%%%%%%%%%%%%%%%%%%%%%%%%%%%%%%%%%%%%%%%%%%%%%%%%%%%%%%%%%%%%%%%%%%%%%%%%%%%%%%%

\subsection{Stationary case} The analysis of the remaining pieces of the operator requires much deeper tools. The key estimates are as follows. 

\begin{proposition}[$L^p$ local smoothing]\label{prop:Lp smoothing low}
    For all $2 < p < 5/2$ there exists  $\varepsilon(p)>0$ such that for all $n \in \N$ the following hold.
    \begin{enumerate}[i)]
        \item If $-C_1 < k_1 \leq 8$, then
    \begin{equation*}
        \Big( \int_1^8 \| \cH^w_{n, k_1, \leq -C_2} f \|_{L^{p}(\R^2)}^{p} \ud w \Big)^{1/p} \lesssim_p 2^{-n/p- n \varepsilon(p)} \| f \|_{L^{p}(\R^2)}.
    \end{equation*}
        \item If $-C_2 < k_2 \leq 8$, then 
    \begin{equation*}
        \Big( \int_1^8 \| \cH^w_{n, \leq - C_1, k_2} f \|_{L^{p}(\R^2)}^{p} \ud w \Big)^{1/p} \lesssim_p 2^{-n/p- n \varepsilon(p)} \| f \|_{L^{p}(\R^2)}.
    \end{equation*}
    \end{enumerate}
\end{proposition}

We also have the following estimate for comparable frequencies. 

\begin{proposition}[$L^p$ local smoothing]\label{prop:Lp smoothing diag} For all $2 < p < 5/2$ there exists $\varepsilon(p)>0$ such that for all $n \in \N$ and $(k_1, k_2) \in \cB_{\diag}$ we have
    \begin{equation}\label{eq:Lp smoothing diag}
        \Big( \int_1^8 \| \cH_{n,k_1,k_2}^{w}  f \|_{L^{p}(\R^2)}^{p} \ud w \Big)^{1/p} \lesssim_p 2^{-n/p- n \varepsilon(p) - (k_1 + k_2)\varepsilon(p)} \| f \|_{L^{p}(\R^2)}.
    \end{equation}
\end{proposition}

Both Proposition~\ref{prop:Lp smoothing low} and Proposition~\ref{prop:Lp smoothing diag} rely on the local smoothing estimates from Theorem~\ref{thm: general propagator}. We postpone the proofs until \S\ref{sec:loc smoothing}.

%%%%%%%%%%%%%%%%%%%%%%%%%%%%%%%%%%%%%%%%%%%%%%%%%%%%%%%%%%%%%%%%%%%%%%%%%%%%%%%%%%%%%%%%%%%%%%%%

%    Proof of Theorem \ref{thm:key}

%%%%%%%%%%%%%%%%%%%%%%%%%%%%%%%%%%%%%%%%%%%%%%%%%%%%%%%%%%%%%%%%%%%%%%%%%%%%%%%%%%%%%%%%%%%%%%%%

\subsection{Proof of Theorem~\ref{thm:key}}

With this in hand, we can now prove the key estimate for the proof of Theorem \ref{thm:main}.

\begin{proof}[Proof of Theorem~\ref{thm:key} (assuming Proposition~\ref{prop:L2 non-stationary}, \ref{prop:Lp smoothing low} and \ref{prop:Lp smoothing diag})] We treat each of the three estimates separately, and further subdivide into cases depending on the values of $k_1$, $k_2$. In each case we highlight whether the \textbf{non-stationary} estimates from Proposition~\ref{prop:L2 non-stationary} or the \textbf{local smoothing} estimates from Proposition~\ref{prop:Lp smoothing low} or Proposition~\ref{prop:Lp smoothing diag} are used.\medskip

\noindent \underline{Proof of \eqref{eq:low freq eta}.} We consider two separate regimes:\medskip

\noindent \textit{Case 1: $[k_1 > 8]$.} \textbf{Non-stationary.} We apply Proposition~\ref{prop:L2 non-stationary} i) with $k_2 := -C_2$, noting $(k_1, k_2) \in \cG_1$. Integrating the resulting bounds over $1 \leq w \leq 8$, we obtain a strong form of \eqref{eq:low freq eta}, valid for all $2 \leq p < \infty$ with rapid decay in $2^{-n}$ and $2^{-k_1}$. \medskip

\noindent \textit{Case 2: $[-C_1 < k_1 \leq 8]$.} \textbf{Local smoothing.} We apply Proposition~\ref{prop:Lp smoothing low} i). Since $k_1$ is bounded, we do not require any decay in $2^{-k_1}$.\medskip

    \noindent \underline{Proof of \eqref{eq:low freq xi}.} We consider two separate regimes:\medskip
    
\noindent \textit{Case 1: $[k_2 > 8]$.} \textbf{Non-stationary.} We apply Proposition~\ref{prop:L2 non-stationary} ii) with $k_1 := -C_1$, noting $(k_1, k_2) \in \cG_2$. Integrating the resulting bounds over $1 \leq w \leq 8$, we obtain a strong form of \eqref{eq:low freq xi}, valid for all $2 \leq p < \infty$ with rapid decay in $2^{-n}$ and $2^{-k_2}$.\medskip
    
\noindent \textit{Case 2: $[-C_2 < k_2 \leq 8]$.} \textbf{Local smoothing.} We apply Proposition~\ref{prop:Lp smoothing low} ii). Since $k_2$ is bounded, we do not require any decay in $2^{-k_2}$. \medskip

    \noindent \underline{Proof of \eqref{eq:high freq}.} We consider three separate regimes: \medskip

\noindent \textit{Case 1a: $[(k_1, k_2) \in \cG_1]$.} \textbf{Non-stationary.} We apply Proposition~\ref{prop:L2 non-stationary} i), as in Case 1 of the proof of \eqref{eq:low freq eta}. Since $k_1 > k_2 + 8$, we can trade $2^{-k_1}$ decay for $2^{-k_2}$ decay in this regime.\medskip

\noindent \textit{Case 1b: $[(k_1, k_2) \in \cG_2]$.} \textbf{Non-stationary.} We apply Proposition~\ref{prop:Lp smoothing low} ii), as in Case 1 of the proof of \eqref{eq:low freq xi}. Since $k_2 > k_1 + 8$, we can trade $2^{-k_2}$ for $2^{-k_1}$ decay in this regime.\medskip

\noindent \textit{Case 2: $[(k_1, k_2) \in \cB_{\diag}]$.} \textbf{Local smoothing.} We apply Proposition~\ref{prop:Lp smoothing diag}. \medskip

\noindent This completes the proof of \eqref{eq:high freq}.
\end{proof}

It remains to prove Proposition~\ref{prop:L2 non-stationary}, \ref{prop:Lp smoothing low} and \ref{prop:Lp smoothing diag}.

%%%%%%%%%%%%%%%%%%%%%%%%%%%%%%%%%%%%%%%%%%%%%%%%%%%%%%%%%%%%%%%%%%%%%%%%%%%%%%%%%%%%%%%%%%%%%%%%

%    Further decomposition

%%%%%%%%%%%%%%%%%%%%%%%%%%%%%%%%%%%%%%%%%%%%%%%%%%%%%%%%%%%%%%%%%%%%%%%%%%%%%%%%%%%%%%%%%%%%%%%%

\subsection{Further decomposition}\label{subsec:further dec} Before turning to the main arguments, we carry out further reductions regarding Proposition~\ref{prop:Lp smoothing diag}. Here we consider two separate cases depending on the size of $n$. 

Let $\kappa > 0$ be a fixed parameter to be determined. \medskip

\noindent \textit{Case 1: $[k_1 > \kappa n]$}. In this case, Proposition~\ref{prop:Lp smoothing diag} is a consequence of the following two elementary lemmas. 

\begin{lemma}\label{lem:high diagonal L2}
    Let $2 \leq p < \infty$ and $n \in \N$. If $(k_1, k_2) \in \Z^2$ and $k_2 \geq 8$, then 
    \begin{equation*}
       \sup_{1 \leq w \leq 8} \| \cH^w_{n,k_1,k_2}\|_{L^p(\R^2) \to L^p(\R^2)} \lesssim_p 2^{-n/p - k_2/p}.
    \end{equation*} 
\end{lemma}

In the case $k_2 < 8$ we have the following surrogate.

\begin{lemma}\label{lem:high diagonal L2 surrogate}
    Let $2 \leq p \leq \infty$, $n \in \N$ and $(k_1, k_2) \in \Z^2$. Then 
    \begin{equation*}
    \sup_{1 \leq w \leq 8} \| \cH^w_{n,k_1,k_2}\|_{L^p(\R^2) \to L^p(\R^2)} \lesssim_p 1.
    \end{equation*}
\end{lemma}

Lemma~\ref{lem:high diagonal L2} is a consequence of elementary $L^2$ arguments, based on van der Corput's lemma, whilst Lemma~\ref{lem:high diagonal L2 surrogate} follows from crude kernel estimates. We postpone the proofs until \S\S\ref{sec:kernel}-\ref{sec:L2 theory}.\medskip

\noindent \textit{Case 2: $[k_1 \leq \kappa n]$}. Here we perform a further decomposition and follow a similar strategy to that of Pramanik--Seeger \cite{PS2007} in the context of maximal estimates for averaging operators over curves in $\R^3$: $L^2$-estimates when the Fourier decay of the multiplier is slowest (since this only happens in a localised region, one can get a gain in the $L^\infty$ estimates), and $L^p$-local smoothing estimates for $p>2$ when the decay is essentially $|\xi|^{-1/2}$ (this happens for a large portion of the frequency space, so no $L^\infty$ gain is available and $L^2$-methods fall short for summability).

We decompose
\begin{equation*}
    m_{n,k_1,k_2}^w= m_{n,k_1,k_2}^{w, \leq \kappa} + m_{n,k_1,k_2}^{w, > \kappa}
\end{equation*}
where
\begin{equation*}
    m_{n,k_1,k_2}^{w, \leq  \kappa} (\xi,\eta):= m_{n,k_1,k_2}^w (\xi,\eta)\beta_0(2^{-2k_2+2\floor{n\kappa}}\Delta(w;\xi,\eta));
\end{equation*}
here $\Delta(w;\xi,\eta)$ is as in \eqref{240612e6_9uuu}. 
Let $\cH_{n,k_1,k_2}^{w,\leq \kappa}$ and $\cH_{n,k_1,k_2}^{w,> \kappa}$ the associated multiplier operators.\medskip

\noindent \textit{Case 2a: $[k_1 \leq \kappa n;\,\, \cH_{n,k_1,k_2}^{w,\leq \kappa}]$}. For the operator $\cH_{n,k_1,k_2}^{w,\leq \kappa}$ we use $L^2$-methods to obtain the following smoothing estimate.

\begin{proposition}\label{prop:comparable L2 smoothing}
Let $2 \leq p < 5/2$ and $n \in \N$. If $(k_1, k_2) \in \cB_{\diag}$ satisfies $k_1 \leq \kappa n$, then
\begin{equation*}
    \Big( \int_1^8 \| \cH_{n,k_1,k_2}^{w, \leq \kappa} f \|_{L^p(\R^2)}^p \ud w \Big) \lesssim_p 2^{-n/p - \kappa n(5/p - 2)} \| f \|_{L^p(\R^2)}.
\end{equation*}
\end{proposition}

To prove Proposition~\ref{prop:comparable L2 smoothing}, we argue slightly differently depending on whether $k_2>8$ or $k_2 \leq 8$. In particular, we reduce matters to proving Lemma~\ref{lem:L2 smoothing} below. 

Suppose $(k_1, k_2) \in \cB_{\diag}$ with $k_2 \leq 8$. In this case, we further decompose
\begin{equation*}
    m_{n,k_1,k_2}^{w, \leq \kappa} = \sum_{ \ell = \floor{\kappa n}}^{\floor{n/3}} m_{n,k_1,k_2}^{w,\ell}
\end{equation*}
where
\begin{equation*}
    m_{n,k_1,k_2}^{w, \ell} (\xi,\eta):= \begin{cases}
    m_{n,k_1,k_2}^w(\xi,\eta) \beta(2^{-2k_2+2\ell} \Delta(w;\xi,\eta)) \qquad & \text{ if $ \floor{\kappa n} \leq  \ell < \floor{n/3}$} \\
    m_{n,k_1,k_2}^w(\xi,\eta) \beta_0(2^{-2k_2+2\ell} \Delta(w;\xi,\eta)) \qquad & \text{ if $ \ell = \floor{n/3}$}
    \end{cases}.
\end{equation*}

The rationale of this decomposition is that  $\Delta(w; \xi, \eta)$ quantifies the simultaneous vanishing of the first and second-order derivatives of the phase function $\phi^w_{\xi,\eta}$. 

Denoting by $\cH^{w, \ell}_{n,k_1,k_2}$ the associated multiplier operators, the main estimate for Case 2a reads as follows. 

\begin{lemma}\label{lem:L2 smoothing} Let $2 \leq p < \infty$, $n \in \N$ and $(k_1, k_2) \in \cB_{\diag}$.
\begin{enumerate}[i)]
    \item If $k_2 > 8$ and $k_1 \leq \kappa n$, then
        \begin{equation*}
            \Big( \int_1^8 \| \cH^{w, \leq \kappa}_{n,k_1,k_2} f \|_{L^p(\R^2)}^p \ud w \Big)^{1/p} \lesssim_p 2^{-n/p- \kappa n(5/p - 2) - k_2/p} \| f \|_{L^p(\R^2)}.
        \end{equation*}
    \item If $k_2 \leq 8$ and $\floor{\kappa n} \leq \ell \leq \floor{n/3}$, then
    \begin{equation}\label{eq:L2 ell dec norm}
        \Big(\int_1^8 \|\cH_{n,k_1,k_2}^{w,\ell} f\|_{L^p(\R^2)}^p \ud w \Big)^{1/p} \lesssim_p 2^{-n/2 - \ell(5/p - 2)} \|  f\|_{L^p(\R^2)}.
    \end{equation}
\end{enumerate}
\end{lemma}

This is essentially a refined version of Proposition~\ref{prop:comparable L2 smoothing}.

\begin{proof}[Proof of Proposition~\ref{prop:comparable L2 smoothing} (assuming Lemma~\ref{lem:L2 smoothing})] This is immediate, summing the estimate \eqref{eq:L2 ell dec norm} in $\ell$ for the $k_2 \leq 8$ case.
\end{proof}

Lemma~\ref{lem:L2 smoothing} is a consequence of elementary $L^2$ arguments, based on van der Corput's lemma. Here we exploit the fact that the additional cutoff localises to a small $w$-interval. We postpone the proof until \S\ref{sec:L2 theory}.\medskip

\noindent \textit{Case 2b: $[k_1 \leq \kappa n;\,\, \cH_{n,k_1,k_2}^{w,> \kappa}]$.} The analysis of the operator $\cH_{n,k_1,k_2}^{w,> \kappa}$ is more subtle, since $L^2$-methods alone do not suffice to prove a satisfactory bound. The main estimate reads as follows. 

\begin{proposition}\label{prop:comparable Lp smoothing}  There exists some $M \in \N$ such that the following holds. For all $2< p <4$ there exists $\varepsilon (p) > 0$ such that if $n \in \N$, $0 < \kappa < 1$ and $(k_1, k_2) \in \cB_{\diag}$ satisfy $k_1 \leq \kappa n$, then
    \begin{equation}\label{eq: comparable Lp smoothing}
        \Big( \int_1^8 \| \cH^{w, > \kappa}_{n,k_1,k_2} f \|_{L^p(\R^2)}^p \ud w \Big)^{1/p} \lesssim_p 2^{M\kappa n} 2^{-n/p- n \varepsilon(p)} \| f \|_{L^p(\R^2)}.
    \end{equation}
\end{proposition}

Proposition~\ref{prop:comparable Lp smoothing} is a consequence of the local smoothing estimate from Theorem~\ref{thm: general propagator}. We postpone the proof until \S\ref{sec:loc smoothing}.

The preceding results combine to give Proposition~\ref{prop:Lp smoothing diag}.

\begin{proof}[Proof of Proposition~\ref{prop:Lp smoothing diag} (assuming Lemma~\ref{lem:high diagonal L2} and Proposition~\ref{prop:comparable L2 smoothing} and \ref{prop:comparable Lp smoothing})] Let $n \in \N$ and $(k_1, k_2) \in \cB_{\diag}$. Following the above scheme, we subdivide the argument into cases depending on the values of $k_1,k_2$ and $n$. In each case we highlight whether the \textbf{van der Corput} estimates from Lemma~\ref{lem:high diagonal L2} or Proposition~\ref{prop:comparable L2 smoothing} or the \textbf{local smoothing} estimates from Proposition~\ref{prop:comparable Lp smoothing} are used.\medskip

Fix any $2<p<5/2$ and let $M$ and $\varepsilon(p)>0$ be the parameters given by Proposition \ref{prop:comparable Lp smoothing}. Set $\kappa\equiv \kappa(p):=\varepsilon(p)/2M>0$. \medskip

\noindent \textit{Case 1a: $[k_1 > \kappa n]$}. \textbf{van der Corput.} If $k_2 \geq 8$, then we apply Lemma~\ref{lem:high diagonal L2} and integrate the resulting uniform bounds over $1 \leq w \leq 8$. Since $(k_1, k_2) \in \cB_{\diag}$, we can trade $2^{-k_2}$ decay for $2^{-k_1}$ decay. Furthermore, since $k_1 > \kappa n$, we can trade $2^{-k_1}$ decay for $2^{-n}$ decay. Thus, provided this trade is suitably executed, we can obtain an inequality of the form \eqref{eq:Lp smoothing diag} for any choice of $2 \leq p < \infty$.

On the  other hand, if $k_2 < 8$, then the conditions $(k_1, k_2) \in \cB_{\diag}$ and $\kappa n < k_1$ mean that all three parameters $k_1$, $k_2$ and $n$ are bounded. Thus, we do not require any decay in $2^{-k_1}$, $2^{-k_2}$ or $2^{-n}$, so the estimate from Lemma~\ref{lem:high diagonal L2 surrogate} suffices. \medskip

\noindent \textit{Case 2: $[k_1 \leq \kappa n]$.} It suffices to prove bounds of the form \eqref{eq:Lp smoothing diag} with $\cH_{n,k_1,k_2}^{w}$ replaced by the each of the localised operators $\cH_{n,k_1,k_2}^{w, \leq \kappa}$ and $\cH_{n,k_1,k_2}^{w, > \kappa}$. Furthermore, since $k_1 \leq \kappa n$ and $(k_1, k_2) \in \cB_{\diag}$, we can trade $2^{-n}$ decay for either $2^{-k_1}$ or $2^{-k_2}$  decay. In particular, it suffices to prove estimates which exhibit only appropriate $2^{-n}$ decay. 
\begin{itemize}
    \item \textbf{van der Corput.} For the operator $\cH_{n,k_1,k_2}^{w, \leq \kappa}$, the desired bound \eqref{eq:Lp smoothing diag} follows from Proposition~\ref{prop:comparable L2 smoothing}. % with, for instance, $\varepsilon(p):=\kappa (5/p - 2)/3 >0$. 
    \item \textbf{Local smoothing.} For the operator $\cH_{n,k_1,k_2}^{w, > \kappa}$, the desired bound  \eqref{eq:Lp smoothing diag} follows from Proposition~\ref{prop:comparable Lp smoothing}, since our choice of $\kappa$ guaranteed that \eqref{eq: comparable Lp smoothing} holds with constant $2^{-n/p - n \varepsilon (p)/2}$.
\end{itemize}
This concludes the proof of \eqref{eq:Lp smoothing diag} for $\cH^{w}_{n,k_1,k_2}$ for the specific choice of $p$ and therefore for all $2 < p <5/2$.
\end{proof}

%%%%%%%%%%%%%%%%%%%%%%%%%%%%%%%%%%%%%%%%%%%%%%%%%%%%%%%%%%%%%%%%%%%%%%%%%%%%%%%%%%%%%%%%%%%%%%%%

%    Recap

%%%%%%%%%%%%%%%%%%%%%%%%%%%%%%%%%%%%%%%%%%%%%%%%%%%%%%%%%%%%%%%%%%%%%%%%%%%%%%%%%%%%%%%%%%%%%%%%

\subsection{Recap} The proof of Theorem~\ref{thm:main} is now reduced to showing the following:

\begin{itemize}
    \item Lemma~\ref{lem:low}. This follows by pointwise domination by the strong maximal function bounds. The details are presented in \S\ref{subsec:low}.
    \item Lemma~\ref{lem:high diagonal L2 surrogate}. This follows from an elementary kernel estimate  in \S\ref{sec:kernel}.
    \item Proposition~\ref{prop:L2 non-stationary}, Lemma~\ref{lem:high diagonal L2} and Lemma~\ref{lem:L2 smoothing}. These results are all based on elementary $L^2$-based theory and are treated in \S\ref{sec:L2 theory}.
     \item Proposition~\ref{prop:Lp smoothing low} and Proposition~\ref{prop:comparable Lp smoothing}. These results are deeper, relying on the local smoothing inequality Theorem~\ref{thm: general propagator}, and are treated in \S\ref{sec:loc smoothing}.
\end{itemize}

%%%%%%%%%%%%%%%%%%%%%%%%%%%%%%%%%%%%%%%%%%%%%%%%%%%%%%%%%%%%%%%%%%%%%%%%%%%%%%%%%%%%%%%%%%%%%%%%

%    Kernel estimates

%%%%%%%%%%%%%%%%%%%%%%%%%%%%%%%%%%%%%%%%%%%%%%%%%%%%%%%%%%%%%%%%%%%%%%%%%%%%%%%%%%%%%%%%%%%%%%%%

\section{Kernel estimates}\label{sec:kernel} In this section we apply elementary integration-by-parts arguments to bound the kernels of our localised operators. This will imply $L^\infty$ estimate of the corresponding operators which will be interpolated with the $L^2$-theory of \S\ref{sec:L2 theory}.

Furthermore, we also establish a kernel estimate for the low frequency piece of our operator, which implies Lemma \ref{lem:low}.

\subsection{The operators $\cH_{n,k_1,k_2}$} 
We begin with the operators $\cH_{n,k_1,k_2}^w$ introduced in \S\ref{subsec:freq loc} (together with their low frequency variants), whose kernels are given by the inverse Fourier transforms of the multipliers $m_{n,k_1,k_2}^w$ from \S\ref{subsec:critical points}. 

\begin{lemma}\label{lem:kernel}
Let $n \in \N$, $1 \leq w \leq 8$ and $k_1, k_2 \in \Z$. Then
\begin{equation*}
    \| \mathcal{F}^{-1} [m_{n, k_1, k_2} ^w ] \|_{L^1(\R^2)} \lesssim 1,
\end{equation*}
and the same holds if $m_{n,k_1,k_2}^w$ is replaced by $m_{n,\leq k_1,k_2}^w$, $m_{n,k_1,\leq k_2}^w$.
\end{lemma}

\begin{proof} We have
    \begin{align*}
        |\mathcal{F}^{-1} [m_{n,k_1,k_2}^w](x,y)| & \sim \Big|\int_{\R}\int_{\widehat{\R}^2} e^{i (x \xi + y \eta + 2^n \phi_{\xi, \eta}^w(t))} \beta(2^{-k_1}\xi) \beta(2^{-k_2}\eta) \ud \xi \ud \eta \,a(t) \ud t\Big| \\
        &  \lesssim_N \int_{\R} \frac{2^{k_1} 2^{k_2}}{(1+2^{k_1} |x-2^n t| + 2^{k_2}|y-2^nt^2|)^N} \, a(t) \ud t
    \end{align*}
    by repeated integration-by-parts. Integrating in $(x,y)$ and using Fubini's theorem, we obtain the result.
\end{proof}

As a consequence of Lemma~\ref{lem:kernel}, we have the basic $L^p$ estimate from Lemma~\ref{lem:high diagonal L2 surrogate}.

\begin{proof}[Proof of Lemma~\ref{lem:high diagonal L2 surrogate}] Since, for all $f \in \cS(\R^2)$, we have
\begin{equation*}
    \cH^w_{n,k_1,k_2}f =  K_{n,k_1,k_2}^w \ast f \qquad \textrm{where} \qquad K_{n,k_1,k_2}^w := \mathcal{F}^{-1} [m_{n,k_1,k_2}^w],
\end{equation*}
 the desired bound follows from Lemma~\ref{lem:kernel} and Young's inequality.
\end{proof}

\subsection{The operators $\cH_{n,k_1,k_2}^{2, \leq \kappa}$ and $\cH_{n,k_1,k_2}^{w, \ell}$}

We now turn to the kernels of the operators $\cH_{n,k_1,k_2}^{w,\leq \kappa}$ and $\cH_{n,k_1,k_2}^{w,\ell}$ introduced in \S\ref{subsec:further dec}.

\begin{lemma}\label{lem:kernel est ell} Let $n \in \N$, $1 \leq w \leq 8$ and $(k_1, k_2) \in \cB_{\diag}$.
    \begin{enumerate}[i)]
        \item If $k_2 > 8$ and $k_1 \leq \kappa n$, then 
        \begin{equation*}
            \|\mathcal{F}^{-1}[m_{n,k_1,k_2}^{w,\leq \kappa}]\|_{L^1(\R^2)} \lesssim 2^{2\floor{\kappa n}}.
        \end{equation*}
        %\begin{equation*}
        %    |\mathcal{F}^{-1}[m_{n,k_1,k_2}^{w,\leq \epsilon}](x,y)| \lesssim_N \int_{1/2}^2 \frac{2^{2k_2-2\floor{\epsilon n}}}{(1+2^{2k_2-2\floor{\epsilon n}}|x-2^n t| + 2^{-2\floor{\epsilon n}}|y-2^n t^2|)^N} \ud t
        %\end{equation*}
              
        \item If $k_2 \leq 8$, then for all $\floor{\kappa n} \leq \ell \leq \floor{n/3}$, we have
        \begin{equation*}
            \|\mathcal{F}^{-1}[m_{n,k_1,k_2}^{w,\ell}]\|_{L^1(\R^2)} \lesssim 2^{2\ell}.
        \end{equation*}
    \end{enumerate}
\end{lemma}

\begin{proof} The proof of i) follows from a minor adaptation of that of ii), and so we only provide the latter. 

It suffices to show the pointwise bound
\begin{equation*}
            |\mathcal{F}^{-1}[m_{n,k_1,k_2}^{w,\ell}](x,y)| \lesssim_N \int_{1/2}^2 \frac{2^{-2\ell}}{(1+2^{-2\ell}|x-2^n t| + 2^{-2\ell}|y-2^n t^2|)^N} \ud t
        \end{equation*}
holds for all $N \in \N_0$, uniformly in $1 \leq w \leq 8$. Write $m_{n,k_1,k_2}^{w,\ell} = a_{k_1,k_2}^{w,\ell} \cdot m_n^w$ where
    \begin{equation*}
        a_{k_1,k_2}^{w,\ell}(\xi,\eta):= \beta(2^{-2k_2+2\ell} \Delta (w;\xi,\eta)) \beta(2^{-k_1} \xi) \beta(2^{-k_2} \eta)
    \end{equation*}
    and $m_n^w$ is as defined in \eqref{eq: multiplier 1}. Then 
    \begin{align*}
        \mathcal{F}^{-1}[m_{n,k_1,k_2}^{w,\ell}](x,y) = \int_\R e^{i 2^n w t^3} \mathcal{F}^{-1}[a_{k_1,k_2}^{w,\ell}](x-2^n t, y-2^n t^2) a(t) \ud t.
    \end{align*}
    
    In light of the above, it suffices to show that
    \begin{equation*}
        |\mathcal{F}^{-1}[a_{k_1,k_2}^{w,\ell}](x,y)| \lesssim_N \frac{2^{-2\ell}}{(1+2^{-2\ell}|x| + 2^{-2\ell}|y|)^N}
    \end{equation*}
    for all $N \in \N_0$. However, this follows from routine integration-by-parts and the fact that, for fixed $\eta$, the $\xi$-support of $a_{k_1,k_2}^{w,\ell}$ lies in a fixed interval of length $O(2^{-2\ell})$, in view of the support condition $|\Delta(w;\xi,\eta)|\lesssim 2^{-2\ell}$ and that $\Delta(w;\xi,\eta)=\eta^2+3w \xi$.
\end{proof}

%%%%%%%%%%%%%%%%%%%%%%%%%%%%%%%%%%%%%%%%%%%%%%%%%%%%%%%%%%%%%%%%%%%%%%%%%%%%%%%%%%%%%%%%%%%%%%%%

%    Proof of Lemma~\ref{lem:low}

%%%%%%%%%%%%%%%%%%%%%%%%%%%%%%%%%%%%%%%%%%%%%%%%%%%%%%%%%%%%%%%%%%%%%%%%%%%%%%%%%%%%%%%%%%%%%%%%

\subsection{Proof of Lemma~\ref{lem:low}}\label{subsec:low}

Here we provide the proof of Lemma~\ref{lem:low}. The key ingredient is the following kernel estimate, which shares similarities with the above kernel estimates.

\begin{lemma}\label{lem:low kernel} For $\ell \in \N$ and $j \in \Z$ and $v \in V_j$, let $K_{\ell,j}^v$ denote the kernel of the operator $H_{-j+\ell}^v P_{\leq -C_1  + 2\ell + j}^{(1)} P_{\leq -C_2 + \ell + 2j}^{(2)}$. Then
\begin{equation*}
 |K_{\ell,j}^v(x,y)| \lesssim_N 2^{-\ell N}  2^{3(j+\ell)} (1+2^{2 \ell + j} |x| + 2^{\ell + 2j}|y|)^{-N} \qquad \textrm{for all $N \in \N_0$.}
\end{equation*}
\end{lemma}

\begin{proof} By the rescaling argument from \S\ref{subsec:freq loc}, it follows that
\begin{equation*}
   K_{\ell,j}^v(x,y) = 2^{3(j+\ell)} \frac{1}{(2\pi)^2} \int_{\widehat{\R}^2} e^{i (2^{2\ell + j}x \xi + 2^{\ell + 2j}y \eta)} \fm_{3\ell}^{2^{-3j}v}(\xi,\eta) \ud \xi \ud \eta
\end{equation*}
where, using the definitions from \eqref{eq: multiplier 2}, we have
\begin{equation*}
    \fm_n^w(\xi,\eta) :=  \beta_0(2^{C_1}\xi) \beta_0(2^{C_2}\eta) \int_{\R} e^{ i2^n\phi_{\xi,\eta}^w(t)} a(t) \ud t.
\end{equation*}
Since $v \in V_j = [2^{3j}, 2^{3(j+1)}]$, repeated integration-by-parts shows that
\begin{equation}\label{eq:low 1}
    | K_{\ell,j}^v(x,y)| \lesssim_N \sup_{1 \leq w \leq 8} \|\fm_{3\ell}^w\|_{C^N(\widehat{\R}^2)} 2^{3(j+\ell)} \big(1 + 2^{2\ell + j}|x| + 2^{\ell + 2j}|y| \big)^{-N}
\end{equation}
for all $N \in \N$. Differentiating $\fm_n^w$, the problem reduces to showing
\begin{equation}\label{eq:low 2}
   \sup_{1 \leq w \leq 8} \sup_{\substack{|\xi| \leq 2^{-C_1 +1} \\ |\eta| \leq 2^{-C_2+1}}} \Big| \int_{\R} e^{i2^n \phi_{\xi,\eta}^w(t)} t^m a(t) \ud t \Big| \lesssim_{m, M} 2^{-nM} \qquad \textrm{for all $m$, $M \in \N_0$.} 
\end{equation}
Indeed, the desired bounds then follow by taking $M := 2N$, so the gain in $2^{-2Nn}$ from \eqref{eq:low 2} compensates for the loss of $2^{Nn}$ in \eqref{eq:low 1}. 

Let $1 \leq w \leq 8$, $(\xi, \eta) \in \widehat{\R}^2$ satisfy $|\xi| \leq 2^{-C_1 +1}$ and $|\eta| \leq 2^{-C_2+1}$ and $t \in \supp a$, so that $1/2 \leq |t| \leq 2$. Since $(\phi_{\xi,\eta}^w)'(t) = - \xi - 2\eta t + 3w t^2$, we have
\begin{equation*}
    |(\phi_{\xi,\eta}^w)'(t)| \geq 3/4 - |\xi| - 4|\eta| \geq 1/8,
\end{equation*}
provided $C_1$, $C_2 \geq 4$. Moreover, $|(\phi_{\xi,\eta}^w)''(t)|\leq 2|\eta| + 6 |t| w \leq 2^6$ and $|(\phi_{\xi,\eta}^w)'''(t)|\leq 6 w \leq 2^6$. The desired estimate \eqref{eq:low 2} now follows from an application of the non-stationary phase Lemma \ref{non-stationary lem} below.
\end{proof}

Lemma~\ref{lem:low} easily follows from the above kernel estimate. 

\begin{proof}[Proof of Lemma~\ref{lem:low}] Using a standard argument (see, for instance, \cite[Chapter 2, \S2.1]{Stein1993}), Lemma~\ref{lem:low kernel} implies the pointwise bound
\begin{equation*}
        \sup_{j \in \Z} \sup_{2^{3j} \leq v \leq 2^{3(j+1)}} |H_{-j+\ell}^v P_{\leq -C_1 + 2\ell + j}^{(1)} P_{\leq -C_2 + \ell + 2j }^{(2)} f (x,y)| \lesssim_N 2^{-N \ell} M_{\textrm{st}} f(x,y),
    \end{equation*}
    where $M_{\textrm{st}}$ denotes the strong Hardy--Littlewood maximal function. Thus, the desired bound follows from the $L^p$ boundedness of $M_{\mathrm{st}}$. 
\end{proof}

%%%%%%%%%%%%%%%%%%%%%%%%%%%%%%%%%%%%%%%%%%%%%%%%%%%%%%%%%%%%%%%%%%%%%%%%%%%%%%%%%%%%%%%%%%%%%%%%

%    $L^2$ Theory

%%%%%%%%%%%%%%%%%%%%%%%%%%%%%%%%%%%%%%%%%%%%%%%%%%%%%%%%%%%%%%%%%%%%%%%%%%%%%%%%%%%%%%%%%%%%%%%%

\section{$L^2$ Theory}\label{sec:L2 theory}

Here we prove Proposition~\ref{prop:L2 non-stationary}, Lemma~\ref{lem:high diagonal L2} and Lemma~\ref{lem:L2 smoothing}, using elementary $L^2$-based theory. Throughout this section, $\phi_{\xi, \eta}^w(t) := -t\xi-t^2\eta+t^3w$ and $a$ are as defined in \eqref{eq: multiplier 2}; it is useful to note the formul\ae
\begin{equation}\label{eq: derivatives}
    (\phi_{\xi,\eta}^w)'(t) = -\xi -2 t \eta + 3 t^2 w, \quad (\phi_{\xi,\eta}^w)''(t)= -2 \eta + 6 t w, \quad (\phi_{\xi,\eta}^w)'''(t) =  6  w 
\end{equation}
 and that $(\phi_{\xi,\eta}^w)^{(k)}(t)=0$ for $k \geq 4$.

%%%%%%%%%%%%%%%%%%%%%%%%%%%%%%%%%%%%%%%%%%%%%%%%%%%%%%%%%%%%%%%%%%%%%%%%%%%%%%%%%%%%%%%%%%%%%%%%

%    The non-stationary case

%%%%%%%%%%%%%%%%%%%%%%%%%%%%%%%%%%%%%%%%%%%%%%%%%%%%%%%%%%%%%%%%%%%%%%%%%%%%%%%%%%%%%%%%%%%%%%%%

 \subsection{The non-stationary case}%\label{subsec:non-stationary} 
 The proof of Proposition~\ref{prop:L2 non-stationary} makes repeated use of the following elementary oscillatory integral estimate. 

 \begin{lemma}[Non-stationary phase]\label{non-stationary lem} Let $N \in \N$ and $\phi \in C^N(-1,1)$ be real valued and $a \in C^N_c(-1,1)$ satisfy $\|a\|_{C^N}  \leq 1$. Further suppose
 \begin{enumerate}[i)]
     \item $|\phi'(t)| \geq 1$ for all $t \in (-1,1)$;
     \item $|\phi^{(k)}(t)| \leq M$ for all $1 \leq k \leq N$ and $t \in (-1,1)$.
 \end{enumerate}
 For all $\lambda \geq 1$, we have
 \begin{equation*}
     \Big|\int_{\R} e^{i \lambda \phi(t)}a(t)\ud t\Big| \lesssim_{N,M} \lambda^{-N}.
 \end{equation*}
Here the implied constant depends on $N$ and $M$, but is otherwise independent of the choice of $\phi$ and $a$. 
 \end{lemma}

Lemma~\ref{non-stationary lem} is a simple and well-known consequence of repeated integration-by-parts. We omit the details. 

\begin{lemma}\label{lem:non-stationary multiplier} Let $n \in \N$, $k_1$, $k_2 \in \Z$ and $(\xi,\eta) \in \widehat{\R}^2$,  $1 \leq w \leq 8$.
\begin{enumerate}[i)]
    \item If $2^{(k_2 \vee 0) + 8} \leq |\xi| $ and $|\eta| \leq 2^{k_2+1}$, then
    \begin{equation}\label{eq:mult decay 1}
        |m_{n}^w (\xi, \eta)| \lesssim_N |\xi|^{-N} 2^{ -n N}.
    \end{equation}
    \item If $|\xi| \leq 2^{k_1 +1}$ and  $|\eta| \geq 2^{(k_1 \vee 0) + 8}$, then
    \begin{equation}\label{eq:mult decay 2}
        |m_{n}^w (\xi, \eta)| \lesssim_N |\eta|^{-N} 2^{ -n N}.
    \end{equation}
\end{enumerate}
\end{lemma}

\begin{proof} Each estimate is proved using repeated integration-by-parts (Lemma~\ref{non-stationary lem}). We make repeated use of the fact that $1/2 \leq |t| \leq 2$ on the support of $a$.\medskip

\noindent i) If $1/2 \leq t \leq 2$, then $|-2t\eta+3t^2 w| \leq 2^{k_2+3}+3\cdot 2^5 \leq 2^{(k_2 \vee 0) + 7} \leq  |\xi|/2$. Thus, by \eqref{eq: derivatives} and the triangle inequality,
    \begin{equation*}
        |(\phi_{\xi,\eta}^w)'(t)| \geq  |\xi|/2  \qquad \textrm{for all $1/2 \leq t \leq 2$.}
    \end{equation*}
    A similar calculation shows that $|(\phi_{\xi,\eta}^w)''(t)| \leq |\xi|$ and $|(\phi_{\xi,\eta}^w)'''(t)| \leq |\xi|$. Since all higher order derivatives vanish, the desired bound \eqref{eq:mult decay 1} follows from an application of Lemma~\ref{non-stationary lem}, at least after applying suitable changes of variables and scalings to massage the setup to match the hypotheses of the lemma.\medskip

\noindent ii) If $1/2 \leq t \leq 2$, then $|- \xi + 3t^2 w| \leq 2^{k_1+1} + 3\cdot 2^5 \leq 2^{(k_1 \vee 0) + 7} \leq |\eta|$. Thus, by \eqref{eq: derivatives} and the triangle inequality,
    \begin{equation*}
        |(\phi_{\xi,\eta}^w)'(t)| \geq |\eta| \qquad \textrm{for all $1/2 \leq t \leq 2$.}
    \end{equation*}
    A similar calculation shows that $|(\phi_{\xi,\eta}^w)''(t)| \leq 3|\eta|$ and $|(\phi_{\xi,\eta}^w)'''(t)| \leq |\eta|$. Thus, as before, the desired bound \eqref{eq:mult decay 2} follows from an application of Lemma~\ref{non-stationary lem}.
\end{proof}

\begin{proof}[Proof of Proposition \ref{prop:L2 non-stationary}] In each case it suffices to prove the $L^2$ bound only. The full range of $L^p$ then follows via interpolation with the $L^{\infty}$ bound from Lemma~\ref{lem:kernel}.\medskip

\noindent \textit{i)} If $(k_1, k_2) \in \cG_1$, then $k_1 > (k_2 \vee 0) + 8$ and so $2^{(k_2 \vee 0) + 8} \leq 2^{k_1-1} \leq |\xi|$ and $|\eta| \leq 2^{k_2 + 1}$ for $(\xi,\eta) \in \supp m_{n,k_1, k_2}^w$ or $(\xi,\eta) \in \supp m_{n,k_1, \leq k_2}^w$. The desired $L^2$ bound now follows from Lemma~\ref{lem:non-stationary multiplier} i) and Plancherel's theorem. \medskip

\noindent \textit{ii)} If $(k_1, k_2) \in \cG_2$, then $k_2 > (k_1 \vee 0) + 8$ and so $2^{(k_1 \vee 0) + 8} \leq 2^{k_2-1} \leq |\eta|$ and $|\xi| \leq 2^{k_1 + 1}$ for $(\xi,\eta) \in \supp m_{n,k_1, k_2}^w$ or $(\xi,\eta) \in \supp m_{n,\leq k_1, k_2}^w$. The desired $L^2$ bound now follows from Lemma~\ref{lem:non-stationary multiplier} ii) and Plancherel's theorem. \medskip
\end{proof}

%%%%%%%%%%%%%%%%%%%%%%%%%%%%%%%%%%%%%%%%%%%%%%%%%%%%%%%%%%%%%%%%%%%%%%%%%%%%%%%%%%%%%%%%%%%%%%%%

%    van der Corput estimates

%%%%%%%%%%%%%%%%%%%%%%%%%%%%%%%%%%%%%%%%%%%%%%%%%%%%%%%%%%%%%%%%%%%%%%%%%%%%%%%%%%%%%%%%%%%%%%%%

\subsection{van der Corput estimates} We now turn to the proofs of Lemma~\ref{lem:high diagonal L2} and Lemma~\ref{lem:L2 smoothing}. We first observe a basic pointwise bound for the underlying multiplier, based on the van der Corput oscillatory integral estimate.

\begin{lemma}\label{lem:vdc multiplier 1} Let $n \in \N$, $(\xi, \eta) \in \widehat{\R}^2$ and $1 \leq w \leq 8$. If $|\eta| \geq 2^7$, then
    \begin{equation*}
        |m_n^w (\xi, \eta)| \lesssim 2^{ -n/2}|\eta|^{-1/2}.
    \end{equation*}
\end{lemma}

\begin{proof} By \eqref{eq: derivatives}, we have $|(\phi_{\xi,\eta}^w)''(t)| \geq 2|\eta| - 6|t|w \geq |\eta|$ for $|\eta|  \geq 3 \cdot 2^5 \geq 6 |t|w$. Thus, the desired bound follows from van der Corput's lemma. 
\end{proof}

Lemma~\ref{lem:vdc multiplier 1} and Lemma~\ref{lem:kernel} easily combine to give Lemma~\ref{lem:high diagonal L2}.

\begin{proof}[Proof of Lemma~\ref{lem:high diagonal L2}] For $k_2 > 8$ we have $|\eta| \geq 2^{k_2 - 1} > 2^7$ on the support of $m_{n,k_1,k_2}^w$. Thus, we can apply Lemma~\ref{lem:vdc multiplier 1} and Plancherel's theorem to deduce the desired bound for $p = 2$. The $p = \infty$ case again follows from Lemma~\ref{lem:kernel}, and we conclude the proof by interpolating between these two endpoints.
\end{proof}

It remains to prove Lemma~\ref{lem:L2 smoothing}. The key ingredient is the following multiplier estimate, which is again based on van der Corput's lemma. 

\begin{lemma}\label{lem:vdc multiplier 2} Let $n \in \N$, $(k_1, k_2) \in \cB_{\diag}$ and $(\xi, \eta) \in \widehat{\R}^2$ and $1 \leq w \leq 8$.
\begin{enumerate}[i)]
    \item If $k_2 > 8$ and $k_1 \leq \kappa n$, then
    \begin{equation*}
        \Big(\int_1^8 |m_{n,k_1,k_2}^{w,\leq \kappa} (\xi, \eta)|^2 \ud w \Big)^{1/2} \lesssim 2^{-n/2-\kappa n/2 - k_2/2 }.
    \end{equation*}
    \item If $k_2 \leq 8$ and $\floor{\kappa n} \leq \ell \leq \floor{n/3}$, then
    \begin{equation*}
        \Big(\int_1^8 |m_{n,k_1,k_2}^{w,\ell} (\xi, \eta)|^2 \ud w \Big)^{1/2} \lesssim 2^{-(n + \ell)/2}.
    \end{equation*}
\end{enumerate}
\end{lemma}

\begin{proof} For both i) and ii) we shall exploit the fact that the $\Delta(w;\xi,\eta)$ localisation corresponds to localising the $w$ variable to a small interval. In particular, 
\begin{equation}\label{eq: w loc}
   \Big|w+\frac{\eta^2}{3\xi}\Big| = \frac{1}{3|\xi|}|\Delta(w;\xi,\eta)| \leq 2^{-k_1} |\Delta(w;\xi,\eta)| 
\end{equation}
for  $2^{k_1 - 1} \leq |\xi| \leq 2^{k_1 + 1}$. \medskip
 
 \noindent \textit{i)} On the support of $m_{n,k_1,k_2}^{w,\leq \kappa}$ we have $|\Delta(w;\xi,\eta)|\leq 2^{2k_2-2\kappa n +1}$. This combines with \eqref{eq: w loc} to give
    \begin{equation*}
        \Big|w+\frac{\eta^2}{3\xi}\Big| \leq 2^{2k_2-k_1-2\kappa n + 1}.
    \end{equation*}
     Since $(k_1, k_2) \in \cB_{\diag}$ and $k_1 \leq \kappa n$, it follows that $k_2 \leq k_1 + C \leq \kappa n + C$ for $C := (C_1 \vee C_2) + 16$. Thus, for fixed $(\xi, \eta) \in \widehat{\R}^2$, we conclude that $m_{n,k_1,k_2}^{w,\leq \kappa}(\xi,\eta)$ is non-zero only if $w$ belongs to a fixed interval of length $O(2^{-\kappa n})$ centred at $-\eta^2/(3\xi)$. 
     
     On the other hand, since $k_2>8$ we have $|\eta| \geq 2^{k_2-1} > 2^7$ and Lemma~\ref{lem:vdc multiplier 1} implies that
    \begin{equation*}
        |m_{n,k_1,k_2}^{w,\leq \kappa}(\xi,\eta)| \lesssim 2^{-n/2-k_2/2}
    \end{equation*}
    holds uniformly in $1 \leq w \leq 8$. Combining the above observations, the desired result immediately follows.\medskip

\noindent \textit{ii)}  On the support of $m_{n,k_1,k_2}^{w,\ell}$ we have $|\Delta(w; \xi,\eta)|\leq 2^{2k_2-2\ell + 1}$. This combines with \eqref{eq: w loc} to give
    \begin{equation*}
        \Big|w+\frac{\eta^2}{3\xi}\Big| \leq 2^{2k_2-k_1-2\ell + 1}.
    \end{equation*}
    Since $k_2 \leq 8$, it follows that $2k_2 - k_1 \leq 32 + C_1$. Thus, for fixed $(\xi, \eta) \in \widehat{\R}^2$, we conclude that $m_{n,k_1,k_2}^{w,\ell}(\xi,\eta)$ is non-zero only if $w$ belongs to a fixed interval of length $O(2^{-2 \ell})$ centred at $-\eta^2/(3\xi)$. 
    
    On the other hand, we claim that
    \begin{equation}\label{eq:mult dec ell}
        |m_{n,k_1,k_2}^{w,\ell}(\xi,\eta)| \lesssim 2^{-(n-\ell)/2}
    \end{equation}
    holds uniformly in $1 \leq w \leq 8$. Once, this is established and combined with our earlier observation, the desired result immediately follows. 

    The case $\ell= \floor{n/3}$ of \eqref{eq:mult dec ell} follows from van der Corput's lemma with third order derivatives, since $|(\phi_{\xi,\eta}^w)'''(t)|= 6w \geq 1$. For the case $\floor{\kappa n} \leq \ell < \floor{n/3}$ of \eqref{eq:mult dec ell}, we first write $\beta = \beta^+ + \beta^-$ where $\beta^+ \in C^{\infty}_c(\R)$ is supported on $[1/2, 2]$ and $\beta^- \in C^{\infty}_c(\R)$ is supported on $[-2, -1/2]$. We decompose 
    \begin{equation*}
    m_{n, k_1, k_2}^{w,\ell} = m_{n, k_1, k_2}^{w, \ell, +} + m_{n, k_1, k_2}^{w, \ell, -}
    \end{equation*}
    where
    \begin{equation*}
      m_{n, k_1, k_2}^{w, \ell, \pm}(\xi, \eta) := \beta^{\pm}(2^{-2k_2+2\ell}\Delta(w;\xi,\eta)) m_{n,k_1,k_2}^w(\xi, \eta).
    \end{equation*}
    
    Since $\Delta(w; \xi, \eta)<0$ on the support of $m_{n,k_1,k_2}^{w,\ell, -}$, the phase function $\phi_{\xi, \eta}^w$ has no (real) critical points (see \eqref{240612e6_9uuu}). Furthermore, 
    \begin{equation*}
        (\phi_{\xi,\eta}^w)'(t) \geq (\phi_{\xi,\eta}^w)'\Big(\frac{\eta}{3w}\Big) = - \xi - \frac{\eta^2}{3w} \geq 2^{2k_2-2\ell-6}
    \end{equation*}
    for all $t \in \supp a$ and $(\xi,\eta) \in \supp m_{n,k_1,k_2}^{w,\ell, -}$, since $t \mapsto (\phi_{\xi,\eta}^w)'(t)$ has a minimum at $t=\eta/3w$. Consequently, by van der Corput's lemma with first order derivatives,
    \begin{equation}\label{eq:mult dec ell -}
        |m_{n,k_1,k_2}^{w, \ell, -} (\xi, \eta)| \lesssim 2^{-(n-2\ell)} \qquad \text{for $(\xi,\eta) \in \supp m_{n,k_1,k_2}^{w,\ell, -}$}.
    \end{equation}

    On the support of $m_{n,k_1,k_2}^{w,\ell, +}$, we have $\Delta (w;\xi,\eta)>0$ and the phase function $\phi_{\xi,\eta}^w$ admits the critical points $t_{\pm}(w; \xi,\eta)$, which are the roots of the quadratic polynomial $(\phi_{\xi,\eta}^w)'(t)$ with leading coefficient $3w$. Thus,
    \begin{align}
    \label{eq:phi root 1}
        (\phi_{\xi,\eta}^w)'(t) &= 3w(t - t_+(w;\xi,\eta))(t - t_-(w;\xi,\eta)),\\
    \label{eq:phi root 2}   
        (\phi_{\xi,\eta}^w)''(t) &= 6w\Big(t - \frac{t_+(w;\xi,\eta) + t_-(w;\xi,\eta)}{2} \Big). 
    \end{align}
    From the quadratic formula \eqref{240612e6_9uuu}, we have 
    \begin{equation*}
        \frac{t_+(w;\xi,\eta) - t_-(w;\xi,\eta)}{2} = \frac{\sqrt{\Delta(w;\xi,\eta)}}{3 w}. 
    \end{equation*}
    Since $\Delta(w;\xi,\eta) \geq 2^{2k_2 - 2\ell -1}$ on the support of $m_{n,k_1,k_2}^{w,\ell,+}$, we conclude that there exists an absolute constant $c > 0$ such that $|t_+(w;\xi,\eta) - t_-(w;\xi,\eta)| \geq 4 c 2^{-\ell}$ on the support of $m_{n,k_1,k_2}^{w,\ell,+}$.

    We now write $m_n^w = \fm_{n, 1}^w + \fm_{n, 2}^w$ where
    \begin{equation*}
        \fm_{n, i}^w(\xi, \eta) := \int_{T_i} e^{i 2^n\phi_{\xi,\eta}^w(t)} a(t) \ud t, \qquad i = 1,\,2,
    \end{equation*}
    for
    \begin{align*}
        T_1 &:= \big\{t \in \R : 1/2 \leq |t| \leq 2 \textrm{ and } \min_{\pm} |t - t_{\pm}(w;\xi,\eta)| \geq c 2^{-\ell} \big\},\\
        T_2 &:= \big\{t \in \R : 1/2 \leq |t| \leq 2 \big\} \setminus T_1.
    \end{align*} 
Note that each of the sets $T_i$ consists of a union of a bounded number of intervals. It follows from \eqref{eq:phi root 1} that $|(\phi_{\xi,\eta}^w)'(t)| \geq 3c^2 2^{-2\ell}$ for all $t \in T_1$ whenver $(\xi,\eta) \in \supp m_{n,k_1,k_2}^{w,\ell,+}$. Thus, van der Corput's lemma with first order derivatives gives
\begin{equation}\label{eq:mult dec ell 1}
    |\fm_{n, 1}^w(\xi,\eta)| \lesssim 2^{-(n-2\ell)} \qquad \textrm{for $(\xi,\eta) \in \supp m_{n,k_1,k_2}^{w,\ell,+}$.}
\end{equation}
On the other hand, it follows from \eqref{eq:phi root 2} that $|(\phi_{\xi,\eta}^w)''(t)| \geq 6c 2^{-\ell}$ for all $t \in T_2$ whenever $(\xi,\eta) \in \supp m_{n,k_1,k_2}^{w,\ell,+}$. Thus, van der Corput's lemma with second order derivatives gives
\begin{equation}\label{eq:mult dec ell 2}
    |\fm_{n, 2}^w(\xi,\eta)| \lesssim 2^{-(n-\ell)/2} \qquad \textrm{for $(\xi,\eta) \in \supp m_{n,k_1,k_2}^{w,\ell,+}$.}
\end{equation}
Since $\ell \leq n/3$, the estimates \eqref{eq:mult dec ell -} and \eqref{eq:mult dec ell 1} and \eqref{eq:mult dec ell 2} combine to give \eqref{eq:mult dec ell}.
\end{proof}

Using the preceding oscillatory integral estimates, we conclude this section with the proof of Lemma~\ref{lem:L2 smoothing}.

\begin{proof}[Proof of Lemma~\ref{lem:L2 smoothing}] We first observe that the $p=2$ case of Lemma~\ref{lem:L2 smoothing} is a direct consequence of Lemma~\ref{lem:vdc multiplier 2} and Plancherel's theorem.\medskip

\noindent \textit{i)} Let $(k_1, k_2) \in \cB_{\diag}$ satisfy $k_2 > 8$ and $k_1 \leq \kappa n$. The kernel estimate from Lemma~\ref{lem:kernel est ell} i) implies 
\begin{equation*}
    \sup_{1 \leq w \leq 8} \| \cH_{n,k_1,k_2}^{w,\leq \kappa} f \|_{L^\infty(\R^2)} \lesssim 2^{2\kappa n} \| f \|_{L^\infty(\R^2)}.
\end{equation*}
Interpolating this with the $p = 2$ case gives the desired estimate.\medskip

\noindent \textit{ii)} Let $(k_1, k_2) \in \cB_{\diag}$ satisfy $k_2 \leq 8$. For all $ \floor{\kappa n} \leq  \ell \leq \floor{n/3}$, the kernel estimate from Lemma~\ref{lem:kernel est ell} ii) implies 
\begin{equation*}
    \sup_{1 \leq w \leq 8} \| \cH_{n,k_1,k_2}^{w,\ell} f \|_{L^\infty(\R^2)} \lesssim 2^{2\ell} \| f \|_{L^\infty(\R^2)}.
\end{equation*}
Interpolating this with the $p = 2$ case again gives the desired estimate.
\end{proof}

%%%%%%%%%%%%%%%%%%%%%%%%%%%%%%%%%%%%%%%%%%%%%%%%%%%%%%%%%%%%%%%%%%%%%%%%%%%%%%%%%%%%%%%%%%%%%%%%

%    $L^p$ local smoothing

%%%%%%%%%%%%%%%%%%%%%%%%%%%%%%%%%%%%%%%%%%%%%%%%%%%%%%%%%%%%%%%%%%%%%%%%%%%%%%%%%%%%%%%%%%%%%%%%

\section{$L^p$ local smoothing}\label{sec:loc smoothing}

%%%%%%%%%%%%%%%%%%%%%%%%%%%%%%%%%%%%%%%%%%%%%%%%%%%%%%%%%%%%%%%%%%%%%%%%%%%%%%%%%%%%%%%%%%%%%%%%

%    Phase computations

%%%%%%%%%%%%%%%%%%%%%%%%%%%%%%%%%%%%%%%%%%%%%%%%%%%%%%%%%%%%%%%%%%%%%%%%%%%%%%%%%%%%%%%%%%%%%%%%

\subsection{Setup} Recall the multiplier $m_n^w$ defined in \eqref{eq: multiplier 1} is an oscillatory integral with phase $\phi_{\xi,\eta}^w$. Furthermore, in \eqref{240612e6_9uuu} we introduced the functions 
\begin{equation*}
   t_{\pm}(w; \xi, \eta) := \frac{\eta \pm \sqrt{\Delta(w; \xi, \eta)}}{3w} \qquad \textrm{for} \qquad \Delta(w; \xi, \eta) := \eta^2 + 3w\xi,
\end{equation*}
 which correspond to the two (potentially non-real) roots of $(\phi_{\xi,\eta}^w)'$. By standard stationary phase computations (see the proof of Proposition~\ref{prop: general ls} below), provided the multiplier is suitably localised, we expect $m_n^w$ to oscillate with a fundamental frequency given by one of the phases 
 \begin{equation*}
   \phi_{\pm}(w; \xi, \eta):=
    \phi^w_{\xi, \eta} \circ t_{\pm}(w; \xi, \eta).    
 \end{equation*}
    This leads us to consider the pair of translation-invariant phases
\begin{equation}\label{eq: our phase}
    \Phi_{\pm}(x, y, w; \xi, \eta):= x\xi+y\eta+ \phi_{\pm}(w; \xi, \eta).
\end{equation}

In this section we prove local smoothing estimates for operators of the form $U^{\lambda}[\Phi_{\pm}; \fa]$ introduced in \eqref{eq: Schrod prop} for $\Phi_{\pm}$ as in \eqref{eq: our phase} and $\fa$ suitably localised. These estimates are then used to prove Proposition~\ref{prop:Lp smoothing low} and Proposition~\ref{prop:comparable Lp smoothing}, thereby concluding the proof of Theorem~\ref{thm:main}.

%%%%%%%%%%%%%%%%%%%%%%%%%%%%%%%%%%%%%%%%%%%%%%%%%%%%%%%%%%%%%%%%%%%%%%%%%%%%%%%%%%%%%%%%%%%%%%%%

%    Localised amplitudes

%%%%%%%%%%%%%%%%%%%%%%%%%%%%%%%%%%%%%%%%%%%%%%%%%%%%%%%%%%%%%%%%%%%%%%%%%%%%%%%%%%%%%%%%%%%%%%%%

\subsection{Localised amplitudes} Here and throughout the remainder of the section, we let $C_{\star} \geq 1$ denote a fixed absolute constant, which chosen sufficiently large for the purposes of the forthcoming arguments. 

To ensure our operators satisfy the H\"ormander condition \eqref{eq: Hormander}, we work with the following class of amplitude functions.

\begin{definition}\label{dfn: amplitude A} For $0 < \tau \leq 1$, let $\fA_{\pm}(\tau)$ denote the class of $\fa_{\pm} \in C^{\infty}_c(\R \times \widehat{\R}^2)$ such that for all $(w; \xi, \eta) \in \supp \fa_{\pm}$ we have $|\xi|$, $|\eta| \leq C_{\star}$, $\tau \leq w \leq C_{\star}$ and
    \begin{equation}\label{eq: amplitude A supp}
        \tau \leq \Delta(w; \xi, \eta) \quad \textrm{and} \quad 1/4 \leq t_{\pm}(w; \xi, \eta) \leq 4.
    \end{equation}
\end{definition}

\begin{lemma}\label{lem: Hormander cond} If $\fa_{\pm} \in \fA_{\pm}(\tau)$, then the following hold.
\begin{enumerate}[i)]
    \item The phase $\Phi_{\pm}$ is real-analytic and satisfies \eqref{eq: Hormander} on a neighbourhood of $\supp \fa_{\pm}$. In particular, for $\bxi=(\xi,\eta)$, the matrix
\begin{equation}\label{eq: Hessian}
    \partial_{\bxi\bxi}^2 \partial_w\phi_{\pm}(w; \xi, \eta) := 
    \begin{bmatrix}
            \partial_w\partial^2_{\xi\xi} \phi_{\pm}(w; \xi, \eta) & \partial_w\partial^2_{\xi \eta} \phi_{\pm}(w; \xi, \eta)\\
            \partial_w\partial^2_{\xi \eta} \phi_{\pm}(w; \xi, \eta) & \partial_w\partial^2_{\eta\eta} \phi_{\pm}(w; \xi, \eta)
\end{bmatrix}
\end{equation}
satisfies
\begin{equation}\label{eq: quantified Hormander}
    |\det \partial_{\bxi\bxi}^2 \partial_w\phi_{\pm}(w; \xi, \eta)| \gtrsim 1 \qquad \textrm{for all $(w; \xi, \eta) \in \supp \fa_{\pm}$},
\end{equation}
where the implied constants can be taken to be absolute. 
\item We also have the derivative bounds
\begin{equation}\label{eq: quantified derivs}
    |\partial_w^{\alpha} \partial_{\xi}^{\beta_1} \partial_{\eta}^{\beta_2} \phi_{\pm}(w; \xi, \eta)| \lesssim_{\alpha, \beta_1, \beta_2} \tau^{-\alpha - \beta_1 - \beta_2 + 1} \qquad \textrm{for $(w; \xi, \eta) \in \supp \fa_{\pm}$}
\end{equation}
for all $\alpha$, $\beta_1$, $\beta_2 \in \N_0$ with $\alpha + \beta_1 + \beta_2 \geq 1$.
\end{enumerate}
\end{lemma}

\begin{proof} First note, by the discriminant hypothesis in \eqref{eq: amplitude A supp}, the function $\Phi_{\pm}$ is real analytic on a neighbourhood of $\supp \fa_{\pm}$. Furthermore, the derivative bounds in ii) follow by a straightforward direct computation: see \S\ref{subsec: deriv dictionary}.\footnote{For derivatives of second and higher order, there is actually a $\tau^{1/2}$-gain over \eqref{eq: quantified derivs}, but this is irrelevant for our purposes.}

It remains to prove \eqref{eq: quantified Hormander}. A computation (\textit{c.f.} \S\ref{subsec: H2 comp}) shows that
\begin{equation}\label{eq: H2 formula}
\det
\partial_{\bxi\bxi}^2 \partial_w\phi_{\pm}(w; \xi, \eta)
=
- \frac{9}{4} \cdot \frac{t_{\pm}(w; \xi, \eta)^4}{\Delta(w; \xi, \eta)^2}.
\end{equation}
The desired bound immediately follows from \eqref{eq: amplitude A supp}. 
\end{proof}

The above observation allows us to apply the classical theory of H\"ormander-type operators to the phase/amplitude pair $[\Phi_{\pm}; \fa_{\pm}]$ whenever $\fa_{\pm} \in \fA_{\pm}(\tau)$. However, this is insufficient for our purposes and we wish to apply Theorem~\ref{thm: general propagator} to establish improved estimates. For this, we consider a more restrictive class of amplitudes. 

\begin{definition}\label{dfn: amplitude N E} For $0 < \sigma, \tau \leq 1$ and $W \geq 1$ we define the following.
\begin{enumerate}[i)]
    \item Let $\fN_{\pm}(\sigma, \tau)$ denote the class of $\fn_{\pm} \in \fA_{\pm}(\tau)$ which, in addition, satisfy
    \begin{equation*}
        |\xi| \geq \sigma \qquad \textrm{for all $(w; \xi, \eta) \in \supp \fn_{\pm}$}
    \end{equation*}
    and
    \begin{equation}\label{eq: amplitude deriv n}
        |\partial_w^{\alpha} \partial_{\xi}^{\beta_1}
 \partial_{\eta}^{\beta_2} \fn_{\pm}(w; \xi, \eta)| \lesssim_{\alpha, \beta_1, \beta_2} \tau^{-\alpha-\beta_2} (\tau \wedge \sigma)^{-\beta_1} \qquad \textrm{for all $\alpha$, $\beta_1$, $\beta_2 \in \N_0$.}
 \end{equation}
    \item Let $\fE_{\pm}(\sigma)$ denote the class of $\fe_{\pm} \in \fA_{\pm}(C_{\star}^{-1})$ which, in addition, satisfy
    \begin{equation}\label{eq: amplitude loc e}
        |\xi| \leq \sigma \qquad \textrm{for all $(w; \xi, \eta) \in \supp \fe_{\pm}$}
    \end{equation}
    and
    \begin{equation}\label{eq: amplitude deriv e}
        |\partial_w^{\alpha} \partial_{\xi}^{\beta_1}
 \partial_{\eta}^{\beta_2} \fe_{\pm}(w; \xi, \eta)| \lesssim_{\alpha, \beta_1, \beta_2}  \sigma^{-\beta_1-\beta_2} \qquad \textrm{for all $\alpha$, $\beta_1$, $\beta_2 \in \N_0$.}
 \end{equation}
 \end{enumerate}
 Here the implied constants in \eqref{eq: amplitude deriv n} and \eqref{eq: amplitude deriv e} can be thought of as a fixed sequence, the terms of which are defined sufficiently large so as to satisfy the forthcoming requirements of the proof. 
\end{definition}

Definition~\ref{dfn: amplitude N E} i) is specifically tailored for application of Theorem~\ref{thm: general propagator}. On the other hand, Definition \ref{dfn: amplitude N E} ii) allows application of decoupling theory: see \S\ref{subsec: decoupling arg}.

\begin{lemma}\label{lem: Nik cond} If $\fn_{\pm} \in \fN_{\pm}(\sigma, \tau)$ for some $0 < \sigma, \tau < 1$, then $\Phi_{\pm}$ satisfies the Nikodym non-compression hypothesis on $\supp \fn_{\pm}$ with parameter $\gamma$ satisfying $\gamma \sim \sigma$. In particular, the matrix 
\begin{equation*}
    N_{\pm}(w;\xi,\eta) := 
    \begin{bmatrix}
        \partial_w \partial_{\xi \xi}^2 \phi_{\pm} (w; \xi, \eta) & \partial_w \partial_{\xi \eta}^2 \phi_{\pm} (w; \xi, \eta) & \partial_w \partial_{\eta \eta}^2 \phi_{\pm} (w; \xi, \eta) \\
        \partial_w^{2} \partial_{\xi \xi}^2 \phi_{\pm} (w; \xi, \eta) & \partial_w^{2} \partial_{\xi \eta}^2 \phi_{\pm} (w; \xi, \eta) & \partial_w^{2}  \partial_{\eta \eta}^2 \phi_{\pm} (w; \xi, \eta) \\
        \partial_w^{3} \partial_{\xi \xi}^2 \phi_{\pm} (w; \xi, \eta) & \partial_w^{3} \partial_{\xi \eta}^2 \phi_{\pm} (w; \xi, \eta) & \partial_w^{3}  \partial_{\eta \eta}^2 \phi_{\pm} (w; \xi, \eta) \\
    \end{bmatrix}
    \end{equation*}
satisfies
\begin{equation}\label{eq: Nik cond}
    |\det N_{\pm}(w; \xi, \eta)| \gtrsim \sigma \qquad \textrm{for all $(w; \xi, \eta) \in \supp \fh_{\pm}$},
\end{equation}
where the implied constants can be taken to be absolute.
\end{lemma}

\begin{proof} Direct computation (\textit{c.f.} \S\ref{subsec: Nik comp}) reveals that $\det \partial^2_{\bxi\bxi} \phi_{\pm} (w; \xi, \eta) \equiv 0$ as a function of $(w; \xi, \eta)$ and so we only need to verify \eqref{eq: Nik cond}. A further, and somewhat involved(!), computation (\textit{c.f.} \S\ref{subsec: Nik comp}) shows that 
 \begin{equation}\label{eq: Nik formula}
    \det N_{\pm}(w; \xi, \eta)
    = -\frac{3^7}{2^4} \cdot \frac{\xi \cdot t_{\pm}(w; \xi, \eta)^8}{\Delta(w; \xi, \eta)^5}.
\end{equation}
The desired bound therefore follows immediately from the support conditions in Definition~\ref{dfn: amplitude A} and Definition~\ref{dfn: amplitude N E} i).
\end{proof}

%%%%%%%%%%%%%%%%%%%%%%%%%%%%%%%%%%%%%%%%%%%%%%%%%%%%%%%%%%%%%%%%%%%%%%%%%%%%%%%%%%%%%%%%%%%%%%%%

%    Local smoothing via Nikodym non-compression

%%%%%%%%%%%%%%%%%%%%%%%%%%%%%%%%%%%%%%%%%%%%%%%%%%%%%%%%%%%%%%%%%%%%%%%%%%%%%%%%%%%%%%%%%%%%%%%%

\subsection{Local smoothing via Nikodym non-compression} We combine the observations of the preceding section with Theorem~\ref{thm: general propagator} to deduce the following.

\begin{proposition}\label{prop: ls Nikodym}
    There exists some $M \in \N$ such that the following holds. For all $2 < p < 4$ there exists $\varepsilon_{\mathrm{n}}(p) > 0$ such that 
    \begin{equation*}
    \big\| U^{\lambda}[\Phi_{\pm}; \fn_{\pm}]f\big\|_{L^p(\R^3)}
     \lesssim  \lambda^{1/2-\varepsilon_{\mathrm{n}}(p)} (\sigma\tau)^{-M}
    \|f\|_{L^p(\R^2)}
    \end{equation*}
for all $\lambda \geq 1$, whenever $\fn_{\pm} \in \fN_{\pm}(\sigma, \tau)$ for some $0 < \sigma, \tau < 1$.
\end{proposition}

\begin{proof} By Lemma~\ref{lem: Hormander cond} i), we know that $\Phi_{\pm}$ is translation-invariant, real analytic and satisfies H1) and H2) on an open neighbourhood of $\supp \fn_{\pm}$. Furthermore, by Lemma~\ref{lem: Nik cond}, there exists some $\gamma \sim \sigma$ such that $\Phi_{\pm}$ satisfies the Nikodym non-compression hypothesis with parameter $\gamma$ on an open neighbourhood of $\supp \fn_{\pm}$. The hypothesis \eqref{eq: amplitude deriv n} and Lemma~\ref{lem: Hormander cond} ii) guarantee that the derivative bounds \eqref{eq: phase/amp derivs} hold for $W = O((\sigma\tau)^{-M_0})$ for some absolute exponent $M_0 \in \N$. We may therefore directly apply Theorem~\ref{thm: general propagator} to obtain the desired result.
\end{proof}

%%%%%%%%%%%%%%%%%%%%%%%%%%%%%%%%%%%%%%%%%%%%%%%%%%%%%%%%%%%%%%%%%%%%%%%%%%%%%%%%%%%%%%%%%%%%%%%%

%    Local smoothing for the exceptional piece via decoupling

%%%%%%%%%%%%%%%%%%%%%%%%%%%%%%%%%%%%%%%%%%%%%%%%%%%%%%%%%%%%%%%%%%%%%%%%%%%%%%%%%%%%%%%%%%%%%%%%

\subsection{Local smoothing for the exceptional piece via decoupling} To complement Proposition~\ref{prop: ls Nikodym}, we prove a local smoothing result for operators of the form $U^{\lambda}[\Phi_{\pm}; \fe_{\pm}]$ with $\fe_{\pm} \in \fE_{\pm}(\sigma)$. 

\begin{proposition}\label{prop: ls decoupling} For all $2 \leq p \leq 4$  and $\delta > 0$, given $\lambda \geq 1$ and $\lambda^{-1/2} \leq \sigma \leq 1$, we have 
\begin{equation}\label{eq: ls decoupling}
    \big\| U^{\lambda}[\Phi_{\pm}; \fe_{\pm}]f\big\|_{L^{p}(\R^3)}
     \lesssim_{\delta} \lambda^{1/2 + \delta} \sigma^{1/2 - 1/p} 
    \|f\|_{L^{p}(\R^2)}.
    \end{equation}
    whenever $\fe_{\pm} \in \fE_{\pm}(\sigma)$.
\end{proposition}
The mechanism behind Proposition~\ref{prop: ls decoupling} is very different from that of Proposition~\ref{prop: ls Nikodym} and, in particular, relies on $\ell^p$ decoupling theory. This is a standard tool for proving $L^p$-local smoothing estimates for wave-type propagators \cite{Wolff2000}; it is typically not as relevant in the Schrödinger context. However, here decoupling allows for the (crucial) gain of $\sigma^{1/2 - 1/p}$, which arises from the frequency localisation near the line $\cL := \{(0, \eta) : \eta \in \R\}$ in the amplitude class $\fE_\pm(\sigma)$. The proof is complicated by the fact that our operator is \textit{hyperbolic}: the eigenvalues of the matrix \eqref{eq: Hessian} have opposite signs. This rules out application of the more powerful $\ell^2$ decoupling theory, and instead we rely on a two-step decoupling process which combines a lower dimensional `small cap' decoupling inequality, exploiting curvature properties of $\Phi_{\pm}$ restricted to $\cL$, with a variable coefficient extension of the $\ell^p$ decoupling theory from \cite{BD2017}. We postpone the details until \S\ref{subsec: decoupling arg} below.

%%%%%%%%%%%%%%%%%%%%%%%%%%%%%%%%%%%%%%%%%%%%%%%%%%%%%%%%%%%%%%%%%%%%%%%%%%%%%%%%%%%%%%%%%%%%%%%%

%    Applying the local smoothing estimates

%%%%%%%%%%%%%%%%%%%%%%%%%%%%%%%%%%%%%%%%%%%%%%%%%%%%%%%%%%%%%%%%%%%%%%%%%%%%%%%%%%%%%%%%%%%%%%%%

\subsection{Applying the local smoothing estimates} We now combine Proposition~\ref{prop: ls Nikodym} and Proposition~\ref{prop: ls decoupling} with stationary phase arguments to prove a general $L^p$ smoothing estimate from which Proposition~\ref{prop:Lp smoothing low} and Proposition~\ref{prop:comparable Lp smoothing} follows.

\begin{definition}\label{dfn: amplitude B} For $0 < \tau \leq 1$, let $\fB(\tau)$ denote the class of $b \in C^{\infty}_c(\R \times \widehat{\R}^2)$ such that:
\begin{enumerate}[i)]
    \item $|\xi|$, $|\eta| \leq C_{\star}$, $\tau \leq w \leq C_{\star}$ and $|\Delta(w; \xi, \eta)| \geq \tau$ for all $(w; \xi, \eta) \in \supp b$. 
\item If $\tau < C_{\star}^{-1}$, then $|\xi| > C_{\star}^{-1}$ for all $\xi \in \mathrm{supp}_{\xi}\,b$.
\item $|\partial_w^{\alpha}\partial_{\xi}^{\beta_1} \partial_{\eta}^{\beta_2} b(w;\xi, \eta)| \lesssim_{\alpha, \beta_1, \beta_2} \tau^{-\alpha - \beta_1 - \beta_2}$ for all $\alpha$, $\beta_1$, $\beta_2 \in \N_0$.
\end{enumerate}
\end{definition}

Given $b \in \fB(\tau)$, for each $w \in \R$ and $n \in \N$ let $\cH_n^w[b]$ denote the operator with Fourier multiplier
\begin{equation*}
 m_n^w[b](\xi, \eta) :=  b(w; \xi, \eta) m_n^w(\xi, \eta) = b(w; \xi, \eta) \int_{\R} e^{i2^n\phi_{\xi, \eta}^w(t)} a(t) \,\ud t. 
\end{equation*}
With these definitions, our general smoothing estimate reads thus.

\begin{proposition}\label{prop: general ls} There exists some $M \in \N$ such that the following holds. For all $2 < p < 4$, there exists some $\varepsilon(p) > 0$ such that
\begin{equation}\label{eq: general ls}
    \Big(\int_0^{\infty} \big\| \cH_n^w[b]f \big\|_{L^p(\R^2)}^p\,\ud w \Big)^{1/p} \lesssim_p  2^{-n/p - \varepsilon(p)n} \tau^{-M} \|f\|_{L^p(\R^2)}
\end{equation}
holds whenever $b \in \fB(\tau)$ for some $2^{-n/4} \leq \tau \leq 1$. 
\end{proposition}

\begin{proof} We break the proof into steps. \medskip

\noindent \textit{Step 1: Initial decomposition.} Write $m_n^w[b] = \sum_\pm m_{n, \pm}^w[b] + m_{n, \error}^w[b]$ where
\begin{equation*}
    m_{n, \pm}^w[b](\xi, \eta) := b_{\main}(w; \xi, \eta) \int_{\R} e^{i2^n\phi_{\xi, \eta}^w(t)} a_{\pm}(w; \xi, \eta; t) \,\ud t
\end{equation*}
and, for $C_\star \geq 1$ as above,
\begin{align*}
    b_{\main}(w; \xi, \eta) &:= b(w; \xi, \eta)\chi_{(0,\infty)} \big( \Delta(w;\xi, \eta)\big), \\
    a_{\pm}(w; \xi, \eta; t) &:= a(t) \beta_0\big(C_\star^2\tau^{-1/2}(t - t_{\pm}(w; \xi, \eta)) \big). 
\end{align*}
The hypotheses on the support of $b$ ensures the function $b_{\main}$ is smooth. Furthermore, for $(w;\xi, \eta) \in \supp b_{\main}$, the critical points $t_{\pm}(w; \xi, \eta)$ are real valued and so the $a_{\pm}$ are well-defined. Recall that the $t_{\pm}(w; \xi,\eta)$ satisfy 
\begin{equation*}
  \frac{t_+(w; \xi,\eta) - t_-(w; \xi,\eta)}{2} = \frac{\sqrt{\Delta(w; \xi, \eta)}}{3w} \geq (3C_\star)^{-1} \tau^{1/2}. 
\end{equation*}
It therefore follows that for each fixed $(w; \xi, \eta) \in \supp b_{\main}$, the $t$-supports of the $a_{\pm}$ are disjoint. \medskip

\noindent \textit{Step 2: Non-stationary term.} Here we argue as in the proof of Lemma \ref{lem:vdc multiplier 2} ii). 

Let $b_{\error} := b - b_{\main}$. For $(w; \xi, \eta) \in \supp b_{\error}$ we have $\Delta(w; \xi, \eta)<0$ and
\begin{equation*}
    (\phi_{\xi, \eta}^w)'(t) \geq (\phi_{\xi, \eta}^w)'\Big(\frac{\eta}{3w}\Big) = -\frac{\Delta(w; \xi, \eta)}{3w} \geq (3C_\star)^{-1}\tau > 0 \quad \text{for all $t \in \supp a$,}
\end{equation*}
where in the first inequality we have used that $t \mapsto (\phi_{\xi,\eta}^w)'(t)$ has a minimum at $t=\eta/3w$. 
Thus, by non-stationary phase (Lemma~\ref{non-stationary lem}), we see that 
\begin{equation}\label{eq: error 1}
    |\partial_{\xi}^{\beta_1} \partial_{\eta}^{\beta_2} m_n^w[b_{\error}](\xi, \eta)| \lesssim_{\beta_1, \beta_2, N} 2^{n(\beta_1 + \beta_2)} 2^{-Nn} \tau^{-N}, \quad \textrm{$\beta_1$, $\beta_2$, $N \in \N_0$.}
\end{equation}

On the other hand, let $a_{\error} := a - a_+ - a_-$. It follows from \eqref{eq:phi root 1} that, if $(w; \xi, \eta) \in \supp b_{\main}$, then
$|(\phi_{\xi, \eta}^w)'(t)| \geq 3C_\star^{-4} \tau^{2}$  for all $t \in \supp a_{\error}(w; \xi, \eta; \,\cdot\,)$.
Since 
\begin{equation*}
    (m_n^w-m_{n,\pm}^w)[b_{\main}](\xi, \eta)=b_{\main}(w; \xi, \eta) \int_{\R} e^{i 2^n \phi_{\xi,\eta}^w (t)} a_{\error}(w; \xi, \eta; t) \ud t,
\end{equation*}
another application of non-stationary phase yields 
\begin{equation}\label{eq: error 2}
    |\partial_{\xi}^{\beta_1} \partial_{\eta}^{\beta_2} (m_n^w-m_{n,\pm}^w)[b_{\main}](\xi, \eta)| \lesssim_{\beta_1, \beta_2, N} 2^{n(\beta_1 + \beta_2)} 2^{-Nn} \tau^{-2N}
\end{equation}
for all $\beta_1$, $\beta_2$, $N \in \N_0$.

Combining \eqref{eq: error 1} and \eqref{eq: error 2} we have, by repeated integration-by-parts,
\begin{equation*}
    |\mathcal{F}^{-1}[m_{n,\error}^w [b]](x,y)|\lesssim_N \frac{2^{-Nn} \tau^{-2N}}{(1+ 2^{-n}|x| + 2^{-n}|y|)^{10}}
\end{equation*}
for all $N \in \N_0$ and $w \in \R$. Thus, by Young's convolution inequality, \eqref{eq: general ls} holds for the piece of the operator associated to $m_{n,\error}^w[b]$.\medskip 

\noindent \textit{Step 3: Passing to variable propagators.} In light of the above observations, it suffices to prove that \eqref{eq: general ls} holds with the operator $\cH_n^w[b]$ replaced with either of the operators $\cH_{n, \pm}^w[b]$ associated to the Fourier multipliers $m_{n, \pm}^w[b]$. By \eqref{eq:phi root 2}, we have 
\begin{equation*}
    (\phi_{\xi,\eta}^w)''(t_{\pm}(w; \xi,\eta)) = \pm 6w\Big(\frac{t_+(w; \xi,\eta) - t_-(w; \xi,\eta)}{2} \Big) = \pm2\sqrt{\Delta(w; \xi, \eta)}
\end{equation*}
and so the critical points of the phase are non-degenerate, in the sense that 
\begin{equation*}
    |(\phi_{\xi,\eta}^w)''(t_{\pm}(w; \xi,\eta))| \geq 2 \tau^{1/2} > 0.
\end{equation*} 
Provided $C_\star \geq 1$ is chosen sufficiently large, we may therefore apply standard asymptotic expansions for oscillatory integrals (see, for instance, \cite[Chapter 1]{Sogge2017}) to deduce that
\begin{equation}\label{eq: stationary phase}
    m_{n, \pm}^w[b](\xi, \eta) = \tau^{-1/4} 2^{-n/2} e^{i 2^n \phi_{\pm}(w; \xi, \eta)} \fa_{\pm}(w; \xi, \eta)
\end{equation}
where each $\fa_{\pm}$ is a smooth function satisfying $\supp \fa_{\pm} \subseteq b_{\main}$ and 
\begin{equation}\label{eq: der bounds stationary phase}
   |\partial^\alpha_w \partial^{\beta_1}_\xi \partial^{\beta_2}_\eta \fa_{\pm}(w; \xi, \eta)| \lesssim_{\alpha, \beta_1, \beta_2} \tau^{-\alpha - \beta_1 - \beta_2} \qquad \textrm{for all $\alpha$, $\beta_1$, $\beta_2 \in \N_0$.}
\end{equation}

In view of \eqref{eq: stationary phase}, for $\Phi_{\pm}$ as in \eqref{eq: our phase}, we have $\cH_{n,\pm}^{2^{-n} w}=\tau^{-1/4} 2^{-n/2}U^{2^n}[\Phi_\pm; \fa_\pm]$. In particular, applying the rescaling $w \mapsto 2^{-n} w$, in order to prove \eqref{eq: general ls} it suffices to show that for all $2 < p < 4$ there exists some $\varepsilon(p)>0$ such that 
\begin{equation}\label{240412e6_16}
 \big\|
    U^{\lambda}[\Phi_{\pm}; \fa_{\pm}]f
    \big\|_{L^p(\R^3)}
     \lesssim \lambda^{
    1/2-\varepsilon(p)} \tau^{-M}
    \|f\|_{L^p(\R^2)}
\end{equation}
holds uniformly over all $\lambda \geq 1$, for some universal exponent $M \in \N$. \medskip

\noindent \textit{Step 4: Applying $L^p$ local smoothing.}
We shall prove \eqref{240412e6_16} by combining Proposition~\ref{prop: ls Nikodym} and Proposition~\ref{prop: ls decoupling}. First of all, we verify that the amplitudes $\fa_{\pm}$ lie in the class $\fA_{\pm}(\tau)$ from Definition~\ref{dfn: amplitude A}. The conditions $|\xi|$, $|\eta| \leq C_\star$, $\tau \leq w \leq C_\star$ and $|\Delta(w; \xi, \eta)| \geq \tau$ for $(w; \xi, \eta) \in \supp \fa_{\pm}$ are inherited from the support of $b$. It remains to show $1/4 \leq t_{\pm}(w; \xi, \eta) \leq 4$ for $(w; \xi, \eta) \in \supp \fa_{\pm}$. Indeed, if the above condition on $t_{\pm}(w; \xi, \eta)$ fails and $C_\star \geq 4$, then it follows that \begin{equation*}
    |t - t_{\pm}(w; \xi, \eta)| > 2C_\star^{-2} \geq 2C_\star^{-2}\tau^{1/2} \qquad \textrm{for all $t \in [1/2, 2]$.} 
\end{equation*}
Thus, we see from the definitions that $a_{\pm}(w; \xi, \eta; \,\cdot\,) \equiv 0$, which forces $\fa_\pm(w; \xi, \eta) = 0$. These conclusions hold in a neighbourhood of $(w; \xi, \eta)$ and so $(w; \xi, \eta) \notin \supp \fa_{\pm}$.

Now fix $2 < p < 4$ and let $M \in \N$ and $\varepsilon_{\mathrm{n}}(p) > 0$ be the quantities guaranteed by Proposition~\ref{prop: ls Nikodym}. Defining $\delta_{\mathrm{n}}(p) := \varepsilon_{\mathrm{n}}(p)/(2M) \leq 1/4$, we write
\begin{equation*}
    \fa_{\pm} =
    \fn_{\pm} + \fe_{\pm} \quad \textrm{where} \quad \fe_{\pm}(w; \xi, \eta) := 
    \fa_{\pm}(w; \xi, \eta) \beta_0\big(\lambda^{\delta_{\mathrm{n}}(p)}\xi\big).
\end{equation*}
We may assume $\lambda^{\delta_{\mathrm{n}}(p)} \geq 4C_{\star}$, since otherwise the desired bound is trivial. Thus, by property ii) from Definition~\ref{dfn: amplitude B}, if $\tau < C_{\star}^{-1}$, then $|\xi|>C_\star^{-1}$ for $\xi \in \mathrm{supp}_\xi \, \fa_\pm$, so $\fe_{\pm} \equiv 0$. Using this observation and \eqref{eq: der bounds stationary phase}, we see that
\begin{equation*}
   \fn_{\pm} \in \fN_{\pm}(\lambda^{-\delta_{\mathrm{n}}(p)}, \tau) \qquad \textrm{and} \qquad \fe_{\pm} \in \fE_{\pm}(\lambda^{-\delta_{\mathrm{n}}(p)}).
\end{equation*}
Thus, Proposition~\ref{prop: ls Nikodym} implies that
\begin{equation*}
\big\|
    U^{\lambda}[\Phi_{\pm}; \fn_{\pm}]f
    \big\|_{L^p(\R^3)}
     \lesssim_p \lambda^{
    1/2-\varepsilon_{\mathrm{n}}(p)/2} \tau^{-M}
    \|f\|_{L^p(\R^2)}.
\end{equation*}
On the other hand, Proposition~\ref{prop: ls decoupling} (with, say, $\delta := 1/4 - 1/(2p) > 0$) implies that
\begin{equation*}
\big\|
    U^{\lambda}[\Phi_{\pm}; \fe_{\pm}]f
    \big\|_{L^p(\R^3)}
     \lesssim_p \lambda^{
    1/2-\varepsilon(p)}
    \|f\|_{L^p(\R^2)}
\end{equation*}
where $\varepsilon(p) := (1/4 - 1/(2p))\delta_{\mathrm{n}}(p)$ satisfies $0 < \varepsilon(p) < \varepsilon_{\mathrm{n}}(p)/2$. Combining these two estimates, we deduce the desired bound \eqref{240412e6_16}.
\end{proof}

%%%%%%%%%%%%%%%%%%%%%%%%%%%%%%%%%%%%%%%%%%%%%%%%%%%%%%%%%%%%%%%%%%%%%%%%%%%%%%%%%%%%%%%%%%%%%%%%

%    Applying the local smoothing estimates

%%%%%%%%%%%%%%%%%%%%%%%%%%%%%%%%%%%%%%%%%%%%%%%%%%%%%%%%%%%%%%%%%%%%%%%%%%%%%%%%%%%%%%%%%%%%%%%%

\subsection{Proof of Proposition~\ref{prop:Lp smoothing low} and Proposition~\ref{prop:comparable Lp smoothing}}

Both Proposition~\ref{prop:Lp smoothing low} and Proposition~\ref{prop:comparable Lp smoothing} follow easily from Proposition~\ref{prop: general ls}. It is convenient to begin with the proof of the latter. Throughout what follows, we let $\rho \in C^{\infty}_c(\R)$ be a fixed bump function satisfying $\rho(w) = 1$ for all $1 \leq w \leq 8$ and $\supp \rho \subseteq [1/2, 16]$.

\begin{proof}[Proof of Proposition~\ref{prop:comparable Lp smoothing}] Fix $n \in \N$, $(k_1, k_2) \in \cB_{\diag}$ satisfying $k_1 \leq \kappa n$. Note that
\begin{equation*}
    \cH_{n, k_1, k_2}^{w, > \kappa}f(x,y) = \cH_{n, k_1, k_2}^{w, > \kappa}f(x,y) \rho(w) = \cH_n^w[b_0]f(x,y)
\end{equation*}
for $1 \leq w \leq 8$ and $(x,y) \in \R^2$, where
\begin{equation*}
      b_0(w;\xi,\eta) := \rho(w)\beta(2^{-k_1}\xi)\beta(2^{-k_2} \eta)\big(1 - \beta_0\big(2^{-2k_2 + 2\floor{n\kappa}}\Delta(w;\xi, \eta)\big)\big).
\end{equation*}
We apply the rescaling $(\xi, \eta) \mapsto (2^{k_2}\xi, 2^{k_2}\eta)$ and $w \mapsto 2^{k_2} w$ to arrive at $\cH_{n+k_2}^w[b]f$ where
\begin{equation*}
    b(w;\xi,\eta) := \rho(2^{k_2}w)\beta(2^{k_2-k_1}\xi)\beta(\eta)\big(1 - \beta_0\big(2^{2\floor{n\kappa}}\Delta(w;\xi, \eta)\big)\big). 
\end{equation*}
Thus, given $2 < p < 4$, matters are reduced to showing 
\begin{equation}\label{eq:comparable Lp smoothing 1}
    \Big(\int_0^{\infty} \big\| \cH_n^w[b]f \big\|_{L^p(\R^2)}^p\,\ud w \Big)^{1/p} \lesssim_p 2^{2M\kappa  n} 2^{-n/p - \varepsilon(p)n} \|f\|_{L^p(\R^2)},
\end{equation}
where $\varepsilon(p) > 0$ and $M \in \N$ are as in the statement of Proposition~\ref{prop: general ls}. However, using the condition $(k_1, k_2) \in \cB_{\diag}$, if $C_{\star} \geq 1$ is chosen sufficiently large and we define $\tau := 2^{-C_2 - 16}2^{-2n \kappa}$, then 
\begin{equation*}
    C_{\star}^{-1} \leq |\xi| \leq C_{\star}, \quad |\eta| \leq C_{\star}, \quad \tau \leq w \leq C_{\star} \quad \textrm{and} \quad |\Delta(w; \xi, \eta)| \geq \tau
\end{equation*}
for all $(w; \xi, \eta) \in \supp b$ and it is clear that $b \in \fB(\tau)$. Thus, \eqref{eq:comparable Lp smoothing 1} is a direct consequence of Proposition~\ref{prop: general ls}. 
\end{proof}

Combining Proposition~\ref{prop:comparable Lp smoothing} with the $L^2$-based theory from \S\ref{sec:L2 theory}, we obtain Proposition~\ref{prop:Lp smoothing diag}. All that remains is to verify Proposition~\ref{prop:Lp smoothing low}.

\begin{proof}[Proof of Proposition~\ref{prop:Lp smoothing low}] Both i) and ii) easily follow by verifying that the relevant multipliers satisfy the hypotheses of Proposition~\ref{prop: general ls}.\medskip

\noindent \textit{i)} Suppose $-C_1 < k_1 \leq 8$ and recall that $C_1=2C_2+12$. We begin with a technical reduction, which relies on Proposition~\ref{prop:Lp smoothing diag}. Defining $\widetilde{C}_2 := C_2 + 8$, we may write
\begin{equation*}
\cH^w_{n, k_1, \leq -C_2} = \cH^w_{n, k_1, \leq -\widetilde{C}_2} + \sum_{k_2=-\widetilde{C}_2 + 1}^{-C_2} \cH^w_{n, k_1, k_2}.
\end{equation*}
For each $-\tilde{C}_2 < k_2 \leq -C_2$, we have $|k_1 - k_2| \leq C_2 + 16$ and so $(k_1, k_2) \in \cB_{\diag}$. Hence, we can apply Proposition~\ref{prop:Lp smoothing diag} to favourably bound each of the operators $\cH^w_{n, k_1, k_2}$ in the above sum.

It remains to estimate $\cH^w_{n, k_1, \leq -\widetilde{C}_2}$. As in the proof of Proposition~\ref{prop:comparable Lp smoothing}, matters are reduced to bounding $\cH_n^w[b]$ where
\begin{equation*}
     b(w;\xi,\eta) := \rho(w)\beta(2^{-k_1}\xi)\beta_0(2^{\tilde{C}_2} \eta).
\end{equation*}
On $\supp b$ we have $|\xi| \geq 2^{-k_1 - 1} \geq 2^{-C_1 - 2}$ and $|\eta| \leq 2^{-\widetilde{C}_2 + 1}$. Consequently,
\begin{equation*}
  |\Delta(w; \xi, \eta)| =  |\eta^2 + 3 \xi w| \geq 3 \cdot 2^{-C_1 - 3} - 2^{-2\tilde{C}_2 + 2} \geq 2^{-C_1 - 3}
\end{equation*}
provided $\tilde{C}_2$ satisfies $2\tilde{C}_2 \geq C_1 + 4$. However, since $\tilde{C}_2 = C_2 + 8$ and $C_1=2C_2  + 12$, this is indeed the case. Recalling Definition~\ref{dfn: amplitude B}, provided $C_{\star} \geq 1$ is chosen sufficiently large, it is therefore clear that $b \in \fB(\tau)$ for $\tau \sim 1$ and Proposition~\ref{prop: general ls} applies to give the desired result. \medskip

\noindent \textit{ii)} Suppose $-C_2 < k_2 \leq 8$ and consider the operator $\cH^w_{n, \leq 
 -C_1, k_2}$. Arguing as before, matters are reduced to bounding $\cH_n^w[b]$ where\footnote{Note that it is only here we consider $|\xi| \ll 1$ so that the decoupling argument is needed.}
\begin{equation*}
    b(w;\xi,\eta) := \rho(w)\beta_0(2^{C_1}\xi)\beta(2^{-k_2} \eta).
\end{equation*}
On the support of $b$ we have $|\xi| \leq 2^{-C_1 + 1}$ and $|\eta| \geq 2^{-k_2 - 1} \geq 2^{-C_2 - 2}$. Consequently,
\begin{equation*}
    |\Delta(w; \xi, \eta)| = |\eta^2 + 3 \xi w| \geq 2^{-2C_2 - 4} - 3 \cdot 2^{-C_1 + 5}  \geq 2^{-2C_2 - 5} > 0,
\end{equation*}
since $C_1 = 2 C_2 + 12$. Thus, provided $C_{\star} \geq 1$ is chosen sufficiently large, it is clear that $b \in \fB(\tau)$ for $\tau \sim 1$ and Proposition~\ref{prop: general ls} applies to give the desired result.
\end{proof}

%%%%%%%%%%%%%%%%%%%%%%%%%%%%%%%%%%%%%%%%%%%%%%%%%%%%%%%%%%%%%%%%%%%%%%%%%%%%%%%%%%%%%%%%%%%%%%%%

%    The decoupling argument 

%%%%%%%%%%%%%%%%%%%%%%%%%%%%%%%%%%%%%%%%%%%%%%%%%%%%%%%%%%%%%%%%%%%%%%%%%%%%%%%%%%%%%%%%%%%%%%%%

\subsection{The decoupling argument}\label{subsec: decoupling arg} It remains to prove Proposition~\ref{prop: ls decoupling}. We begin with some preliminary definitions and lemmas. 

For $0 < \sigma \leq 1$, let $\cI(\sigma)$ be a finitely-overlapping cover of $[-2C_{\star}, 2C_{\star}]$  intervals of length $2 \sigma$ and $(\zeta_I)_{I \in \cI(\sigma)}$ be a smooth partition of unity subordinate to $\cI(\sigma)$, satisfying $|\partial_s^N\zeta_I(s)| \lesssim_N \sigma^{-N}$ for all $N \in \N_0$. Let $P_I$ be the Fourier projection operator acting on functions on $\R^2$ with multiplier $(\xi, \eta) \mapsto \zeta_I(\eta)$. Given any $\fe_{\pm} \in \fE_{\pm}(\sigma)$, we further write $\fe_{\pm, I}(w; \xi, \eta) := \fe_{\pm}(w; \xi, \eta) \zeta_I(\eta)$; note that $\fe_{\pm, I} \in \fE_\pm (\sigma)$.

\begin{lemma}[Coarse scale decoupling]\label{lem: dec coarse} For $2 \leq p \leq 4$ and $\delta > 0$, given $\lambda \geq 1$ and $\lambda^{-1/2} \leq \sigma \leq 1$, we have
\begin{equation}\label{lem: eq coarse}
    \big\| U^{\lambda}[\Phi_{\pm}; \fe_{\pm}]f\big\|_{L^{p}(B_{\lambda})}
    \lesssim_{\delta} 
    \sigma^{-(1/2 - 1/p) - \delta}    \Big( \sum_{I \in \cI(\sigma)}
    \big\| U^{\lambda}[\Phi_{\pm}; \fe_{\pm, I}]f\big\|_{L^{p}(w_{B_{\lambda}})}^p \Big)^{1/p},
\end{equation}
up to a rapidly decreasing $L^p$ error, whenever $B_{\lambda} \subseteq \R^3$ is a  $\lambda$-ball and $\fe_{\pm} \in \fE_{\pm}(\sigma)$. 
\end{lemma}
Here and below, we say that an inequality $A(f; \lambda) \lesssim_L B(f; \lambda)$ holds \textit{up to a rapidly decreasing $L^p$ error} if $A(f; \lambda) \lesssim_{L, N} B(f; \lambda) + \lambda^{-N} \|f\|_{L^p(\R^2)}$ holds for all $N \in \N_0$. Furthermore, given a ball $B \subset \R^d$ for $d = 2$, $3$, we let $w_B \colon \R^d \to (0, 1]$ denote the weight
\begin{equation*}
   w_B(\bz) := \big(1 + r_B^{-1}|\bz - \bz_B| \big)^{-100} \qquad \textrm{for all $\bz \in \R^d$,} 
\end{equation*}
where $\bz_B$ and $r_B$ denote the centre and radius of $B$, respectively.

Lemma~\ref{lem: dec coarse} is a consequence of a `small cap' decoupling result, which is essentially a restatement of the sharp Fourier restriction theorem in the plane. 

\begin{definition}\label{def: curve class} Let $\fG$ denote the class of $C^2$ functions $\gamma \colon J \to \R$, for $J \subseteq \R$ an open interval of length at most $2C_{\star}$, satisfying 
\begin{equation*}
    |\gamma^{(j)}(s)| \leq C_{\star}^5 \quad \textrm{for $j = 1$, $2$} \quad \textrm{and} \quad |\gamma^{(2)}(s)| \geq C_{\star}^{-5} \qquad \textrm{for all $s \in J$.}
\end{equation*}
Given $\gamma \in \fG$ and $\rho > 0$, let $\cZ(\gamma; \rho)$ denote the space of $F \in \cS(\R^2)$ satisfying
\begin{equation*}
     \supp \widehat{F} \subseteq \{(\xi, \eta) \in J \times \widehat{\R} : |\eta - \gamma(\xi)| < \rho \}.
\end{equation*}
\end{definition}

With this definition, the small cap decoupling lemma reads thus.

\begin{lemma}[Small cap decoupling in the plane (\textit{c.f.} {\cite[Lemma 4.2]{MR4293776}})]\label{lem: small cap dec} Let $\gamma \in \fG$. For all $2 \leq p \leq 4$, $\delta > 0$ and $R \geq 1$, we have
\begin{equation*}
    \|F\|_{L^p(B_R)} \lesssim_{p, \delta} R^{1/2 - 1/p + \delta} \Big(\sum_{I \in \cI(R^{-1})} \|P_I F\|_{L^p(w_{B_R})}^p \Big)^{1/p} \qquad \textrm{}
\end{equation*}
whenever $F \in \cZ(\gamma; R^{-1})$ and $B_R \subseteq \R^2$ is an $R$-ball. 
\end{lemma}

By standard interpolation arguments, the lemma follows from the case $p=4$ which is a consequence of the classical $L^4$ Fourier restriction estimate, dating back to \cite{Fefferman1970, Zygmund1974}. The latter implies that, for all $\delta > 0$, we have
\begin{align}
    \nonumber
    \|F\|_{L^4(\R^2)} &\lesssim_{p, \delta} R^{-3/4 + \delta}   \|\widehat{F}\|_{L^4(\R^2)}\\
    \label{eq: 2d restriction}
    &\lesssim R^{-3/4 + \delta} \Big( \sum_{I \in \cI(R^{-1})} \|(P_I F)\;\widehat{}\;\|_{L^4(\R^2)}^4\Big)^{1/4}
\end{align}
whenever $F \in \cS(\R^2)$ satisfies the above Fourier support condition. Indeed, the first inequality in \eqref{eq: 2d restriction} follows from the uncertainty principle reformulation of the restriction estimate as in \cite{Tao2004}, whilst the second inequality is a trivial consequence of the essentially disjoint frequency support. To pass from \eqref{eq: 2d restriction} to the $p = 4$ case of Lemma~\ref{lem: small cap dec}, one applies the Hausdorff--Young inequality to the right-hand norms in \eqref{eq: 2d restriction} to pass to $L^{4/3}$ norms of the $P_I F$. The resulting inequality implies a local version of itself where the integrals are essentially localised to a spatial $R$-ball. One may then apply H\"older's inequality to pass from the localised $L^{4/3}$ norms to $L^4$ norms, at the cost of a factor of $R$.

\begin{proof}[Proof (of Lemma~\ref{lem: dec coarse})] Fix $\fe_{\pm} \in \fE_{\pm}(\sigma)$ and $0 < \delta < 1/2$ and, for notational convenience, in this proof we let $U^{\lambda}_{\pm} := U^{\lambda}[\Phi_{\pm}; \fe_{\pm}]$ and $U^{\lambda}_{\pm, I} := U^{\lambda}[\Phi_{\pm}; \fe_{\pm, I}]$. Note that, by H\"older's inequality, \eqref{lem: eq coarse} holds with the factor $\sigma^{-\delta}$ placed with $\sigma^{-1/2}$. Consequently, we assume without loss of generality that
\begin{equation}\label{eq: coarse dec 0}
   \sigma \leq \min\big\{6^{-1}C_{\star}^{-2}, \lambda^{-\delta}\big\} < 1/100.
\end{equation}

Fix $B_{\lambda} \subseteq \R^2$ a $\lambda$-ball and let $\cB(\sigma^{-1})$ be a finitely-overlapping covering of $B_{\lambda}$ by balls of radius $\sigma^{-1}$. Fixing $2 \leq p \leq 4$ and
$\delta > 0$, we aim to show
\begin{equation}\label{eq: coarse dec 1}
    \big\| U^{\lambda}_{\pm}f(x, \,\cdot\,,\,\cdot\,)\big\|_{L^{p}(B)}
     \lesssim_{\delta, N} 
    \sigma^{-(1/2 - 1/p) - \delta} 
    \Big( \sum_{I \in \cI(\sigma)}
    \big\| U^{\lambda}_{\pm, I}f(x, \,\cdot\,,\,\cdot\,)\big\|_{L^{p}(w_B)}^p \Big)^{1/p}
\end{equation}
holds up to inclusion of a rapidly decreasing $L^p$ error for each $x \in \R$ and every ball $B \in \cB(\sigma^{-1})$. Once \eqref{eq: coarse dec 1} is established, the desired result immediately follows by first taking an $\ell^p$ norm of both sides of the inequality over all $B \in \cB(\sigma^{-1})$ and then taking an $L^p$-norm in $x$.

Let $\rho \in \cS(\R^2)$ satisfy $\supp \hat{\rho} \subseteq B(0,1)$ and $|\rho(\bz)| \gtrsim 1$ for all $\bz \in B(0,1)$. Given $B \in \cB(\sigma^{-1})$ with centre $\bz_B = (y_B, w_B)$, define $\rho_B(\bz) := \rho(\sigma(\bz - \bz_B))$. Furthermore, for each $x \in \R$ let
\begin{equation*}
    G_{\pm, x}^{\lambda, B}(\bz) := U^{\lambda}_{\pm}f(x, \bz_B + \bz)\rho_B(\bz_B + \bz), 
    \end{equation*}
for $\bz = (y,w) \in \R^2$. It then follows that
\begin{equation}\label{eq: coarse dec 2}
    \big\| U^{\lambda}_{\pm}f(x, \,\cdot\,,\,\cdot\,)\big\|_{L^{p}(B)} \lesssim \|G_{\pm, x}^{\lambda, B}\|_{L^p(\R^2)}.
\end{equation}

Our goal is to show that the Fourier transform $(G_{\pm, x}^{\lambda, B})\;\widehat{}\;$ is essentially supported in the $\sigma$-neighbourhood of some curve, allowing us to apply Lemma~\ref{lem: small cap dec} above. This Fourier support property is a consequence of the spatio-temporal localisation to the $\sigma^{-1}$-ball: at this scale the variable coefficient phase $\Phi_{\pm}^{\lambda}$ can be approximated by a constant coefficient phase, and the operator $U^{\lambda}_{\pm}f$ is effectively a Fourier transform.

For $\bz = (y,w) \in \R^2$, we may write 
\begin{equation*}
    G_{\pm, x}^{\lambda, B}(\bz) = \int_{\R^2} K_{\pm, x}^{\lambda, B}(\bz; u,v)f(u,v)\,\ud u \ud v
\end{equation*}
where 
\begin{equation}\label{eq: kernel coarse}
    K_{\pm,x}^{\lambda, B}(\bz; u,v)= \rho_{B} (\bz_B + \bz) \int_{\widehat{\R}^2} e^{-i(u \xi + v \eta) + i \Phi^\lambda_\pm (x, \bz_B + \bz; \bxi)} \fe_\pm^\lambda (w_B + w; \bxi)  \ud \bxi
\end{equation}
for $\bxi = (\xi, \eta)$. Observe that
\begin{equation}\label{eq: coarse dec 3}
    \Phi_{\pm}^\lambda (x, \bz_B + \bz; \xi, \eta) = \Phi_{\pm}^\lambda(x, y_B, w_B; \xi, \eta) + y\eta + w h_{\pm}^{\lambda, B}(\xi, \eta) + \cE_{\pm}^{\lambda, B}(w; \xi, \eta)
\end{equation}
where $h_{\pm}^{\lambda, B}(\xi, \eta) := \partial_w \phi_{\pm}(w_B/\lambda; \xi, \eta)$ and 
\begin{equation*}
    \cE_{\pm}^{\lambda, B}(w; \xi, \eta):=\frac{w^2}{\lambda} \int_0^1 \partial_w^2 \phi_\pm \big((w_B+sw)/\lambda; \xi, \eta\big) \, (1-s) \ud s;
\end{equation*}
in particular $\cE_{\pm}^{\lambda, B}(w; \xi, \eta)$ satisfies
\begin{equation}\label{eq: coarse dec 4}
    |\partial_w^{\alpha} \partial_{\xi}^{\beta_1} \partial_{\eta}^{\beta_2} \cE_{\pm}^{\lambda, B}(w; \xi, \eta)| \lesssim_{\alpha, \beta_1, \beta_2} \sigma^{\alpha} \qquad \textrm{for all $\alpha$, $\beta_1$, $\beta_2 \in \N_0$.}
\end{equation}
Note that \eqref{eq: coarse dec 4} relies on the hypothesis $\sigma^2 \geq \lambda^{-1}$, the support hypothesis $|w| \leq \sigma^{-1}$ and Lemma~\ref{lem: Hormander cond} ii).

Taking the Fourier transform of the kernel representation of $G_{\pm, x}^{\lambda, B}$ in \eqref{eq: kernel coarse}, applying \eqref{eq: coarse dec 3} and Fubini's theorem, we obtain
\begin{equation}\label{eq: coarse dec 5}
    (G_{\pm, x}^{\lambda, B})\;\widehat{}\;(\tilde{\eta}, \tilde{\omega}) = \int_{\R^2} \int_{\widehat{\R}^2} e^{i \Phi_{\pm}^{\lambda}(x-u, y_B - v, w_B; \xi, \eta)} A_{\pm}^{\lambda, B}(\tilde{\eta}, \tilde{\omega}; \xi, \eta)\ud\xi\ud\eta \,f(u,v) \ud u \ud v
\end{equation}
where
\begin{equation*}
A_{\pm}^{\lambda, B}(\tilde{\eta}, \tilde{\omega}; \xi, \eta) := \int_{\R^2} e^{-i(y(\tilde{\eta} - \eta) + w(\tilde{\omega} - h_{\pm}^{\lambda, B}(\xi, \eta)))} \fe_{\pm}^{\lambda, B}(y, w; \xi, \eta)\ud y \ud w    
\end{equation*}
for
\begin{equation*}
    \fe_{\pm}^{\lambda, B}(y, w; \xi, \eta) := e^{i \cE_{\pm}^{\lambda, B}(w; \xi, \eta)} \fe_{\pm}^{\lambda}(w_B + w; \xi, \eta) \rho_B(y_B + y, w_B +  w)
\end{equation*}
and where we have used that
\begin{equation*}
    -u \xi - v \eta + \Phi^\lambda_\pm (x,y_B, w_B; \xi, \eta) = \Phi_\pm^\lambda (x-u, y_B-v, w_B; \xi, \eta).
\end{equation*}
In view of \eqref{eq: coarse dec 4}, \eqref{eq: amplitude deriv e} and \eqref{eq: coarse dec 0}, and setting $\bxi=(\xi, \eta)$, we have 
\begin{equation*}
|\partial_{\bz}^{\balpha} \partial_{\bxi}^{\bbeta} \fe_{\pm}^{\lambda, B}(\bz; \bxi)| \lesssim_{\balpha, \bbeta} \sigma^{-\beta_1-\beta_2} \sigma^{\alpha_1 + \alpha_2}
\end{equation*}
for all $\balpha = (\alpha_1, \alpha_2)$, $\bbeta = (\beta_1, \beta_2) \in \N_0^2$. Thus, repeated integration-by-parts yields
\begin{align}\label{eq: coarse dec 6}
    |\partial_{\bxi}^{\bbeta} \partial_{(\tilde{\eta}, \tilde{\omega})}^{\bgamma} &A_{\pm}^{\lambda, B}(\tilde{\eta}, \tilde{\omega}; \bxi)| \\
    \nonumber
    &\lesssim_{\bbeta, \bgamma, N} \sigma^{-\beta_1 - \beta_2 - \gamma_1 - \gamma_2 -2}\big(1 + \sigma^{-1}|\tilde{\eta} - \eta| + \sigma^{-1}|\tilde{\omega} - h_{\pm}^{\lambda, B}(\bxi)|\big)^{-N}
\end{align}
for all $\bbeta = (\beta_1, \beta_2)$, $\bgamma = (\gamma_1, \gamma_2) \in \N_0^2$ and $N \in \N_0$; here we have used that $\rho_B$ decays rapidly outside of a ball of radius $\sigma^{-1}$.

By \eqref{eq: amplitude loc e} and Lemma~\ref{lem: Hormander cond} ii), there exists an absolute constant $C \geq 1$ such that
\begin{equation}\label{eq: coarse dec 7}
    |\tilde{\omega} - h_{\pm}^{\lambda, B}(\xi, \eta)| \geq |\tilde{\omega} - h_{\pm}^{\lambda, B}(0, \tilde{\eta})| - C(|\tilde{\eta} - \eta| + \sigma)
\end{equation}
whenever $(\xi, \eta) \in \mathrm{supp}_{\xi, \eta}\, A_{\pm}^{\lambda, B}$. On the other hand, given $(\xi, \eta) \in \widehat{\R}^2$, we may write $\eta^2 = \Delta(w; \xi, \eta) - 3w\xi$ for any $w \in \R$. In particular, the hypotheses $\sigma < 6^{-1}C_\star^{-2}$ from \eqref{eq: coarse dec 0} imply that 
\begin{equation}\label{eq: coarse dec 8}
    2C_{\star}^{-1} \leq (2C_{\star})^{-1/2} \leq |\eta| \leq C_{\star} \qquad \textrm{whenever $\eta \in \mathrm{supp}_{\eta}\, \fe_{\pm}$.}
\end{equation} 
Let $\delta_{\circ} := \delta/100$ and define $\cE := \widehat{\R}^2 \setminus \cN$ where 
\begin{equation}\label{eq: coarse dec 9}
  \cN := \big\{ (\tilde{\eta}, \tilde{\omega}) \in \widehat{\R}^2 : C_{\star}^{-1} \leq |\tilde{\eta}| \leq 2C_{\star}, \, |\tilde{\omega} - h_{\pm}^{\lambda,B}(0, \tilde{\eta})| < \lambda^{\delta_{\circ}}  \sigma \big\}. 
\end{equation}
We may assume $4C\lambda^{-\delta} \leq \lambda^{-\delta_{\circ}}$, otherwise the desired result is trivial. Combining \eqref{eq: coarse dec 6}, \eqref{eq: coarse dec 7} and \eqref{eq: coarse dec 8} and using the hypothesis $\sigma < \lambda^{-\delta}$ from \eqref{eq: coarse dec 0}, we see that 
\begin{equation}\label{eq: coarse dec 10}
    |\partial_{\bxi}^{\bbeta} \partial_{(\tilde{\eta}, \tilde{\omega})}^{\bgamma}  A_{\pm}^{\lambda, B}(\tilde{\eta}, \tilde{\omega}; \xi, \eta)| \lesssim_{\delta, \bbeta, \bgamma, N} \lambda^{-N} \qquad \textrm{whenever $(\tilde{\eta}, \tilde{\omega}) \in \cE$}
\end{equation}
for all $\bbeta$, $\bgamma \in \N_0^2$, $N \in \N_0$. Furthermore, by repeatedly applying integration-by-parts to \eqref{eq: coarse dec 5} in the $(\xi, \eta)$ variables and combining the resulting expression with the estimate \eqref{eq: coarse dec 10}, Lemma \ref{lem: Hormander cond} ii) and H\"older's inequality,
\begin{equation}\label{eq: coarse dec 11}
    |\partial_{(\tilde{\eta}, \tilde{\omega})}^{\bgamma} (G_{\pm, x}^{\lambda, B})\;\widehat{}\;(\tilde{\eta}, \tilde{\omega})| \lesssim_{\delta, \bgamma, N} \lambda^{-N} \|f\|_{L^p(\R^2)} \qquad \textrm{whenever $(\tilde{\eta}, \tilde{\omega}) \in \cE$.}
\end{equation}

Now let $P_{\pm}^{\lambda, B}$ be the Fourier projection operator with Fourier multiplier
\begin{equation*}
    \beta_0\big(\lambda^{-\delta_{\circ}}  \sigma^{-1}(\tilde{\omega} - h_{\pm}^{\lambda, B}(0, \tilde{\eta}))\big) \big( \beta_0\big( (2C_{\star})^{-1} \tilde{\eta} \big) - \beta_0(C_{\star}\tilde{\eta}) \big)
\end{equation*}
Combining the Fourier inversion formula with a further integration-by-parts argument and \eqref{eq: coarse dec 11}, we have 
\begin{equation}\label{eq: coarse dec 12}
    \big|\big(I - P_{\pm}^{\lambda, B}\big)G_{\pm, x}^{\lambda, B}(y, w)\big| \lesssim_N \lambda^{-N} \|f\|_{L^p(\R^2)} \qquad \textrm{for all $N \in \N_0$.}
\end{equation}
A minor adjustment of this argument also shows that
\begin{equation}\label{eq: coarse dec 13}
    \big|P_{\tilde{I}}\big(I - P_{\pm}^{\lambda, B}\big)G_{\pm, x}^{\lambda, B}(y, w)\big| \lesssim_N \lambda^{-N} \|f\|_{L^p(\R^2)} \qquad \textrm{for all $N \in \N_0$}
\end{equation}
for all $\tilde{I} \in \cI(\lambda^{\delta_{\circ}}  \sigma)$. 

On the other hand, the projection $P_{\pm}^{\lambda, B}G_{\pm, x}^{\lambda, B}$ has Fourier support within the set $\cN$ defined in \eqref{eq: coarse dec 9}. This is the vertical $2\lambda^{\delta_{\circ}}  \sigma$-neighbourhood of the curve
\begin{equation*}
    \big\{(s, \gamma_{\pm}^{\lambda, B}(s)) : C_\star^{-1} \leq |s| \leq 2C_\star \big\} \qquad \textrm{where} \qquad  \gamma_{\pm}^{\lambda, B}(s) := (\partial_w \phi_{\pm})(w_B/\lambda; 0, s).
\end{equation*}
A direct calculation (\textit{c.f.} \S\ref{subsec: curv comp}) shows that 
\begin{equation}\label{eq: curv formula}
 \big(\gamma_{\pm}^{\lambda, B}\big)''(s) = \frac{2^4}{3^2} \cdot \frac{s}{(\lambda^{-1}w_B)^3} \quad \textrm{and so} \quad  |\big(\gamma_{\pm}^{\lambda, B}\big)''(s)| \geq C_{\star}^{-4} \quad \textrm{for $C_{\star}^{-1} \leq |s| \leq 2C_{\star}$.}
\end{equation}
Moreover, it is not difficult to see that $\gamma_{\pm}^{\lambda, B} \in \fG$, using the notation from Definition~\ref{def: curve class}. Furthermore, $P_{\pm}^{\lambda, B}G_{\pm, x}^{\lambda, B} \in \cZ(\gamma_{\pm}^{\lambda, B}; 2 \lambda^{\delta_{\circ}}  \sigma)$ and so we may apply Lemma~\ref{lem: small cap dec} to deduce that
\begin{equation*}
    \|P_{\pm}^{\lambda, B}G_{\pm, x}^{\lambda, B}\|_{L^p(\R^2)} \lesssim_{p, \delta}  \sigma^{-(1/2 - 1/p)} \lambda^{2\delta_{\circ}}  \Big(\sum_{\tilde{I} \in \cI(\lambda^{\delta_{\circ}}  \sigma)} \|P_{\tilde{I}}P_{\pm}^{\lambda, B}G_{\pm, x}^{\lambda, B}\|_{L^p(\R^2)}^p \Big)^{1/p}.
\end{equation*}
Combining this with \eqref{eq: coarse dec 12} and \eqref{eq: coarse dec 13}, we therefore have
\begin{equation}\label{eq: coarse dec 14}
    \|G_{\pm, x}^{\lambda, B}\|_{L^p(\R^2)} \lesssim_{p, \delta, N}  \sigma^{-(1/2 - 1/p)} \lambda^{2\delta_{\circ}} \Big(\sum_{\tilde{I} \in \cI(\lambda^{\delta_{\circ}}  \sigma)} \|P_{\tilde{I}}G_{\pm, x}^{\lambda, B}\|_{L^p(\R^2)}^p \Big)^{1/p}, 
\end{equation}
up to a rapidly decreasing $L^p$ error. 

We may write
\begin{equation}\label{eq: coarse dec 15}
G_{\pm, x}^{\lambda, B} = \sum_{I \in \cI(\sigma)} G_{\pm, x, I}^{\lambda, B} \quad \textrm{where} \quad  G_{\pm, x, I}^{\lambda, B}(\bz) := U^{\lambda}_{\pm, I}f(x, \bz_B + \bz)\rho_B(\bz_B + \bz) 
\end{equation}
for all $I \in \cI(\sigma)$. Each $G_{\pm, x, I}^{\lambda, B}$ is of the same form as the function $G_{\pm, x}^{\lambda, B}$: the only change is that the underlying amplitude $\fe_{\pm}$ is replaced with $\fe_{\pm, I}$. Given any bounded interval $J \subseteq \R$, let $\eta_J$ denote its centre. Since $\fe_{\pm, I} \in \fE_{\pm}^{\lambda}(\sigma)$, we can repeat the analysis in \eqref{eq: coarse dec 5} and \eqref{eq: coarse dec 6} with $G_{\pm, x}^{\lambda, B}$ replaced with $G_{\pm, x, I}^{\lambda, B}$ to conclude that
\begin{equation*}
    |\partial_{(\tilde{\eta}, \tilde{\omega})}^{\bgamma} (G_{\pm, x, I}^{\lambda, B})\;\widehat{}\;(\tilde{\eta}, \tilde{\omega})| \lesssim_{\delta, \bgamma, N} \lambda^{-N} \|f\|_{L^p(\R^2)} \qquad \textrm{whenever $|\tilde{\eta} - \eta_I| > 2\lambda^{\delta_{\circ}}  \sigma$,}
\end{equation*}
for all $\bgamma \in \N_0^2$. In particular, if  $\tilde{I} \in \cI(\lambda^{\delta_{\circ}}  \sigma)$, $I \in \cI(\sigma)$ satisfy $|\eta_I - \eta_{\tilde{I}}| > 10\lambda^{\delta_{\circ}}  \sigma$, then
\begin{equation}\label{eq: coarse dec 16}
    \|P_{\tilde{I}}G_{\pm, x, I}^{\lambda, B}\|_{L^p(\R^2)} \lesssim_N \lambda^{-N} \|f\|_{L^p(\R^2)}.
\end{equation}

By combining \eqref{eq: coarse dec 16} with an application of H\"older's inequality and the $L^p$-boundedness of the projection operators $P_{\tilde{I}}$, we deduce that
\begin{align*}
   \sum_{\tilde{I} \in \cI(\lambda^{\delta_{\circ}}  \sigma)} \|P_{\tilde{I}}G_{\pm, x}^{\lambda, B}\|_{L^p(\R^2)}^p &\lesssim \lambda^{\delta_{\circ}p} \sum_{I \in \cI(\sigma)} \|G_{\pm, x, I}^{\lambda, B}\|_{L^p(\R^2)}^p \\
   &\lesssim\lambda^{\delta_{\circ}p} \sum_{I \in \cI(\sigma)} \|U_{\pm, I}^{\lambda}f(x,\,\cdot\,,\,\cdot\,)\|_{L^p(w_B)}^p,
\end{align*}
where the last step follows from the definition \eqref{eq: coarse dec 15} of the $G_{\pm, x, I}^{\lambda, B}$ and the rapid decay of $\rho$. Finally, we combine the above display with \eqref{eq: coarse dec 2} and \eqref{eq: coarse dec 14} to obtain \eqref{eq: coarse dec 1}, as required. 
\end{proof}

Let $\Theta(\lambda^{-1/2})$ be a finitely-overlapping cover of $[-2C_\star, 2C_\star]^2$ by balls of radius $\lambda^{-1/2}$ and $(\psi_{\theta})_{\theta \in \Theta(\lambda^{-1/2})}$ be a smooth partition of unity subordinate to $\Theta(\lambda^{-1/2})$, satisfying $|\partial_{\bxi}^{\bbeta}\psi_{\theta}(\bxi)| \lesssim_{\bbeta} \lambda^{(\beta_1 + \beta_2)/2}$ for all $\bbeta = (\beta_1, \beta_2) \in \N_0^2$. Given $\fe_{\pm} \in \fE_{\pm}(\sigma)$ and $\theta \in \Theta(\lambda^{-1/2})$, we define $\fe_{\pm, \theta}(w; \bxi) := \fe_{\pm}(w; \bxi) \psi_{\theta}(\bxi)$.

\begin{lemma}[Fine scale decoupling]\label{lem: dec fine} For $2 \leq p \leq 4$ and $\delta > 0$, given $\lambda \geq 1$, $\lambda^{-1/2} \leq \sigma \leq 1$, we have
\begin{equation}\label{eq: dec fine}
\nonumber
    \big\| U^{\lambda}[\Phi_{\pm}; \fe_{\pm, I}]f\big\|_{L^{p}(B_{\lambda})}
    \lesssim_{\delta} 
      (\lambda \sigma^2)^{1/2 - 1/p + \delta}     \Big( \sum_{\theta \in \Theta(\lambda^{-1/2})}
    \big\| U^{\lambda}[\Phi_{\pm}; \fe_{\pm, \theta}]f\big\|_{L^{p}(w_{B_{\lambda}})}^p \Big)^{1/p},
\end{equation}
up to a rapidly decreasing $L^p$ error, whenever $I \in \cI(\sigma)$, $\fe_{\pm} \in \fE_{\pm}(\sigma)$ and $B_{\lambda} \subseteq \R^3$ is a $\lambda$-ball. 
\end{lemma}

The proof follows from variable coefficient $\ell^p$ decoupling, via a scaling argument.

\begin{proof}[Proof (sketch)] Recall that $\Phi_{\pm}$ satisfies conditions H1) and H2) on a neighbourhood of $\supp \fe_{\pm, I}$ and that $\mathrm{supp}_{\xi, \eta} \,\fe_{\pm, I}$ is contained in a ball of radius $O(\sigma)$. We apply parabolic rescaling to the operator $U^{\lambda}[\Phi_{\pm}; \fe_{\pm, I}]$ on right-hand side of \eqref{eq: dec fine}, sending it to an operator with amplitude localised at unit scale. This can be decoupled using the natural variable coefficient generalisation of the $\ell^p$ decoupling inequality for surfaces of non-vanishing Gaussian curvature from \cite[Theorem 1.1]{BD2017}. We carry out this decoupling into caps at scale $\sigma^{-1}\lambda^{-1/2}$. Undoing the parabolic rescaling, we obtain the desired decoupling for our original operator into caps at scale $\lambda^{-1/2}$.

We remark that the required variable coefficient decoupling inequality can be derived from the constant coefficient following the argument of \cite{MR4078231} (which treats the case of non-homogeneous phase functions); see also \cite[Theorem 7.1.2.]{Schippa_thesis} for a closely related result. 
\end{proof}

Finally, for each $\theta \in \Theta(\lambda^{-1/2})$, let
\begin{equation*}
    K^{\lambda}[\Phi_{\pm}; \fe_{\pm, \theta}](\bx, w) := \frac{1}{(2\pi)^2} \int_{\widehat{\R}^2} e^{i\Phi_{\pm}^{\lambda}(\bx, w; \bxi)} \fe_{\pm, \theta}^{\lambda}(w; \bxi)\,\ud \bxi
\end{equation*}
denote the kernel of the localised propagator $U^{\lambda}[\Phi_{\pm}; \fe_{\pm, \theta}]$, viewed as a Fourier multiplier acting on functions on $\R^2$ for fixed $w \in \R$. 

\begin{lemma}[Localised kernel estimates]\label{lem: loc kernel} For $\lambda^{-1/2} \leq \sigma \leq 1$, we have
\begin{equation*}
    |K^{\lambda}[\Phi_{\pm}; \fe_{\pm, \theta}](\bx, w)| \lesssim \lambda^{-1}\big( 1 + \lambda^{-1/2}|\bx + \partial_{\xi}\phi_{\pm}^{\lambda}(w; \bxi_{\theta})|\big)^{-100}, \quad \theta \in \Theta(\lambda^{-1/2}),
\end{equation*}
whenever $\fe_{\pm} \in \fE_{\pm}(\sigma)$.
\end{lemma}

In particular, Lemma~\ref{lem: loc kernel} implies that
\begin{equation*}
    \sup_{w \in \R} \|K^{\lambda}[\Phi_{\pm}; \fe_{\pm, \theta}](\,\cdot\,, w)\|_{L^1(\R^2)} \lesssim 1 \qquad \textrm{for all $\theta \in \Theta(\lambda^{-1/2})$,} 
\end{equation*}
 which leads to favourable $L^{\infty}$ bounds for the localised propagators. 

\begin{proof}[Proof (of Lemma~\ref{lem: loc kernel})] By a change of variables, 
\begin{equation*}
  K^{\lambda}[\Phi_{\pm}; \fe_{\pm, \theta}](\bx, w) := \frac{\lambda^{-1}}{(2\pi)^2} e^{i \Phi_{\pm}(\bx, w; \bxi_{\theta})} \int_{\widehat{\R}^2} e^{i\lambda^{-1/2}\inn{\bx + \partial_{\bxi}\phi_{\pm}^{\lambda}(w; \bxi_{\theta})}{\bxi}} e_{\pm, \theta}^{\lambda}(w; \bxi)\ud \bxi  
\end{equation*}
where 
\begin{equation*}
    e_{\pm, \theta}^{\lambda}(w; \bxi) := e^{i \cE^{\lambda}_{\pm, \theta}(w;\bxi)} \fe_{\pm, \theta}(w/\lambda; \bxi_{\theta} + \lambda^{-1/2}\bxi)
\end{equation*}
for
\begin{equation*}
    \cE^{\lambda}_{\pm, \theta}(w;\bxi) := \phi_{\pm}^{\lambda}(w; \bxi_{\theta} + \lambda^{-1/2}\xi) - \phi_{\pm}^{\lambda}(w; \bxi_{\theta}) - \lambda^{-1/2}\inn{\partial_{\bxi}\phi_{\pm}^{\lambda}(w; \bxi_{\theta})}{\bxi}.
\end{equation*} 
By Lemma~\ref{lem: Hormander cond} ii) and \eqref{eq: amplitude deriv e}, we have $|\partial_{\bxi}^{\bbeta}e_{\pm, \theta}^{\lambda}(w; \bxi)| \lesssim_{\bbeta}  1$ for all $\bbeta \in \N_0^2$. The desired result now follows by repeated integration-by-parts.
\end{proof}

\begin{proof}[Proof (of Proposition~\ref{prop: ls decoupling})] Let $2 \leq p \leq 4$, $\delta > 0$ and fix $\lambda \geq 1$, $\lambda^{-1/2} < \sigma \leq 1$ and $\fe_{\pm} \in \fE_{\pm}(\sigma)$. By Lemma~\ref{lem: loc kernel} and the fact that $|w| \lesssim \lambda$ whenever $w \in \mathrm{supp}_w\,K^{\lambda}[\Phi_{\pm}; \fe_{\pm}]$, the operator $U^{\lambda}[\Phi_{\pm}; \fe_{\pm}]$ is essentially local. In particular, it suffices to prove  \eqref{eq: ls decoupling} with the left-hand norm restricted to some $\lambda$-ball $B_{\lambda} \subseteq \R^3$.

By combining Lemma~\ref{lem: dec coarse} and Lemma~\ref{lem: dec fine}, using the rapid decay of $w_{B_{\lambda}}$, we deduce that 
\begin{equation*}
    \big\| U^{\lambda}[\Phi_{\pm}; \fe_{\pm}]f\big\|_{L^{p}(B_{\lambda})}
    \lesssim_{\delta} \lambda^{\delta}
      (\lambda \sigma)^{1/2 - 1/p}   
    \Big( \sum_{\theta \in \Theta(\lambda^{-1/2})}
    \big\| U^{\lambda}[\Phi_{\pm}; \fe_{\pm, \theta}]f\big\|_{L^{p}(\R^3)}^p \Big)^{1/p}
\end{equation*}
holds up to the inclusion of a rapidly decreasing $L^p$ error. It remains to show 
\begin{equation}\label{eq: exceptional 1}
    \Big( \sum_{\theta \in \Theta(\lambda^{-1/2})}
    \big\| U^{\lambda}[\Phi_{\pm}; \fe_{\pm, \theta}]f\big\|_{L^{p}(\R^3)}^p \Big)^{1/p} \lesssim \lambda^{1/p}  \|f\|_{L^p(\R^2)}.
\end{equation}

By Plancherel's theorem (since our phases are translation-invariant: more generally one may appeal to H\"ormander's oscillatory integral theorem~\cite{Hormander1973}), 
\begin{equation}\label{eq: exceptional 2}
    \big\| U^{\lambda}[\Phi_{\pm}; \fe_{\pm, \theta}]f\big\|_{L^2(\R^3)} \lesssim \lambda^{1/2} \|f_{\theta}\|_{L^2(\R^2)},
\end{equation}
where $f_{\theta}$ is defined by $\widehat{f}_{\theta} := \tilde{\psi}_{\theta} \cdot \widehat{f}$ for $\tilde{\psi}_{\theta} \in C^{\infty}_c(\widehat{\R}^2)$ satisfying $\tilde{\psi}_{\theta}(\bxi) = 1$ for $\bxi \in 2 \cdot \theta$ and $\supp \tilde{\psi}_{\theta} \subseteq 4 \cdot \theta$. By taking the $\ell^2$ sum of \eqref{eq: exceptional 2} over all $\theta \in \Theta(\lambda^{-1/2})$ and summing the resulting right-hand side using Plancherel's theorem, we obtain the $p = 2$ case of \eqref{eq: exceptional 1}. 

On the other hand, the kernel estimate from Lemma~\ref{lem: loc kernel} readily implies that 
\begin{equation*}
    \|U^{\lambda}[\Phi_{\pm}; \fe_{\pm, \theta}]f\|_{L^{\infty}(\R^3)} \lesssim  \|f\|_{L^{\infty}(\R^2)}
\end{equation*}
for all $\theta \in \Theta(\lambda^{-1/2})$. This is precisely the $p = \infty$ case of \eqref{eq: exceptional 1}. Interpolating this with the above $L^2$ yields \eqref{eq: exceptional 1} for all $2 \leq p \leq \infty$. 
\end{proof}

\appendix

%%%%%%%%%%%%%%%%%%%%%%%%%%%%%%%%%%%%%%%%%%%%%%%%%%%%%%%%%%%%%%%%%%%%%%%%%%%%%%%%%%%%%%%%%%%%%%%%

%    Quantified oscillatory integral estimates 

%%%%%%%%%%%%%%%%%%%%%%%%%%%%%%%%%%%%%%%%%%%%%%%%%%%%%%%%%%%%%%%%%%%%%%%%%%%%%%%%%%%%%%%%%%%%%%%%

\section{Quantified oscillatory integral estimates}\label{appendix: quantification}

Here we discuss the polynomial dependence on $\gamma$ and $W$ in Theorem~\ref{thm: general propagator}. 

We first note that, by rescaling in the $\bx$, $w$ and $\bxi$ variables, it suffices to consider the case $W = 1$. Note that the rescaling enlarges the size of the support of $a$, but we can simply decompose the support into $O(W^C)$ many pieces of diameter $1$ and then translate each piece back to the origin.

It remains to track the dependency on $\gamma$ in the proof of Theorem 2.8 II) from \cite{CGGHMW}. The H\"ormander condition \eqref{eq: Hormander gamma} appears in an application of the `universal' oscillatory integral estimate of Stein \cite{Stein1986} (see \cite[Theorem 1.3]{CGGHMW}, which is used in the proof of \cite[Proposition 9.1]{CGGHMW}), the `universal' geometric estimates of Wisewell~\cite{Wisewell2005} (see \cite[Proposition 1.4]{CGGHMW}, which is used in the proof of \cite[Proposition 6.4]{CGGHMW}) and a direct application in \cite[Corollary 6.7]{CGGHMW}. It is easy to verify the desired polynomial dependence on $\gamma$ from these steps. 

The Nikodym non-compression hypothesis is used in the proof of \cite[Proposition 5.3]{CGGHMW} to prove a crucial sublevel set estimate. The following quantified variant can be established via a straightforward adaptation of the proof of \cite[Proposition 5.3]{CGGHMW}.

\begin{proposition} There exists some $M \in \N$ such that the following holds. Suppose $g_j \colon (-1,1) \times \B^N \to \R$, $1 \leq j \leq 3$, are real analytic and let $I_{\circ} \subseteq (-1,1)$ and $Y_{\circ} \subset \B^N$ be compact sets containing the origin. For $\gamma > 0$, $W \geq 1$ suppose that matrix
\begin{equation*}
    B(s;y) := (b_{i,j}(s;y))_{1 \leq i, j \leq 3}\quad \textrm{where} \quad b_{i,j}(s;y) := \frac{(\partial_t^i g_j)(s;y) }{i!}, \quad 0 \leq i,\,1 \leq j \leq 3,
\end{equation*}
satisfies $|\det B(s;y)| \geq \gamma > 0$ and
\begin{equation*}
     \big(\sum_{\substack{0 \leq i \leq 3 \\ 1 \leq j \leq 3}} |b_{i,j}(s;y)|^2\big)^{1/2} \leq W, \qquad |b_{i,j}(s;y)| \leq W^{i+1} \qquad \textrm{for all $i \in \N_0$, $1 \leq j \leq 3$,}
\end{equation*}
for all $(s;y) \in I_{\circ} \times Y_{\circ}$. Then the sublevel set estimate
\begin{equation}\label{eq: sublevel}
    \Big|\Big\{t \in I_{\circ} : \big|1 - \sum_{j=1}^3 \mu_j g_j(t;y)\big| < \sigma\Big\}\Big| \lesssim (W\gamma^{-1})^M \sigma^{1/3}
\end{equation}
holds uniformly over all $\sigma > 0$, $y \in Y_{\circ}$ and all scalars $\mu_j \in \R$, $1 \leq j \leq m$. 
\end{proposition}

The Nikodym non-compression hypothesis is only used in \cite{CGGHMW} in the proof of the sublevel set estimate in \cite[Proposition 5.3]{CGGHMW}. Moreover, the sublevel set estimate in \cite[Proposition 5.3]{CGGHMW} is applied precisely once in the proof of \cite[Theorem 2.8 II)]{CGGHMW}. Thus, once the sublevel set estimate is quantified as in \eqref{eq: sublevel}, we can guarantee quantification of the final local smooth estimate in terms of the Nikodym non-compression parameter. 

We remark that the H2) condition and Nikodym non-compression conditions do \textit{not} play a role in the iteration scheme in the proof of \cite[Theorem 6.12]{CGGHMW}.

%%%%%%%%%%%%%%%%%%%%%%%%%%%%%%%%%%%%%%%%%%%%%%%%%%%%%%%%%%%%%%%%%%%%%%%%%%%%%%%%%%%%%%%%%%%%%%%%

%    Calculus computations 

%%%%%%%%%%%%%%%%%%%%%%%%%%%%%%%%%%%%%%%%%%%%%%%%%%%%%%%%%%%%%%%%%%%%%%%%%%%%%%%%%%%%%%%%%%%%%%%%

\section{Calculus computations} 

For the reader's convenience, we present the steps of some of the more involved calculus computations appearing in \S\ref{sec:loc smoothing}.

%%%%%%%%%%%%%%%%%%%%%%%%%%%%%%%%%%%%%%%%%%%%%%%%%%%%%%%%%%%%%%%%%%%%%%%%%%%%%%%%%%%%%%%%%%%%%%%%

%    Derivative dictionary 

%%%%%%%%%%%%%%%%%%%%%%%%%%%%%%%%%%%%%%%%%%%%%%%%%%%%%%%%%%%%%%%%%%%%%%%%%%%%%%%%%%%%%%%%%%%%%%%%

\subsection{Derivative dictionary}\label{subsec: deriv dictionary}

Applying the chain rule to the definition of $\phi_{\pm}$, together with the fact that the $t_{\pm}(w; \xi, \eta)$ are roots of $(\phi_{\xi, \eta}^w)'$, we see that
\begin{equation*}
\partial_{\xi} \phi_{\pm} = -t_{\pm}, \qquad \partial_{\eta} \phi_{\pm} = -t_{\pm}^2, \qquad \partial_w \phi_{\pm} = t_{\pm}^3
\end{equation*}
and so
\begin{equation*}
  \phi_{\xi\xi}^2\phi_{\pm} = - \partial_{\xi} t_{\pm}, \qquad \phi_{\xi\eta}^2\phi_{\pm} = - \partial_{\eta} t_{\pm}, \qquad \phi_{\eta\eta}^2\phi_{\pm} = - 2(\partial_{\eta} t_{\pm}) t_{\pm}.   
\end{equation*}
By (simple) direct computation,
\begin{equation}
    \label{eq: t derivs 1}
    \partial_{\xi}t_{\pm} = \pm \frac{1}{2} \cdot \frac{1}{\Delta^{1/2}}, \qquad \partial_{\eta}t_{\pm} = \pm  \frac{t_{\pm}}{\Delta^{1/2}}, \qquad \partial_w t_{\pm} = \mp \frac{3}{2} \cdot \frac{t_{\pm}^2}{\Delta^{1/2}}
\end{equation}
% Not needed?
%   \partial_{\xi\xi}^2t_{\pm} = \mp \frac{3}{4} \cdot \frac{w}{\Delta^{3/2}}, \qquad \partial_{\xi\eta}^2t_{\pm} = \mp  \frac{1}{2} \cdot \frac{\eta}{\Delta^{3/2}}, \qquad \partial_{\eta\eta}^2t_{\pm} = \pm \frac{\xi}{\Delta^{3/2}}.  
and, by combining the above displays, we see that
\begin{equation}\label{eq: xi derivs}
    \partial_{\bxi \bxi}^2 \phi_{\pm} = \mp \frac{1}{2\Delta^{1/2}}
    \begin{bmatrix}
        1 & 2 t_{\pm} \\
        2 t_{\pm} & 4t_{\pm}^2
    \end{bmatrix}.
\end{equation}
We use these identities to verify the various differential conditions on $\phi_{\pm}$.
%%%%%%%%%%%%%%%%%%%%%%%%%%%%%%%%%%%%%%%%%%%%%%%%%%%%%%%%%%%%%%%%%%%%%%%%%%%%%%%%%%%%%%%%%%%%%%%%

%    H\"ormander condition

%%%%%%%%%%%%%%%%%%%%%%%%%%%%%%%%%%%%%%%%%%%%%%%%%%%%%%%%%%%%%%%%%%%%%%%%%%%%%%%%%%%%%%%%%%%%%%%%

\subsection{H\"ormander condition}\label{subsec: H2 comp} Here we verify \eqref{eq: H2 formula}. From \eqref{eq: xi derivs}, we see that $\partial_{\xi\eta}^2 \phi_{\pm} = 2t_{\pm} \partial_{\xi\xi}^2 \phi_{\pm}$ and $\partial_{\eta\eta}^2 \phi_{\pm} = 4t_{\pm}^2 \partial_{\xi\xi}^2 \phi_{\pm}$. Applying the product rule to these identities, 
\begin{equation*}
    \partial_{\bxi \bxi}^2 \partial_w\phi_{\pm} = 
    \begin{bmatrix}
        \partial_w \partial_{\xi\xi}^2 \phi_{\pm} & 2(\partial_w t_{\pm})\partial_{\xi\xi}^2\phi_{\pm} + 2t_{\pm}\partial_w \partial_{\xi\xi}^2 \phi_{\pm} \\
        2(\partial_w t_{\pm})\partial_{\xi\xi}^2\phi_{\pm} + 2t_{\pm}\partial_w \partial_{\xi\xi}^2 \phi_{\pm} & 8(\partial_w t_{\pm})t_{\pm}\partial_{\xi\xi}^2\phi_{\pm} + 4t_{\pm}^2\partial_w \partial_{\xi\xi}^2 \phi_{\pm}
    \end{bmatrix}.
\end{equation*}
Taking the determinant, we may simplify the matrix by performing the operations
\begin{equation*}
    \textrm{Row 2} \to \textrm{Row 2} - 2 t_{\pm} \times \textrm{Row 1} 
     \quad \textrm{and} \quad 
    \textrm{Col 2} \to \textrm{Col 2} - 2 t_{\pm} \times \textrm{Col 1}
\end{equation*}
to obtain
\begin{equation*}
    \det \partial_{\bxi \bxi}^2 \partial_w\phi_{\pm} = \det \begin{bmatrix}
        \partial_w \partial_{\xi\xi}^2 \phi_{\pm} & 2(\partial_w t_{\pm})\partial_{\xi\xi}^2\phi_{\pm}  \\
        2(\partial_w t_{\pm})\partial_{\xi\xi}^2\phi_{\pm}  & 0
    \end{bmatrix}
    = - 4 (\partial_w t_\pm)^2 (\partial^2_{\xi \xi} \phi_{\pm})^2.
\end{equation*}
Applying the formula for $\partial^2_{\xi \xi} \phi_{\pm}$ from \eqref{eq: xi derivs} and for $\partial_w t_{\pm}$ from \eqref{eq: t derivs 1}, the desired result follows.

%%%%%%%%%%%%%%%%%%%%%%%%%%%%%%%%%%%%%%%%%%%%%%%%%%%%%%%%%%%%%%%%%%%%%%%%%%%%%%%%%%%%%%%%%%%%%%%%

%    Nikodym non-compression

%%%%%%%%%%%%%%%%%%%%%%%%%%%%%%%%%%%%%%%%%%%%%%%%%%%%%%%%%%%%%%%%%%%%%%%%%%%%%%%%%%%%%%%%%%%%%%%%

\subsection{Nikodym non-compression}\label{subsec: Nik comp}  By computing the determinant in \eqref{eq: xi derivs}, we see that $\det\partial_{\bxi \bxi}^2 \phi_{\pm} \equiv 0$. 

We now verify \eqref{eq: Nik formula}. Again note  that, by \eqref{eq: xi derivs}, we have $\partial_{\xi\eta}^2 \phi_{\pm} = 2t_{\pm} \partial_{\xi\xi}^2 \phi_{\pm}$ and $\partial_{\eta\eta}^2 \phi_{\pm} = 4t_{\pm}^2 \partial_{\xi\xi}^2 \phi_{\pm}$. Thus, by applying the chain rule and exploiting obvious linear dependence relations, 
\begin{equation}\label{eq: Nik formula 1}
    \det N_{\pm} = 2^4 \det
    \begin{bmatrix}
        \partial_w \partial_{\xi\xi}^2 \phi_{\pm} & (\partial_w t_{\pm}) (\partial_{\xi \xi}^2 \phi_{\pm}) & 0 \\
        \partial_w^2 \partial_{\xi\xi}^2 \phi_{\pm} & A & (\partial_w t_{\pm})^2(\partial_{\xi\xi}^2 \phi_{\pm} ) \\
        \partial_w^3 \partial_{\xi\xi}^2 \phi_{\pm} & B & C
    \end{bmatrix}
\end{equation}
where
\begin{subequations}
\renewcommand{\theequation}{\theparentequation-\roman{equation}}
\begin{align}
\label{eq: Nik formula 2i}
 A &:= (\partial_w^2 t_{\pm}) (\partial_{\xi \xi}^2 \phi_{\pm}) + 2(\partial_w t_{\pm})(\partial_w \partial_{\xi\xi}^2 \phi_{\pm}), \\
\label{eq: Nik formula 2ii}
 B &:= (\partial_w^3 t_{\pm}) (\partial_{\xi \xi}^2 \phi_{\pm}) + 3(\partial_w^2 t_{\pm})(\partial_w \partial_{\xi\xi}^2 \phi_{\pm}) + 3(\partial_w t_{\pm})(\partial_w^2 \partial_{\xi\xi}^2 \phi_{\pm}), \\
 \label{eq: Nik formula 2iii}
 C &:= 3(\partial_w^2 t_{\pm})(\partial_w t_{\pm})(\partial_{\xi\xi}^2 \phi_{\pm}) +  3(\partial_w t_{\pm})^2(\partial_w\partial_{\xi\xi}^2 \phi_{\pm}).
\end{align}
\end{subequations}

By differentiating the right-hand formula in \eqref{eq: t derivs 1}, we have
\begin{equation*}
    \partial_w^2 t_{\pm} = \frac{3^2}{2} \Big( \frac{t_{\pm}^3}{\Delta} \pm \frac{\xi t_{\pm}^2}{2 \Delta^{3/2}} \Big), \qquad \partial_w^3 t_{\pm} = \mp \frac{3^4}{2^2} \Big( \frac{t_{\pm}^4}{\Delta^{3/2}} \pm \frac{\xi t_{\pm}^3}{\Delta^2} + \frac{\xi^2 t_{\pm}^2}{2 \Delta^{5/2}}\Big).
\end{equation*}
Whilst differentiating the top-left entry of the Hessian in \eqref{eq: xi derivs} gives
\begin{equation*}
    \partial_w \partial_{\xi\xi}^2 \phi_{\pm} = - \frac{3\xi}{2\Delta} \partial_{\xi\xi}^2 \phi_{\pm}, \quad \partial_w^2 \partial_{\xi\xi}^2 \phi_{\pm} = \frac{3^3 \xi^2}{2^2 \Delta^2} \partial_{\xi\xi}^2 \phi_{\pm}, \quad \partial_w^3 \partial_{\xi\xi}^2 \phi_{\pm} = -\frac{3^4 5\xi^3}{2^3 \Delta^3} \partial_{\xi\xi}^2 \phi_{\pm}.
\end{equation*}
Substituting these identities into \eqref{eq: Nik formula 2i}, \eqref{eq: Nik formula 2ii} and \eqref{eq: Nik formula 2iii}, we obtain
\begin{subequations}
\renewcommand{\theequation}{\theparentequation-\roman{equation}}
\begin{align}
\label{eq: Nik formula 3i}
 A &= \frac{3^2 t_{\pm}^2}{2 \Delta}\Big(t_{\pm} \pm \frac{3}{2} \cdot \frac{\xi}{\Delta^{1/2}}\Big)(\partial_{\xi\xi}^2 \phi_{\pm}), \\
 \label{eq: Nik formula 3ii}
 B &= \mp \frac{3^4 t_{\pm}^2}{2 \Delta^{3/2}}\Big(\frac{1}{2} \cdot t_{\pm}^2 \pm \frac{\xi t_{\pm}}{\Delta^{1/2}} + \frac{5}{2^2} \cdot \frac{\xi^2}{\Delta} \Big)(\partial_{\xi\xi}^2 \phi_{\pm}), \\
 \label{eq: Nik formula 3iii}
 C &= \mp \frac{3^4 t_{\pm}^4}{2^2 \Delta^{3/2}}\Big(t_{\pm} \pm \frac{\xi}{\Delta^{1/2}}\Big)(\partial_{\xi\xi}^2 \phi_{\pm}).
\end{align}
\end{subequations}
Substituting \eqref{eq: Nik formula 3i}, \eqref{eq: Nik formula 3ii} and \eqref{eq: Nik formula 3iii} into \eqref{eq: Nik formula 1} and factoring appropriately, 
\begin{equation*}
  \det N_{\pm} =  \pm \frac{3^4\xi t_{\pm}^6}{\Delta^{5/2}}  \Big(\mp \frac{1}{2\Delta^{1/2}}\Big)^3 \det
    \begin{bmatrix}
        1 & 1 & 0 \\
        -\frac{3^2\xi}{2\Delta} & \mp \frac{3t_{\pm}}{\Delta^{1/2}} - \frac{3^2\xi}{2\Delta} & 1 \\
        \frac{3^3 5 \xi^2}{2^2 \Delta} & \frac{3^3 t_{\pm}^2}{2 \Delta} \pm \frac{3^3 \xi t_{\pm}}{\Delta^{3/2}} + \frac{3^3 5 \xi^2}{2^2 \Delta}  & \mp \frac{3^2t_{\pm}}{\Delta^{1/2}} - \frac{3^2 \xi}{\Delta}
    \end{bmatrix}.
\end{equation*}
The above matrix determinant can be easily evaluated by subtracting the first column from the second column to reduce to a $2 \times 2$ determinant. Substituting the formula for $\partial_{\xi\xi}^2 \phi_{\pm}$ from \eqref{eq: xi derivs}, we obtain the desired formula \eqref{eq: Nik formula}.
%%%%%%%%%%%%%%%%%%%%%%%%%%%%%%%%%%%%%%%%%%%%%%%%%%%%%%%%%%%%%%%%%%%%%%%%%%%%%%%%%%%%%%%%%%%%%%%%

%    Plane curve curvature computation

%%%%%%%%%%%%%%%%%%%%%%%%%%%%%%%%%%%%%%%%%%%%%%%%%%%%%%%%%%%%%%%%%%%%%%%%%%%%%%%%%%%%%%%%%%%%%%%%

\subsection{Curvature condition}\label{subsec: curv comp} Here we verify \eqref{eq: curv formula}. Note that
\begin{equation*}
    \Delta(w; 0, \eta) = \eta^2 \quad \textrm{and} \quad t_{\pm}(w; 0, \eta) = \frac{\eta \pm |\eta|}{3w}. 
\end{equation*}
In particular, there is always at most one non-zero root which is given by $t_{\sgn\eta}(w; 0, \eta) := 2\eta/(3w)$. Arguing as in the previous sections
\begin{equation*}
    \partial_{\eta \eta}^2 \partial_w \phi_{\pm} = \frac{6 t_{\pm}^2}{\Delta}\Big( t_{\pm} \pm \frac{\xi}{2\Delta^{1/2}}\Big)
\end{equation*}
and substituting the above values gives
\begin{equation*}
   \partial_{\eta\eta}^2\partial_w \phi_{\sgn \eta}(w; 0, \eta) =  6 \cdot \frac{t_{\sgn \eta}(w; 0, \eta)^3}{\eta^2} = \frac{2^4}{3^2} \cdot \frac{\eta}{w^3},
\end{equation*}
as required. 
%%%%%%%%%%%%%%%%%%%%%%%%%%%%%%%%%%%%%%%%%%%%%%%%%%%%%%%%%%%%%%%%%%%%%%%%%%%%%%%%%%%%%%%%%%%%%%%%

%                                          REFERENCES

%%%%%%%%%%%%%%%%%%%%%%%%%%%%%%%%%%%%%%%%%%%%%%%%%%%%%%%%%%%%%%%%%%%%%%%%%%%%%%%%%%%%%%%%%%%%%%%%

\bibliography{reference}
\bibliographystyle{amsplain}

\end{document}